\documentclass[11pt,letterpaper]{amsart}

\usepackage{amsmath,amssymb,amsthm,amscd}
\usepackage{mathtools}
\usepackage{mathrsfs}

\usepackage{graphicx}
\usepackage{xcolor}
\usepackage{tikz}

\usepackage[ocgcolorlinks,linkcolor=blue]{hyperref}
\usepackage{url}
\usepackage[shortlabels]{enumitem}
\usepackage{comment}
\usepackage{verbatim}

\usepackage{bm}
\usepackage{bbm}
\usepackage{dsfont}
\usepackage{relsize}
\usepackage{esint}
\usepackage[normalem]{ulem}
\usepackage[utf8]{inputenc}

\numberwithin{equation}{section}
\numberwithin{figure}{section}

\allowdisplaybreaks
\newcommand{\para}[1]{\vspace{3mm}\noindent\textbf{#1.}}

\mathtoolsset{showonlyrefs}
\graphicspath{{images/}}

\newtheorem{theorem}{Theorem}[section]
\newtheorem{lemma}[theorem]{Lemma}
\newtheorem*{lemma*}{Lemma}
\newtheorem{definition}[theorem]{Definition}
\newtheorem{corollary}[theorem]{Corollary}

\newtheorem{proposition}[theorem]{Proposition}

\newtheorem{remark}[theorem]{Remark}

\title[Nonlocal parabolic inverse problems]{A nonlocal-to-local transfer principle for parabolic Calder\'on problems}

\author[C.-L. Lin]{Ching-Lung Lin}
\address{Department of Mathematics, National Cheng Kung University, Tainan 701, Taiwan.}
\email{cllin2@mail.ncku.edu.tw}

\author[Y.-H. Lin]{Yi-Hsuan Lin}
\address{Department of Applied Mathematics, National Yang Ming Chiao Tung University, Hsinchu 30050, Taiwan}
\email{yihsuanlin3@gmail.com}

\author[G. Uhlmann]{Gunther Uhlmann}
\address{Department of Mathematics, University of Washington
	and Institute for Advanced Study, the Hong Kong University of Science
	and Technology}
\email{gunther@math.washington.edu}


\newcommand{\R}{{\mathbb R}}

\newcommand{\N}{{\mathbb N}}

\newcommand {\p} {\partial}

\newcommand{\LC}{\left(}
\newcommand{\RC}{\right)}
\newcommand{\wt}{\widetilde}

\newcommand{\norm}[1]{\lVert #1 \rVert}
\newcommand{\abs}[1]{\left\lvert #1 \right\rvert}
\DeclareMathOperator{\supp}{supp} 
\DeclareMathOperator{\loc}{loc} 

\begin{document}

	\begin{abstract}
		We prove a nonlocal-to-local transfer principle for fractional powers of variable coefficient parabolic operators. Let $H_\sigma=\partial_t-\nabla\cdot\sigma\nabla$, where $\sigma$ is a uniformly elliptic $C^2$ matrix-valued coefficient. For any $s\in(0,1)$ and any nonempty bounded Lipschitz exterior observation set $W$ with $\overline W\cap\overline\Omega=\emptyset$, we show that equality of the partial exterior DN maps for $H_{\sigma_j}^s$, $j=1,2$, together with the exterior agreement $\sigma_1=\sigma_2$ in $\Omega_e$, implies equality of the full lateral local Cauchy data sets for $H_{\sigma_j}$ on $(\partial\Omega)_T$. The proof uses the parabolic Caffarelli--Silvestre extension and the averaging transform
		\[
		v(t,x)=\int_0^\infty y^{1-2s}\widetilde u(t,x,y)\,dy.
		\]
		This transform sends fractional parabolic solutions to local parabolic energy solutions. The main analytic ingredient is a boundary density theorem: the lateral traces generated by these averaged exterior solutions are dense in the corresponding closed local energy trace space. Unlike the elliptic reduction, no exterior normalization to a constant is imposed; the common exterior coefficient may be variable. The result holds for all spatial dimensions $n\geq2$. Consequently, local uniqueness results for parabolic Calder\'on problems transfer to the corresponding nonlocal problems. In particular, scalar leading coefficients are uniquely determined in $\Omega$ from partial exterior nonlocal measurements, relative to their exterior values.
	\end{abstract}
	
	\maketitle
	\tableofcontents
	
	\section{Introduction}\label{sec: introduction}
	
	This paper studies the relationship between local and nonlocal parabolic Calder\'on problems. In the local problem, the measurements are taken on the lateral boundary of a space-time cylinder. In the nonlocal problem, the measurements are exterior measurements for a fractional power of a parabolic operator. These two data sets are a priori different. Our main result proves that, for fractional powers of variable coefficient parabolic operators, equality of the partial exterior nonlocal data on any nonempty exterior observation set, for two coefficients agreeing in the exterior, implies equality of the full lateral local Cauchy data for the underlying parabolic equations.
	
	More concretely, let $H_{\sigma_j}=\partial_t-\nabla\cdot\sigma_j\nabla$, $j=1,2$, where the coefficients are uniformly elliptic matrix-valued functions and satisfy $\sigma_1=\sigma_2$ in $\Omega_e$. We prove that equality of the partial exterior Dirichlet-to-Neumann (DN) maps for $H_{\sigma_j}^s$, with $s\in(0,1)$, implies equality of the full lateral local Cauchy data sets for $H_{\sigma_j}$. In boundary trace settings, where the local DN maps are well defined, this is equivalent to equality of the full local parabolic DN maps. The mechanism is given by the parabolic Caffarelli--Silvestre extension together with a vertical averaging transform. If $\widetilde u$ denotes the extension of a solution of the nonlocal equation, we introduce
	\begin{equation}\label{intro averaged function}
		v(t,x)=\int_0^\infty y^{1-2s}\widetilde u(t,x,y)\,dy.
	\end{equation}
	We prove that this averaged function is a local parabolic energy solution and that the boundary traces obtained in this way are dense in the corresponding closed local energy trace space. In short, the partial exterior data for the nonlocal problem contain the full boundary data for the local problem.
	
	Fractional-type inverse problems have attracted a lot of attention in recent years. The Calder\'on problem for the fractional Schr\"odinger equation was first investigated in \cite{GSU20}, where the authors demonstrated that the potential in a given region can be determined uniquely by the associated exterior measurement. The essential approach relies on the \emph{strong unique continuation property} for the fractional Laplacian $(-\Delta)^s$, so that one can deduce the useful \emph{Runge approximation} for the fractional Schr\"odinger equation. Based on these robust results, a huge body of literature has developed in this direction.
	
	Let us briefly summarize several works related to fractional inverse problems. In \cite{cekic2020calderon}, the authors determined both the drift term and the potential uniquely, which remains open for the local case. In \cite{CLL2017simultaneously}, the researchers used a single measurement to determine an unknown cavity, a phenomenon with no direct local analogue. Meanwhile, in the works \cite{harrach2017nonlocal-monotonicity,harrach2020monotonicity}, the authors derived an if-and-only-if monotonicity relation, which leads to a simple reconstruction algorithm. Since then, fractional and nonlocal inverse problems have developed rapidly in the field of inverse problems, which consists of the determination of singular potentials, lower order local perturbations, higher order fractional Laplacians, single measurement, and generalizations to many other nonlocal operators. We refer readers to those works \cite{bhattacharyya2021inverse,CMR20,CMRU20,GLX,CLL2017simultaneously,cekic2020calderon,GRSU18,GU2021calder,lin2020monotonicity,LL2020inverse,LL2022inverse,LLR2019calder,LLU2022para,KLW2021calder,RS17,ruland2018exponential,RZ2022unboundedFracCald,RZ2022FracCondCounter,CRZ2022global,RZ-low-2022,CRTZ-2022,zimmermann2023inverse,LRZ2022calder,LZ2023unique} and the references therein. We also point out that the recovery of leading coefficients has been addressed in recent works \cite{feizmohammadi2021fractional_closed,feizmohammadi2021fractional,choulli2023fractional,ruland2023revisiting} by using the local source-to-solution map. The main approach is based on the heat kernel representation of nonlocal operators and Kannai transmutation. These materials transfer the elliptic type nonlocal inverse problem to a local hyperbolic problem, which has been studied by utilizing the boundary control method. In short, let us mention the recent monograph about inverse problems for integro-differential operators \cite{LL25_IPBook} for a comprehensive study.
	
	Many of these fractional inverse problems rely on structural properties of the underlying nonlocal operators, in particular, strong unique continuation and Runge approximation. In most of the works mentioned above, the main nonlocal operator is known a priori, and the inverse problem concerns the recovery of lower-order coefficients. Determining the leading coefficient is generally more difficult. The purpose of this work is to provide a transfer mechanism from nonlocal to local parabolic inverse problems, which can be used to recover leading coefficients for nonlocal parabolic operators through their local counterparts. 
	
	Our result should be viewed as a transfer principle rather than as a direct reconstruction theorem. Instead of recovering the coefficient from the nonlocal data in one step, we prove a relative comparison statement: equality of the nonlocal exterior data for two coefficients that agree in $\Omega_e$ implies equality of the full local parabolic Cauchy data. Thus, any uniqueness theorem for the local parabolic Calder\'on problem immediately gives a corresponding uniqueness theorem for the nonlocal parabolic problem, relative to the prescribed exterior coefficient.
	
	\subsection{Mathematical model and main results}
	
	Let $\Omega\subset\R^n$ be a bounded domain with Lipschitz boundary for $n\geq 2$, and let $T\in(0,\infty)$. For any subset $A\subseteq\R^n$, we write
	\begin{equation}
		A_T:=(-T,T)\times A.
	\end{equation}
	Let $\Omega_e:=\R^n\setminus\overline{\Omega}$. We consider the fractional Calder\'on problem for the nonlocal parabolic equation
	\begin{equation}\label{equ nonlocal para}
		\begin{cases}
			H_\sigma^s u(t,x)=0 &\text{ in }\Omega_T,\\
			u(t,x)=f(t,x) &\text{ in }(\Omega_e)_T,\\
			u(t,x)=0 &\text{ for }x\in\R^n\text{ and }t\leq -T.
		\end{cases}
	\end{equation}
	Here $H_\sigma=\partial_t-\nabla\cdot\sigma\nabla$, and $\sigma=\left(\sigma_{ik}(x)\right)_{1\leq i,k\leq n}\in C^2(\R^n;\R^{n\times n})$ satisfies the symmetry and ellipticity conditions
	\begin{equation}\label{ellipticity condition}
		\sigma_{ik}=\sigma_{ki},\quad \lambda|\xi|^2\leq\sum_{i,k=1}^n\sigma_{ik}(x)\xi_i\xi_k\leq\lambda^{-1}|\xi|^2
	\end{equation}
	for all $x\in\R^n$ and $\xi\in\R^n$, where $\lambda\in(0,1)$ is fixed. The well-posedness of \eqref{equ nonlocal para} has been studied by \cite{BS2024calderon} with respect to suitable function spaces. These function spaces will be introduced in Section \ref{sec:preliminaries}.
	
	Let $W\subset\Omega_e$ be a nonempty open subset. We define the corresponding partial exterior DN map for the nonlocal parabolic problem by
	\begin{equation}\label{nonlocal DN map}
		\Lambda_{\sigma,W_T}^s:\widetilde{\mathbf H}^s(W_T)\to\mathbf H^{-s}(W_T),\quad f\mapsto \left.H_\sigma^s u_f\right|_{W_T},
	\end{equation}
	where $u_f\in\mathbf H^s(\R^{n+1})$ is the unique solution to \eqref{equ nonlocal para}.
	
	\begin{enumerate}[(IP1)]
		\item \label{IP1} \textbf{Nonlocal inverse problem.} Can we determine $\sigma$ by using the nonlocal DN map $\Lambda_{\sigma,W_T}^s$ given by \eqref{nonlocal DN map}?
	\end{enumerate}
	
	In fact, it is a nontrivial problem to determine the leading coefficient $\sigma$ for nonlocal models.
	
	On the other hand, let $\Omega\subset\R^n$ be a bounded set as before, and we consider the classical Calder\'on problem for local parabolic equations. One tries to recover an unknown, possibly anisotropic leading coefficient $\sigma=\sigma(x):\overline{\Omega}\to\R^{n\times n}$, where $\sigma$ could be a sufficiently regular anisotropic symmetric matrix-valued function. More concretely, consider the local parabolic problem
	\begin{equation}\label{equ local para}
		\begin{cases}
			H_\sigma v(t,x)=0 &\text{ in }\Omega_T,\\
			v(t,x)=g(t,x) &\text{ on }(\partial\Omega)_T,\\
			v(-T,x)=0 &\text{ for }x\in\Omega.
		\end{cases}
	\end{equation}
	It is known that the initial-boundary value problem \eqref{equ local para} is well-posed, for example, see \cite[Chapter XVIII]{DL1992}. For this article, it is convenient to formulate the local data in terms of the energy Cauchy data set. If
	\[
	V\in L^2(-T,T;H^1(\Omega)),\quad \p_tV\in L^2(-T,T;H^{-1}(\Omega)),
	\]
	solves $H_\sigma V=0$ in $\Omega_T$ and $V(-T,\cdot)=0$, then its lateral trace belongs to $L^2(-T,T;H^{1/2}(\partial\Omega))$, and its conormal derivative is defined weakly by
	\begin{equation}\label{weak conormal derivative local}
		\left\langle \sigma\nabla V\cdot\nu,\phi\right\rangle=\int_{\Omega_T}\LC \sigma\nabla V\cdot\nabla\Phi - V\p_t\Phi\RC\,dtdx,
	\end{equation}
	where $\Phi\in L^2(-T,T;H^1(\Omega))$ is any extension of $\phi$ with compact support in time. With this convention, if the local Dirichlet problem is well posed for boundary data $g$ in a chosen trace space, we write
	\begin{equation}\label{local DN map}
		\Lambda_\sigma g:=\left.\sigma\nabla V_g\cdot\nu\right|_{(\partial\Omega)_T}
		\in L^2(-T,T;H^{-1/2}(\partial\Omega)),
	\end{equation}
	where $V_g$ is the corresponding energy solution of \eqref{equ local para}. Thus, whenever the local Dirichlet problem is considered in the trace space $L^2(-T,T;H^{1/2}(\partial\Omega))$, the map $\Lambda_\sigma$ is understood as a bounded operator
	\[
	\Lambda_\sigma:
	L^2(-T,T;H^{1/2}(\partial\Omega))
	\to
	L^2(-T,T;H^{-1/2}(\partial\Omega)).
	\]
	Let
	\begin{equation}\label{local solution space D}
		\mathcal D_\sigma:=
		\Big\{
		V\in L^2(-T,T;H^1(\Omega)):\,
		\partial_t V\in L^2(-T,T;H^{-1}(\Omega)),\,
		H_\sigma V=0 \text{ in }\Omega_T,\,
		V(-T,\cdot)=0
		\Big\}.
	\end{equation}
	We denote the corresponding energy trace space by
	\begin{equation}\label{local energy trace space}
		\mathcal T_\sigma:=
		\overline{
			\left\{V|_{(\partial\Omega)_T}:\,V\in\mathcal D_\sigma\right\}
		}^{\,L^2(-T,T;H^{1/2}(\partial\Omega))}.
	\end{equation}
	The closed local full Cauchy data set is defined by
	\begin{equation}\label{local Cauchy data energy}
		\mathcal C_{\sigma,(\partial\Omega)_T}:=
		\overline{
			\big\{
			\big(V|_{(\partial\Omega)_T},
			\left.\sigma\nabla V\cdot\nu\right|_{(\partial\Omega)_T}\big):
			V\in\mathcal D_\sigma
			\big\}
		}^{\,L^2(-T,T;H^{1/2}(\partial\Omega))\times
			L^2(-T,T;H^{-1/2}(\partial\Omega))}.
	\end{equation}
	When the local Dirichlet problem is well posed in a boundary trace space
	$\mathcal X\subset L^2(-T,T;H^{1/2}(\partial\Omega))$, the above Cauchy data set is the closed graph of the corresponding continuous DN map
	\[
	\Lambda_\sigma:\mathcal X\to \mathcal X^\ast.
	\]
	Particularly, when the local Dirichlet problem is well posed for all data in
	$L^2(-T,T;H^{1/2}(\partial\Omega))$, one has
	\[
	\mathcal T_\sigma=L^2(-T,T;H^{1/2}(\partial\Omega)).
	\]
	
	\begin{enumerate}[(IP2)]
		\item \label{IP2} \textbf{Local inverse problem.} Can we determine $\sigma$ by using the local DN map $\Lambda_\sigma$ given by \eqref{local DN map}?
	\end{enumerate}
	
	The local inverse problem \ref{IP2} is known to be uniquely solvable in several settings. For example, Canuto and Kavian \cite{canuto2001determining} proved uniqueness for scalar coefficients $\sigma=\gamma I_n$. In this article, we use ideas related to our earlier work \cite{LLU2022para}, but formulate the result as a relative transfer theorem: equality of the partial nonlocal DN maps for two coefficients agreeing in the exterior implies equality of the corresponding full local Cauchy data sets.
	
	For $s\in(0,1)$, the extension formula for $H_\sigma^s$ is characterized as follows. Let $u\in\mathbf H^s(\R^{n+1})$, and let $\widetilde u=\widetilde u(t,x,y)$ be the solution of the Dirichlet problem for $(t,x,y)\in\R^{n+2}_+:=\R^{n+1}\times(0,\infty)$:
	\begin{equation}\label{degen para PDE}
		\begin{cases}
			y^{1-2s}\partial_t\widetilde u-\nabla_{x,y}\cdot\left(y^{1-2s}\widetilde\sigma(x)\nabla_{x,y}\widetilde u\right)=0 &\text{ in }\R^{n+2}_+,\\
			\widetilde u(t,x,0)=u(t,x) &\text{ on }\R^{n+1},\\
			\widetilde u(t,x,y)=0 &\text{ for }t\leq -T,
		\end{cases}
	\end{equation}
	where
	\begin{equation}\label{tilde sigma}
		\widetilde\sigma(x)=\begin{pmatrix}
			\sigma(x)&0\\
			0&1
		\end{pmatrix}.
	\end{equation}
	Meanwhile, we use the notation $\nabla\equiv\nabla_x$, $\nabla_{x,y}\equiv(\nabla_x,\partial_y)$ and $\mathrm{Id}$ stands for the identity in this article. Then the nonlocal parabolic operator $H_\sigma^s$ is realized in terms of the DN relation
	\begin{equation}\label{extension DN relation}
		H_\sigma^s u(t,x)=d_s\lim_{y\to0}y^{1-2s}\partial_y\widetilde u(t,x,y),
	\end{equation}
	where $d_s$ is a constant depending only on $s\in(0,1)$. Our strategy is to reduce the comparison problem for \ref{IP1} to the corresponding local comparison problem \ref{IP2}. More specifically, we will show the following theorem, which states that equality of the nonlocal DN maps for two coefficients agreeing in the exterior implies equality of the corresponding full local lateral Cauchy data, or equivalently of the local DN maps whenever these maps are well defined.
	
	\begin{theorem}\label{Thm: main}
		Let $\Omega\subset\R^n$ be a bounded Lipschitz domain such that $\Omega_e:=\R^n\setminus\overline\Omega$ is connected, and let $W\subset\Omega_e$ be a nonempty bounded Lipschitz open set with $\overline W\cap\overline\Omega=\emptyset$. Let $n\geq2$, $T\in(0,\infty)$ and $s\in(0,1)$.
		
		Let $\sigma\in C^2(\R^n;\R^{n\times n})$ be symmetric and satisfy \eqref{ellipticity condition}, with $|\sigma|+|\nabla\sigma|\leq M<\infty$ in $\R^n$. Let $\widetilde\sigma$ be of the form \eqref{tilde sigma}. For $f\in\widetilde{\mathbf H}^s(W_T)$, extended by zero to $(\Omega_e)_T$, let $u_f$ be the solution of \eqref{equ nonlocal para}, and let $\widetilde u_f$ be its extension solving \eqref{degen para PDE}. Define
		\begin{equation}\label{the function v}
			v_f(t,x):=\int_0^\infty y^{1-2s}\widetilde u_f(t,x,y)\,dy.
		\end{equation}
		Then $v_f\in L^2(-T,T;H^1(\Omega))$, $\partial_t v_f\in L^2(-T,T;H^{-1}(\Omega))$, and $v_f$ is a weak solution of $H_\sigma v_f=0$ in $\Omega_T$.
		
		Moreover, let $\sigma_1,\sigma_2\in C^2(\R^n;\R^{n\times n})$ be two coefficients satisfying the same assumptions and suppose that
		\begin{equation}\label{exterior agreement assumption}
			\sigma_1=\sigma_2\quad\text{in }\Omega_e.
		\end{equation}
		If
		\begin{equation}\label{equal nonlocal DN maps}
			\Lambda_{\sigma_1,W_T}^s f=\Lambda_{\sigma_2,W_T}^s f
			\quad \text{for all }f\in \widetilde{\mathbf H}^s(W_T),
		\end{equation}
		then
		\[
		\mathcal C_{\sigma_1,(\partial\Omega)_T}
		=
		\mathcal C_{\sigma_2,(\partial\Omega)_T}.
		\]
		Equivalently, in any boundary trace setting in which the local DN maps are well defined and continuous, one has
		\[
		\Lambda_{\sigma_1}=\Lambda_{\sigma_2}
		\quad\text{on }(\partial\Omega)_T.
		\]
	\end{theorem}
	
	\begin{remark}
		The exterior assumption \eqref{exterior agreement assumption} is not an exterior normalization. In particular, the common exterior coefficient need not be $\mathrm{Id}$ and may be a variable matrix-valued coefficient. This is the point at which the present parabolic comparison differs from the elliptic Caffarelli--Silvestre reduction in \cite{CGRU2023reduction}, where the exterior coefficient is normalized to be the identity. The role of \eqref{exterior agreement assumption} is only to ensure that, in the exterior domain, the two extension equations have the same principal part, so that the equality of the partial nonlocal Cauchy data can be propagated from $W_T$ to $(\partial\Omega)_T$ by unique continuation.
	\end{remark}
	
	We can reformulate Theorem \ref{Thm: main} in terms of the next result.
	
	\begin{proposition}\label{Proposition main theorem}
		Under the fixed-coefficient assumptions of Theorem \ref{Thm: main}, define the nonlocal partial Cauchy data set
		\begin{equation}
			\mathcal C_{\sigma,W_T}^s:=\left\{\left(f,\Lambda_{\sigma,W_T}^s f\right): f\in C_c^\infty(W_T)\right\}\subset\widetilde{\mathbf H}^s(W_T)\times\mathbf H^{-s}(W_T),
		\end{equation}
		where $u_f\in\mathbf H^s(\R^{n+1})$ is the solution to \eqref{equ nonlocal para}. Define the local full Cauchy data set by the closed energy Cauchy data set
		\begin{equation}
			\mathcal C_{\sigma,(\partial\Omega)_T}:=
			\overline{
				\left\{
				\big(
				V|_{(\partial\Omega)_T},
				\sigma\nabla V\cdot\nu|_{(\partial\Omega)_T}
				\big):
				V\in\mathcal D_\sigma
				\right\}
			}^{\,L^2(-T,T;H^{1/2}(\partial\Omega))\times
				L^2(-T,T;H^{-1/2}(\partial\Omega))}.
		\end{equation}
		When the local Dirichlet problem is well posed in the chosen trace space, this set is the graph of $\Lambda_\sigma$.
		Then, for each fixed coefficient $\sigma$, the assignment
		\begin{equation}
			\mathrm T_\sigma\left(f,\Lambda_{\sigma,W_T}^s f\right) :=
			\big(v_f\big|_{(\partial\Omega)_T},
			\left.\sigma\nabla v_f\cdot\nu\right|_{(\partial\Omega)_T}\big)
		\end{equation}
		is linear on $\mathcal C_{\sigma,W_T}^s$. Equipping $\mathcal C_{\sigma,W_T}^s$ with the norm induced by
		\[
		\widetilde{\mathbf H}^s(W_T)\times\mathbf H^{-s}(W_T),
		\]
		this assignment is bounded, and
		\begin{equation}
			\overline{\mathrm T_\sigma\big(\mathcal C_{\sigma,W_T}^s\big)}^{L^2(-T,T;H^{1/2}(\partial\Omega))\times
				L^2(-T,T;H^{-1/2}(\partial\Omega))}=\mathcal C_{\sigma,(\partial\Omega)_T}.
		\end{equation}
	\end{proposition}

	Furthermore, if $\sigma_1,\sigma_2$ satisfy \eqref{exterior agreement assumption} and $\mathcal C_{\sigma_1,W_T}^s=\mathcal C_{\sigma_2,W_T}^s$, then $\mathcal C_{\sigma_1,(\partial\Omega)_T}=\mathcal C_{\sigma_2,(\partial\Omega)_T}$. Thus, Theorem \ref{Thm: main} reduces the nonlocal parabolic Calder\'on comparison problem to the corresponding local parabolic comparison problem. The reduction is obtained through the parabolic Caffarelli--Silvestre extension and the averaging transform \eqref{the function v}. The main point is that exterior nonlocal data generate, after averaging, a dense family in the closed local energy trace space.
	
	\begin{remark}
		Theorem \ref{Thm: main} shows that, in the relative comparison setting considered here, the recovery of leading coefficients for nonlocal parabolic operators can be reduced to the corresponding local parabolic inverse problem. The partial exterior nonlocal data contain, through the averaging construction, enough information to recover the full local Cauchy data.
	\end{remark}
	
	\begin{corollary}\label{Corollary: uniqueness}
		Let $\sigma_j=\gamma_j I_n$, $j=1,2$, satisfy the assumptions of Theorem \ref{Thm: main}. Assume that $\gamma_1=\gamma_2$ in $\Omega_e$. Assume, in addition, that the scalar coefficients belong to a class for which the full local parabolic Calder\'on problem is uniquely solvable, for instance the class considered in \cite{canuto2001determining}. If $\Lambda_{\sigma_1,W_T}^s=\Lambda_{\sigma_2,W_T}^s$, then $\gamma_1=\gamma_2$ in $\Omega$. In particular, the partial exterior nonlocal DN map determines the scalar leading coefficient in $\Omega$ relative to its exterior value.
	\end{corollary}
	
	We next discuss the case when the leading coefficient is a matrix-valued function. It is known that a non-uniqueness mechanism for anisotropic local parabolic equations is given by boundary-fixing changes of variables. Let $\sigma(x)=\left(\sigma_{ij}(x)\right)_{1\leq i,j\leq n}\in C^2(\R^n;\R^{n\times n})$ be a matrix-valued function satisfying \eqref{ellipticity condition}. Consider $\Phi:\overline{\Omega}\to\overline{\Omega}$ as a $C^\infty$ diffeomorphism such that $\Phi|_{\partial\Omega}=\mathrm{Id}$. Then $v(t,x)$ is a solution to the parabolic equation
	\begin{equation}\label{equ nonunique 1}
		\partial_t v-\nabla\cdot(\sigma\nabla v)=0\quad\text{for }(t,x)\in\Omega_T
	\end{equation}
	if and only if $\widetilde v(t,z):=v(t,\Phi^{-1}(z))$ is a solution to
	\begin{equation}\label{equ nonunique 2}
		\partial_t\left(\Phi_\ast1(z)\widetilde v\right)-\nabla\cdot\left(\Phi_\ast\sigma\nabla\widetilde v\right)=0\quad\text{for }(t,z)\in\Omega_T,
	\end{equation}
	where $\Phi_\ast$ denotes the \emph{push-forward}
	\begin{equation}\label{push forward local parabolic}
		\Phi_\ast1(z)=\left.\frac{1}{|\det D\Phi(x)|}\right|_{x=\Phi^{-1}(z)},\quad
		\Phi_\ast\sigma(z)=\left.\frac{D\Phi(x)\sigma(x)D\Phi(x)^T}{|\det D\Phi(x)|}	\right|_{x=\Phi^{-1}(z)}.
	\end{equation}
	Here, $D\Phi$ denotes the matrix differential of $\Phi$ and $D\Phi^T$ is the transpose of $D\Phi$. Since $\Phi|_{\partial\Omega}=\mathrm{Id}$, it follows that the full Cauchy data of \eqref{equ nonunique 1} and \eqref{equ nonunique 2} are the same, i.e.,
	\begin{equation}
		\mathcal C_{\sigma,(\partial\Omega)_T}=\mathcal C_{(\Phi_\ast1,\Phi_\ast\sigma),(\partial\Omega)_T},
	\end{equation}
	where the right-hand side denotes the Cauchy data for the weighted parabolic operator $\Phi_\ast1\,\partial_t-\nabla\cdot\Phi_\ast\sigma\nabla$.
	
	\begin{remark}[Anisotropic gauge and weighted parabolic operators]\label{Remark: anisotropic gauge}
		The preceding calculation shows that the natural anisotropic gauge for parabolic equations does not preserve the unweighted class $\partial_t-\nabla\cdot\sigma\nabla$. Instead, it produces the weighted parabolic class
		\begin{equation}
			\kappa(x)\partial_t-\nabla\cdot\sigma(x)\nabla,
		\end{equation}
		with $\kappa=\Phi_\ast1$. Therefore, a non-uniqueness statement for fully anisotropic leading coefficients should be formulated in this weighted class, together with the corresponding fractional powers of $\kappa\partial_t-\nabla\cdot\sigma\nabla$. Since the main purpose of the present article is the local--nonlocal transfer principle for the unweighted operator $H_\sigma=\partial_t-\nabla\cdot\sigma\nabla$, we state the scalar uniqueness consequence in Corollary \ref{Corollary: uniqueness} and record the anisotropic gauge separately.
	\end{remark}

	\subsection{Relation to previous works}
	
	The present article is related to two lines of work. First, the nonlocal parabolic Calder\'on problem was studied in our earlier work \cite{LLU2022para}, where a reduction from the nonlocal parabolic problem to the local parabolic problem was obtained by a different strategy. The point of the present paper is not merely to give another proof of that reduction. Rather, we isolate a relative transfer principle at the level of closed energy Cauchy data sets: if two coefficients agree in the exterior and their partial exterior nonlocal data agree on $W_T$, then their full local lateral Cauchy data sets agree. This formulation separates the local--nonlocal transfer mechanism from any subsequent local uniqueness theorem and avoids imposing a global exterior normalization.
	
	Second, our work is inspired by the elliptic Caffarelli--Silvestre reduction of Covi, Ghosh, R\"uland and Uhlmann \cite{CGRU2023reduction}. In the elliptic setting, their result shows that partial data for the fractional Calder\'on problem determine the full DN map for the local elliptic Calder\'on problem in dimensions $n\geq3$, under an exterior normalization of the leading coefficient. The parabolic problem considered here has a different structure. The time direction imposes causality, the extension problem inherits an implicit past-time condition from the support of the nonlocal solution, and the density argument is based on a backward parabolic adjoint problem. These features allow the argument to work in all spatial dimensions $n\geq2$.
	
	A further point of distinction is the exterior assumption. In the elliptic reduction of \cite{CGRU2023reduction}, the coefficient is normalized to be $1$ in the exterior domain. In the present parabolic comparison theorem, we impose only the exterior agreement condition $\sigma_1=\sigma_2$ in $\Omega_e$. Thus, the common exterior coefficient may be variable. This is sufficient for the exterior unique continuation step, because the two extension equations then have the same coefficients in $\Omega_e\times(0,\infty)$. Finally, the anisotropic gauge issue is recorded separately in Remark \ref{Remark: anisotropic gauge}: boundary-fixing changes of variables naturally lead to weighted parabolic operators, not to the unweighted class $H_\sigma=\partial_t-\nabla\cdot\sigma\nabla$ considered in the main theorem.
	
	\subsection{Outline of the argument}
	
	Let us compute the relation between $\widetilde u_f$ and $v_f$, where $\widetilde u_f$ is a solution to \eqref{degen para PDE} and $v_f$ is defined by \eqref{the function v}. Via direct computations, it follows that $v_f$ is a solution to the following parabolic equation:
	\begin{equation}\label{para v}
		\begin{split}
			0&=\int_0^\infty\left\{y^{1-2s}\partial_t\widetilde u_f-\nabla_{x,y}\cdot\left(y^{1-2s}\widetilde\sigma\nabla_{x,y}\widetilde u_f\right)\right\}\,dy\\
			&=\partial_t\left(\int_0^\infty y^{1-2s}\widetilde u_f\,dy\right)-\nabla\cdot\sigma\nabla\left(\int_0^\infty y^{1-2s}\widetilde u_f\,dy\right)-\int_0^\infty\partial_y\left(y^{1-2s}\partial_y\widetilde u_f(t,x,y)\right)\,dy\\
			&=H_\sigma v_f,
		\end{split}
	\end{equation}
	for $(t,x)\in\Omega_T$. Here, we used the fact that
	\begin{equation}
		\begin{split}
			\int_0^\infty\partial_y\left(y^{1-2s}\partial_y\widetilde u_f(t,x,y)\right)\,dy
			&=\lim_{y\to\infty}y^{1-2s}\partial_y\widetilde u_f(t,x,y)-\lim_{y\to0}y^{1-2s}\partial_y\widetilde u_f(t,x,y)\\
			&=\lim_{y\to\infty}y^{1-2s}\partial_y\widetilde u_f(t,x,y)-\frac{1}{d_s}H_\sigma^s u_f(t,x)=0,
		\end{split}
	\end{equation}
	where we used the DN relation \eqref{extension DN relation}, the equation $H_\sigma^s u_f=0$ in $\Omega_T$, and the decay property of $\widetilde u_f$ and its derivatives at infinity.
	
	Formally, the corresponding DN map of \eqref{para v} is given by
	\begin{equation}
		\begin{split}
			\Lambda_\sigma:\ L^2(-T,T;H^{1/2}(\partial\Omega))&\to L^2(-T,T;H^{-1/2}(\partial\Omega)),\\
			\left.\int_0^\infty y^{1-2s}\widetilde u_f(t,x,y)\,dy\right|_{(\partial\Omega)_T}
			&\mapsto
			\sigma\nabla\bigg(\int_0^\infty y^{1-2s}\widetilde u_f(t,x,y)\,dy\bigg)\cdot\nu\bigg|_{(\partial\Omega)_T}.
		\end{split}
	\end{equation}
	Making the preceding formal computations rigorous plays an essential role in the proof of Theorem \ref{Thm: main}. In fact, we need to prove that the function $v_f(t,x)$ given by \eqref{the function v} belongs to suitable function spaces with $\left.v_f\right|_{(\partial\Omega)_T}\in L^2(-T,T;H^{1/2}(\partial\Omega))$. Therefore, in the comparison setting, equality of the nonlocal exterior data allows us to identify the corresponding averaged local Cauchy data.
	
	The rigorous proof has three main components. First, we prove that the averaging transform $\widetilde u_f\mapsto v_f$ is well-defined in the local parabolic energy class. This requires heat-kernel estimates for the parabolic extension and yields $v_f\in L^2(-T,T;H^1(\Omega))$ and $\partial_t v_f\in L^2(-T,T;H^{-1}(\Omega))$. Second, we show that $v_f$ is indeed a weak solution of the local parabolic equation $H_\sigma v_f=0$ in $\Omega_T$. Third, we prove the density theorem: the traces $\left.v_f\right|_{(\partial\Omega)_T}$ generated by exterior data $f\in C_c^\infty(W_T)$ are dense in the closed local energy trace space $\mathcal T_\sigma$. In settings where the local Dirichlet problem is well posed for all data in $L^2(-T,T;H^{1/2}(\partial\Omega))$, this gives density in the full space $L^2(-T,T;H^{1/2}(\partial\Omega))$. The density argument uses Hahn--Banach, a backward parabolic adjoint problem in the extension space, and unique continuation.
	
	Once these three ingredients are available, the proof of the transfer theorem is natural. Suppose that $\sigma_1=\sigma_2$ in $\Omega_e$ and that the corresponding partial nonlocal DN maps agree on $W_T$. The equality of the exterior nonlocal Cauchy data gives equality of the Dirichlet and weighted Neumann traces of the two extensions on $W_T\times\{0\}$. Since the extension equations have the same coefficients in the exterior domain, boundary unique continuation propagates this equality through $(\Omega_e)_T\times(0,\infty)$ and up to $(\partial\Omega)_T\times(0,\infty)$. Integrating the extension variable then gives equality of the averaged local Cauchy data
	\[
	\big(\left.v_{1,f}\right|_{(\partial\Omega)_T},\left.\sigma_1\nabla v_{1,f}\cdot\nu\right|_{(\partial\Omega)_T}\big)
	=
	\big(\left.v_{2,f}\right|_{(\partial\Omega)_T},\left.\sigma_2\nabla v_{2,f}\cdot\nu\right|_{(\partial\Omega)_T}\big).
	\]
	The density theorem and the continuity of the local Cauchy data map then give equality of the full local data.

	\para{Organization of the article}
	Section \ref{sec:preliminaries} recalls the definition of the fractional parabolic operator, the relevant function spaces, and the well-posedness theory. Section \ref{sec: duality} introduces the parabolic Caffarelli--Silvestre extension and the averaging transform connecting the nonlocal and local equations. Section \ref{sec: regularity} establishes the regularity estimates needed for the averaging transform. Section \ref{sec: key equation} proves the local weak equation satisfied by the averaged solution. Section \ref{sec: density} proves the boundary density theorem. Finally, Section \ref{sec: proof of thm} proves the main results.

	\section{Preliminaries}\label{sec:preliminaries}
	
	We review several basic properties and tools, which will be utilized in our work.
	
	\subsection{Nonlocal parabolic operators}
	
	Let $H_\sigma:=\p_t-\nabla\cdot\sigma\nabla$, where $\sigma=\LC\sigma_{ik}\RC_{1\leq i,k\leq n}$ is a matrix-valued function satisfying \eqref{ellipticity condition} in $\R^n$. The fractional power $H_\sigma^s=\LC\p_t-\nabla\cdot\sigma\nabla\RC^s$, for $s\in(0,1)$, is defined by the evolutive semigroup associated with $H_\sigma$; see \cite{BDLCS2021harnack, BS2024calderon}. It is known that the parabolic operator $\p_t-\nabla\cdot\sigma\nabla$ in $\R\times\R^n$ possesses a globally defined fundamental solution $p(x,z,\tau)$ satisfying
	\begin{equation}
		\int_{\R^n}p(x,z,\tau)\,dz=1,\quad x\in\R^n,\quad \tau>0.
	\end{equation}
	The corresponding evolutive semigroup $\LC\mathcal P_\tau\RC_{\tau\geq0}$ is given by
	\begin{equation}\label{e-semni gp}
		\mathcal P_\tau u(t,x):=\int_{\R^n}p(x,z,\tau)u(t-\tau,z)\,dz,\quad u\in\mathcal S(\R^{n+1}),
	\end{equation}
	where $p(x,z,\tau)$ is the heat kernel associated to the elliptic operator $\nabla\cdot\sigma\nabla$, namely
	\begin{equation}\label{heat kernel}
		\p_\tau p(x,z,\tau)-\nabla\cdot\sigma\nabla p(x,z,\tau)=0.
	\end{equation}
	Here $\mathcal S(\R^{n+1})$ denotes the Schwartz space. The heat kernel satisfies the Gaussian bounds
	\begin{equation}\label{est-heat-kernel-sec2}
		C_1\LC\frac{1}{4\pi\tau}\RC^{n/2}e^{-\frac{c_1|x-z|^2}{4\tau}}\leq p(x,z,\tau)\leq C_2\LC\frac{1}{4\pi\tau}\RC^{n/2}e^{-\frac{c_2|x-z|^2}{4\tau}},
	\end{equation}
	for some positive constants $c_1,c_2,C_1$ and $C_2$. Moreover, the heat kernel possesses the pointwise estimate
	\begin{equation}\label{heat kernel estimate}
		\big|\nabla_x^\ell p(x,z,\tau)\big|\lesssim \tau^{-\frac{n+\ell}{2}}e^{-c\frac{|x-z|^2}{\tau}},\quad \ell=0,1,
	\end{equation}
	where the case $\ell=0$ follows from \cite{ST10}, and the case $\ell=1$ follows from \cite{CJKS20}.
	
	Since $\LC\mathcal P_\tau\RC_{\tau\geq0}$ is a strongly continuous contractive semigroup, the fractional parabolic operator can be defined by the Balakrishnan formula.
	
	\begin{definition}\label{Def: fractional parabolic operator}
		Let $s\in(0,1)$ and $u\in\mathcal S(\R^{n+1})$. Then $H_\sigma^s=\LC\p_t-\nabla\cdot\sigma\nabla\RC^s$ is defined by
		\begin{equation}\label{H^s}
			H_\sigma^s u(t,x):=-\frac{s}{\Gamma(1-s)}\int_0^\infty \LC\mathcal P_\tau u(t,x)-u(t,x)\RC\frac{d\tau}{\tau^{1+s}}.
		\end{equation}
	\end{definition}
	
	Let $\mathcal F_t$ denote the Fourier transform in the time variable. The preceding definition can also be written in terms of spectral measures. Let $\LC E_\lambda\RC_{\lambda\geq0}$ be the spectral resolution associated with the nonnegative elliptic operator $-\nabla\cdot\sigma\nabla$. Then
	\begin{equation}
		\mathcal P_\tau=\int_0^\infty e^{-\lambda\tau}\,dE_\lambda,\quad -\frac{s}{\Gamma(1-s)}\int_0^\infty\frac{e^{-(\lambda+\mathsf{i}\rho)\tau}-1}{\tau^{1+s}}\,d\tau=(\lambda+\mathsf{i}\rho)^s,
	\end{equation}
	for $\lambda>0$ and $\rho\in\R$, where $\mathsf{i}=\sqrt{-1}$. Therefore,
	\begin{equation}
		\mathcal F_t\LC\mathcal P_\tau u\RC(\rho,\cdot)=e^{-\mathsf{i}\rho\tau}\mathcal P_\tau\LC\mathcal F_tu(\rho,\cdot)\RC,
	\end{equation}
	and the Fourier analogue of \eqref{H^s} is
	\begin{equation}\label{spectral formula Hs}
		\mathcal F_t\LC H_\sigma^s u\RC(\rho,\cdot)=\int_0^\infty(\lambda+\mathsf{i}\rho)^s\,dE_\lambda\LC\mathcal F_tu(\rho,\cdot)\RC.
	\end{equation}
	
	\subsection{Function spaces}
	
	Given $u\in\mathcal S(\R^{n+1})$, the spectral formula \eqref{spectral formula Hs} yields
	\begin{equation}
		\norm{\mathcal F_t\LC H_\sigma^s u\RC(\rho,\cdot)}_{L^2(\R^n)}^2=\int_0^\infty\abs{\lambda+\mathsf{i}\rho}^{2s}\,d\norm{E_\lambda\LC\mathcal F_tu(\rho,\cdot)\RC}^2.
	\end{equation}
	For $a\in\R$, we define $\mathbf H^a(\R^{n+1})$ as the completion of $\mathcal S(\R^{n+1})$ with respect to the norm
	\begin{equation}\label{spectral H norm}
		\norm{u}_{\mathbf H^a(\R^{n+1})}:=\LC\int_{\R}\int_0^\infty \LC1+\abs{\lambda+\mathsf{i}\rho}^2\RC^{a/2}\,d\norm{E_\lambda\LC\mathcal F_tu(\rho,\cdot)\RC}^2\,d\rho\RC^{1/2}.
	\end{equation}
	Given an open set $\mathcal O\subset\R^{n+1}$, we set
	\begin{equation}
		\mathbf H^a(\mathcal O):=\left\{u|_{\mathcal O}:\,u\in\mathbf H^a(\R^{n+1})\right\},\quad \widetilde{\mathbf H}^a(\mathcal O):=\overline{C_c^\infty(\mathcal O)}^{\mathbf H^a(\R^{n+1})},
	\end{equation}
	where
	\begin{equation}
		\norm{u}_{\mathbf H^a(\mathcal O)}:=\inf\left\{\norm{v}_{\mathbf H^a(\R^{n+1})}:\,v|_{\mathcal O}=u\right\}.
	\end{equation}
	The dual spaces are denoted by
	\begin{equation}
		\mathbf H^{-a}(\mathcal O)=\widetilde{\mathbf H}^a(\mathcal O)^\ast,\quad \widetilde{\mathbf H}^{-a}(\mathcal O)=\LC\mathbf H^a(\mathcal O)\RC^\ast.
	\end{equation}
	
	We will also use the parabolic fractional Sobolev space
	\begin{equation}
		\mathbb H^a(\R^{n+1}):=\left\{u\in L^2(\R^{n+1}):\,\LC|\xi|^2+\mathsf{i}\rho\RC^{a/2}\widehat u(\rho,\xi)\in L^2(\R^{n+1})\right\},
	\end{equation}
	where $\widehat u(\rho,\xi)=\int_{\R^{n+1}}e^{-\mathsf{i}(t,x)\cdot(\rho,\xi)}u(t,x)\,dtdx$ denotes the Fourier transform with respect to the $(t,x)$ variables. The graph norm is
	\begin{equation}\label{norm of H^s}
		\norm{u}_{\mathbb H^a(\R^{n+1})}^2:=\int_{\R^{n+1}}\LC1+\LC|\xi|^4+|\rho|^2\RC^{1/2}\RC^a\abs{\widehat u(\rho,\xi)}^2\,d\rho d\xi.
	\end{equation}
	Equivalently, $\mathbb H^a(\R^{n+1})=\mathbb H^{a/2,a}(\R^{n+1})$, where the exponents $a/2$ and $a$ denote the fractional derivatives in time and space, respectively. In particular, when $a=s\in(0,1)$, the parabolic version of the Kato square root problem in \cite[Section 3]{BS2024calderon}, which is based on \cite{AEN2020}, gives
	\begin{equation}\label{space identification}
		\mathbb H^s(\R^{n+1})=\mathbf H^s(\R^{n+1}),\quad s\in(0,1).
	\end{equation}
	For any closed set $E\subset\R^{n+1}$, we denote
	\begin{equation}
		\mathbb H_E^s:=\left\{u\in\mathbb H^s(\R^{n+1}):\,\supp(u)\subset E\right\}.
	\end{equation}
	We also recall that the fractional Sobolev space $H^s(\R^n)$, for $s\in(0,1)$, is given by
	\begin{equation}\label{fractional Sobolev space}
		H^s(\R^n):=\left\{f\in\mathcal S'(\R^n):\,\norm{f}_{H^s(\R^n)}:=\left\|\LC1+|\xi|^2\RC^{s/2}\widehat f(\xi)\right\|_{L^2(\R^n)}<\infty\right\}.
	\end{equation}
	For an open set $D\subset\R^n$, we also use the standard notation
	\[
	H^a(D):=\{U|_D:\ U\in H^a(\R^n)\},\quad
	\widetilde H^a(D):=\overline{C_c^\infty(D)}^{H^a(\R^n)}.
	\]
	The corresponding dual spaces are denoted by
	\[
	H^{-a}(D):=(\widetilde H^a(D))^\ast,\quad
	\widetilde H^{-a}(D):=(H^a(D))^\ast.
	\]
	
	In the rest of the paper, we use $\lesssim$ (resp. $\approx$) to denote that an inequality (resp. equality) holds up to a positive constant whose exact value is irrelevant in our arguments.
	
	\subsection{Well-posedness for the nonlocal parabolic equation}\label{subsection: well-posed}
	
	The well-posedness of \eqref{equ nonlocal para} has been studied in \cite{BS2024calderon,LLR2019calder,LLU2022para}. More precisely, given exterior data $f\in\widetilde{\mathbf H}^s((\Omega_e)_T)$, there exists a unique solution $u_f\in\mathbf H^s(\R^{n+1})$ solving \eqref{equ nonlocal para}. Moreover, the past-time condition in \eqref{equ nonlocal para} is understood in the sense that the solution is extended by zero for $t\leq -T$. We will also use the following elementary convention. Since all arguments below take place in the finite cylinder $\Omega_T$, we may multiply the exterior data by a smooth time cutoff which is identically one on $(-T,T)$ and compactly supported in a slightly larger time interval. By the causality of the semigroup representation, this modification does not change the solution in $\Omega_T$. Hence, in the density argument and in all integrations by parts in time, we may work first with smooth exterior data for which the corresponding solution and its extension vanish for $t\leq -T$ and are compactly supported in time after a harmless cutoff. The general case then follows by density.
	
	In particular, for $f\in C^\infty_c(W_T)$, the solution $u_f$ satisfies $u_f(t,x)=0$, $t\leq -T$, $x\in\R^n$, and the representation formula \eqref{wt u expression} gives
	\begin{equation}
		\wt u_f(t,x,y)=0,\quad t\leq -T,\quad (x,y)\in\R^{n+1}_+.
	\end{equation}
	Finally, if $u_f\in\mathbf H^s(\R^{n+1})$ is the solution to \eqref{equ nonlocal para}, then
	\begin{equation}\label{equ nonlocal para estimate}
		\norm{u_f}_{\mathbf H^s(\R^{n+1})}\lesssim\norm{f}_{\widetilde{\mathbf H}^s((\Omega_e)_T)}.
	\end{equation}
	This estimate follows from \cite[Theorem 3.3]{BS2024calderon}.
	
	\subsection{The parabolic extension problem and a duality principle}\label{subsec: extension-duality}
	
	The nonlocal parabolic operator $H_\sigma^s=\LC\p_t-\nabla\cdot\sigma\nabla\RC^s$ can be characterized by a parabolic Caffarelli--Silvestre extension. We recall the precise formulation that will be used throughout the paper.
	
	As in \cite[Section 3.1]{BS2024calderon}, given an open set $\Sigma\subseteq\R^{n+2}_+=\R^{n+1}\times(0,\infty)$, consider the energy space
	\begin{equation}
		\mathcal L^{1,2}(\Sigma;y^{1-2s}dtdxdy):=\left\{\wt u:\,\wt u,\p_{x_k}\wt u,\p_y\wt u\in L^2(\Sigma,y^{1-2s}dtdxdy),\ k=1,\ldots,n\right\},
	\end{equation}
	where
	\begin{equation}
		\norm{\wt u}_{L^2(\Sigma,y^{1-2s}dtdxdy)}^2:=\int_\Sigma y^{1-2s}\abs{\wt u}^2\,dtdxdy,
	\end{equation}
	and
	\begin{equation}
		\norm{\wt u}_{\mathcal L^{1,2}(\Sigma;y^{1-2s}dtdxdy)}:=\LC\int_\Sigma y^{1-2s}\LC\abs{\wt u}^2+\abs{\nabla\wt u}^2+\abs{\p_y\wt u}^2\RC\,dtdxdy\RC^{1/2}.
	\end{equation}
	Here and below $\nabla\equiv\nabla_x$, $\nabla_{x,y}\equiv(\nabla_x,\p_y)$, and $\wt\sigma$ is given by \eqref{tilde sigma}.
	
	\begin{proposition}[Extension problem]\label{Prop: extension}
		Given $s\in(0,1)$ and $u\in\mathbf H^s(\R^{n+1})$, there exists a solution $\wt u$ to \eqref{degen para PDE} satisfying the following properties:
		\begin{enumerate}[(a)]
			\item\label{item a prop extension} $\lim_{y\to0}\wt u(t,x,y)=u(t,x)$ in $\mathbf H^s(\R^{n+1})$.
			\item\label{item b prop extension} $H_\sigma^s u=d_s\lim_{y\to0}y^{1-2s}\p_y\wt u(t,x,y)$ in $\mathbf H^{-s}(\R^{n+1})$, for some constant $d_s$ depending only on $s\in(0,1)$.
			\item\label{item c prop extension} $\norm{\wt u}_{\mathcal L^{1,2}(\R^{n+2}_+;y^{1-2s}dtdxdy)}\leq C\norm{u}_{\mathbf H^s(\R^{n+1})}$, for some constant $C>0$ independent of $u$ and $\wt u$.
		\end{enumerate}
	\end{proposition}
	
	Since the proof is given in \cite[Theorem 3.1]{BS2024calderon}, we omit it. Let us recall the representation formula for the extension. The function $\wt u$ can be written as
	\begin{equation}\label{wt u expression}
		\wt u(t,x,y)=\int_0^\infty\int_{\R^n}P_y^s(x,z,\tau)u(t-\tau,z)\,dzd\tau,
	\end{equation}
	where
	\begin{equation}\label{kernel for extension}
		P_y^s(x,z,\tau):=\frac{1}{2^{2s}\Gamma(s)}\frac{y^{2s}}{\tau^{1+s}}e^{-\frac{y^2}{4\tau}}p(x,z,\tau).
	\end{equation}
	Here $p(x,z,\tau)$ is the heat kernel satisfying \eqref{heat kernel}. The constant $d_s$ in Proposition \ref{Prop: extension} is
	\begin{equation}\label{d_s}
		d_s:=\frac{2^{2s-1}\Gamma(s)}{\Gamma(1-s)}.
	\end{equation}
	
	We next record a duality relation for the weighted parabolic extension equations. This observation is elementary, but it explains the origin of the averaging transform, which will be introduced in Section \ref{sec: duality}.
	
	\begin{proposition}[Duality]\label{Prop: duality}
		Let $\wt\sigma\in C^2(\R^n;\R^{(n+1)\times(n+1)})$ be of the form \eqref{tilde sigma}, let $s\in(0,1)$, and let $h\in C^0(\R^{n+1})$. Suppose that $u_1\in C^2(\R^{n+2}_+)$ with $y^{2s-1}\p_yu_1\in C^0(\overline{\R^{n+2}_+})$ is a classical solution of
		\begin{equation}\label{conjugate u1}
			\begin{cases}
				y^{2s-1}\p_tu_1-\nabla_{x,y}\cdot\LC y^{2s-1}\wt\sigma\nabla_{x,y}u_1\RC=0 &\text{ in }\R^{n+2}_+,\\
				\displaystyle\lim_{y\to0}y^{2s-1}\p_yu_1=h &\text{ on }\R^{n+1}\times\{0\}.
			\end{cases}
		\end{equation}
		Then the function $u_2(t,x,y):=-y^{2s-1}\p_yu_1(t,x,y)$ is a classical solution of
		\begin{equation}\label{conjugate u2}
			\begin{cases}
				y^{1-2s}\p_tu_2-\nabla_{x,y}\cdot\LC y^{1-2s}\wt\sigma\nabla_{x,y}u_2\RC=0 &\text{ in }\R^{n+2}_+,\\
				u_2=-h &\text{ on }\R^{n+1}\times\{0\}.
			\end{cases}
		\end{equation}
	\end{proposition}
	
	\begin{proof}
		For $y>0$, we compute
		\begin{align}
			&\quad \, -y^{1-2s}\p_tu_2+\nabla_{x,y}\cdot\LC y^{1-2s}\wt\sigma\nabla_{x,y}u_2\RC \notag\\
			&=y^{1-2s}\p_t\LC y^{2s-1}\p_yu_1\RC-\nabla_{x,y}\cdot\LC y^{1-2s}\wt\sigma\nabla_{x,y}\LC y^{2s-1}\p_yu_1\RC\RC \notag\\
			& =\p_y\p_tu_1-\p_y\nabla_x\cdot\LC\sigma\nabla_xu_1\RC-\p_y\left\{y^{1-2s}\left[\p_y\LC y^{2s-1}\p_yu_1\RC\right]\right\} \notag\\
			& =\p_y\left\{\p_tu_1-y^{1-2s}\nabla_{x,y}\cdot\LC y^{2s-1}\wt\sigma\nabla_{x,y}u_1\RC\right\}=0,
		\end{align}
		where we used \eqref{conjugate u1} in the last identity. The boundary condition follows from the definition of $u_2$.
	\end{proof}
	
	From the relation $u_2(t,x,y)=-y^{2s-1}\p_yu_1(t,x,y)$, we have $-\p_yu_1(t,x,y)=y^{1-2s}u_2(t,x,y)$. Assuming that $\lim_{\mu\to\infty}u_1(t,x,\mu)=0$, integration on $(y,\infty)$ gives
	\begin{equation}\label{duality integral observation}
		u_1(t,x,y)=\int_y^\infty\mu^{1-2s}u_2(t,x,\mu)\,d\mu.
	\end{equation}
	This observation is the starting point for constructing an averaged local parabolic solution from the extension of a fractional parabolic solution.
	
	\section{The averaging transform and auxiliary functions}\label{sec: duality}
	
	We now introduce the functions that connect the nonlocal parabolic equation to the local parabolic equation. Let $u\in\mathbf H^s(\R^{n+1})$ be the solution to \eqref{equ nonlocal para}, and let $\wt u$ be its parabolic extension solving \eqref{degen para PDE}. Motivated by \eqref{duality integral observation}, we define
	\begin{equation}\label{the function w}
		w(t,x,y):=\int_y^\infty\mu^{1-2s}\wt u(t,x,\mu)\,d\mu,\quad y>0.
	\end{equation}
	For each fixed $y>0$, the function $w(t,x,y)$ is finite by the decay estimates proved later in Section \ref{sec: regularity}. We define the vertical average of $\wt u$ by
	\begin{equation}\label{v=w0}
		v(t,x):=\lim_{y\to 0}w(t,x,y)=\int_0^\infty y^{1-2s}\wt u(t,x,y)\,dy.
	\end{equation}
	The rigorous regularity properties of $v$ are established in Section \ref{sec: regularity}. In particular, we will prove that $v\in L^2(-T,T;H^1(\Omega))$ and $\p_t v\in L^2(-T,T;H^{-1}(\Omega))$. Let us explain formally why $v$ is expected to solve the local parabolic equation. Since $\wt u$ solves \eqref{degen para PDE}, a direct computation gives
	\begin{equation}\label{comp 1}
		\begin{split}
			&\quad\, \nabla_{x,y}\cdot\LC y^{2s-1}\wt\sigma\nabla_{x,y}w(t,x,y)\RC\\
			&=y^{2s-1}\int_y^\infty\mu^{1-2s}\nabla\cdot\LC\sigma\nabla\wt u(t,x,\mu)\RC\,d\mu-\p_y\wt u(t,x,y) \\
			&=-y^{2s-1}\int_y^\infty\p_\mu\LC\mu^{1-2s}\p_\mu\wt u(t,x,\mu)\RC\,d\mu +y^{2s-1}\int_y^\infty\mu^{1-2s}\p_t\wt u(t,x,\mu)\,d\mu-\p_y\wt u(t,x,y) \\
			&=y^{2s-1}\p_t\LC\int_y^\infty\mu^{1-2s}\wt u(t,x,\mu)\,d\mu\RC \\
			&=y^{2s-1}\p_t w(t,x,y).
		\end{split}
	\end{equation}
	Here we used the equation \eqref{degen para PDE}, the definition \eqref{the function w}, and the decay of $\wt u$ in the extension variable. Therefore, $w$ formally solves
	\begin{equation}\label{extension problem for v-w}
		\begin{cases}
			y^{2s-1}\p_t w-\nabla_{x,y}\cdot\LC y^{2s-1}\wt\sigma\nabla_{x,y}w\RC=0 &\text{ in }\R^{n+2}_+,\\
			w(t,x,0)=v(t,x) &\text{ on }\R^{n+1}.
		\end{cases}
	\end{equation}
	Thus, by applying the extension characterization with the exponent $1-s$, the operator $H_\sigma^{1-s}$ is formally realized as the Neumann data of \eqref{extension problem for v-w}. More precisely,
	\begin{equation}\label{u-v relation}
		\begin{split}
			H_\sigma^{1-s}v(t,x)=d_{1-s}\lim_{y\to0}y^{2s-1}\p_yw(t,x,y)=-d_{1-s}\wt u(t,x,0)=-d_{1-s}u(t,x),
		\end{split}
	\end{equation}
	where $d_{1-s}$ is the constant in \eqref{d_s} with $s$ replaced by $1-s$. With the normalization \eqref{d_s}, one has $d_{1-s}=d_s^{-1}$. Applying $H_\sigma^s$ to \eqref{u-v relation} and using the semigroup property gives the formal identity $H_\sigma v=-d_s^{-1}H_\sigma^s u$ in $\R^{n+1}$. Consequently, if $u$ solves $H_\sigma^s u=0$ in $\Omega_T$, then the averaged function $v$ should solve
	\begin{equation}
		H_\sigma v=\LC\p_t-\nabla\cdot\sigma\nabla\RC v=0\quad\text{in }\Omega_T.
	\end{equation}
	The main purpose of Sections \ref{sec: regularity} and \ref{sec: key equation} is to make this formal argument rigorous. Section \ref{sec: regularity} proves that the integral defining $v$ belongs to the local parabolic energy class, while Section \ref{sec: key equation} proves the weak equation satisfied by $v$.

	\section{Regularity estimates of solutions}\label{sec: regularity}
	
	In this section, we establish regularity properties of the averaged function $v(t,x)$ defined by \eqref{v=w0}. Moreover, with suitable decay estimates at hand, we can make our arguments hold rigorously.
	
	\begin{proposition}\label{Prop: bound w}
		Let $\Omega, W\subset \R^n$ be bounded open sets with Lipschitz boundaries, for $n\geq 2$ and $T>0$. Let $u\in \mathbf{H}^s(\R^{n+1})$ with compact support in $\overline{(\Omega \cup W)_T}$ and $\wt u\in \mathcal{L}^{1,2}(\R^{n+2}_+; y^{1-2s}dtdxdy)$ be the corresponding solution of the extension problem \eqref{degen para PDE}. Let $w:\R^{n+2}_+\to \R$ be the function given by \eqref{the function w}. Then for every $y>0$, the function $w(\cdot,\cdot,y)\in L^m(-T,T;W^{1,\infty}(\R^n))$ for any $m\in [1,2]$. Moreover, for every bounded Lipschitz open set $\Omega'\subset\R^n$, the limit \eqref{v=w0} exists in $L^2(-T,T;H^1(\Omega'))$, and there holds
		\begin{equation}\label{estimate u and v local}
			\left\| v \right\|_{L^2(-T,T;H^1(\Omega'))}\lesssim  \left\| u \right\|_{\mathbf{H}^s(\R^{n+1})}.
		\end{equation}
		In particular, taking $\Omega'=\Omega$, one has
		\begin{equation}\label{estimate u and v}
			\left\| v \right\|_{L^2(-T,T;H^1(\Omega))}\lesssim \left\| u \right\|_{\mathbf{H}^s(\R^{n+1})}.
		\end{equation}
	\end{proposition}
	
	Before showing Proposition \ref{Prop: bound w}, we first prove several decay estimates for the solution $\wt u$ to \eqref{degen para PDE} with respect to the $y$-direction.
	
	\begin{lemma}\label{Lemma: decay estimate}
		Let $\Omega, W\subset \R^n$ be bounded open sets with Lipschitz boundaries such that $\overline{\Omega} \cap \overline{W}=\emptyset$, for $n\in \N$ and $T>0$. Given $s\in (0,1)$, let $u\in \mathbf{H}^s(\R^{n+1})$ with $\supp(u)\subset \overline{(\Omega\cup W)_T}$, and let $\wt u$ be the corresponding solution of the extension problem \eqref{degen para PDE}. Then for any $m\in[1,2]$, the function $\wt u(t,x,y)$ satisfies the following decay estimates:
		\begin{enumerate}[(a)]
			\item\label{item a decay x}
			\begin{equation}\label{decay estimate in x}
				\begin{split}
					\left\|   \wt u(\cdot,x ,y)  \right\|_{L^m(\R)} & \lesssim y^{-n} \norm{u}_{L^m(\R; L^1(\R^{n}))}, \\
					\left\|  \nabla \wt u(\cdot,x ,y)  \right\|_{ L^m(\R)} & \lesssim y^{-n-1} \norm{u}_{L^m (\R;L^1(\R^{n}))},
				\end{split}
			\end{equation}
			for any $(x,y)\in \R^{n+1}_+:=\R^n \times (0,\infty)$.
			
			\item\label{item c decay y} For $1\leq p,q,r \leq \infty$ with $1+\frac{1}{r}=\frac{1}{p}+\frac{1}{q}$, $\wt u$ satisfies
			\begin{equation}\label{decay estimate in y}
				\begin{split}
					\left\|  \wt u(\cdot ,\cdot,y)\right\|_{L^m(\R;L^r(\R^{n}))} &\lesssim y^{\frac{n}{p}-n}\norm{u}_{L^m(\R;L^q(\R^n))},\\
					\left\|\nabla_{x,y}\wt u(\cdot ,\cdot , y)	\right\|_{L^m(\R;L^r(\R^{n}))}  &\lesssim y^{\frac{n}{p}-n-1}\norm{u}_{L^m(\R;L^q(\R^n))}.
				\end{split}
			\end{equation}
		\end{enumerate}
	\end{lemma}
	
	Before proving the lemma, let us offer an auxiliary result, which will be used in the proof of Lemma \ref{Lemma: decay estimate}.
	
	\begin{lemma}\label{Lemma: comp 1 in appendix}
		For any $b, A>0$, let $f_b(A):=	\int_0^\infty \tau^{-(b+1)} e^{-\frac{A}{4\tau}}\, d\tau $, then there holds
		\begin{equation}\label{comp1 in appendix}
			\int_0^\infty \tau^{-(b+1)} e^{-\frac{A}{4\tau}}\, d\tau =f_b(1)A^{-b}.
		\end{equation}
	\end{lemma}
	
	\begin{proof}
		By integration by parts, one has
		\begin{equation}
			\begin{split}
				f_b(A)=-\int_0^\infty \frac{d}{d\tau} \LC \frac{1}{b}\tau^{-b}\RC e^{-\frac{A}{4\tau}}\, d\tau  = \frac{1}{b}\int_0^\infty \tau^{-b} \frac{d}{d\tau} \LC e^{-\frac{A}{4\tau}}\RC d\tau =\frac{A}{4b} \int_0^\infty \tau^{-(b+2)}e^{-\frac{A}{4\tau}}\, d\tau.
			\end{split}
		\end{equation}
		Moreover, it is easy to see that $f_b'(A)=-\frac{1}{4} \int_0^\infty \tau^{-(b+2)}e^{-\frac{A}{4\tau}}\, d\tau$. From the preceding computations, one has $f_{b}(A)=-\frac{A}{b}f_b ' (A)$, which implies $\frac{f_b' (A)}{f_b(A)}=-\frac{b}{A}$. Solving the above ordinary differential equation, it follows that
		\begin{equation}
			f_b(A)= f_b(1)A^{-b},
		\end{equation}
		which proves the assertion.
	\end{proof}
	
	With the aid of the preceding lemma, we can show Lemma \ref{Lemma: decay estimate}.
	
	\begin{proof}[Proof of Lemma \ref{Lemma: decay estimate}]
		For \ref{item a decay x}, combined with formulas \eqref{wt u expression} of $\wt u$ and \eqref{kernel for extension}, one has
		\begin{equation}\label{representation formula for wt u}
			\wt u(t,x,y)=c_s y^{2s}\int_{\R^{n}} \int_0^\infty  e^{-\frac{y^2}{4\tau }} p(x,z,\tau)u(t-\tau,z)\, \frac{d\tau }{\tau^{1+s}} dz,
		\end{equation}
		where $p(x,z,\tau)$ denotes the heat kernel satisfying \eqref{heat kernel} and $c_s$ is a constant depending only on $s$ of the form
		\begin{equation}\label{c_s}
			c_s:=\frac{1}{2^{2s}\Gamma(s)}.
		\end{equation}
		Hence, with the heat kernel estimate \eqref{heat kernel estimate} at hand, we can have the pointwise estimate
		\begin{equation}\label{comp of decay1}
			\begin{split}
				\big| \nabla^\ell_x \wt u(t,x,y) \big|  &\lesssim y^{2s} \int_{\R^n}\int_0^\infty  \left| \nabla^\ell_x p(x,z,\tau) \right| e^{-\frac{y^2}{4\tau }} \abs{u(t-\tau,z)}\,  \frac{d\tau}{\tau^{1+s}} dz \\
				&\lesssim y^{2s} \int_{\R^n}\int_0^\infty   \tau^{-(\frac{n+\ell}{2}+1+s)}e^{- (c \frac{|x-z|^2}{\tau}+\frac{y^2}{4\tau }) }  \abs{u(t-\tau,z)}\,  d\tau  dz,
			\end{split}
		\end{equation}
		for $\ell=0,1$. Note that the right-hand side of \eqref{comp of decay1} can be viewed as a convolution of space and time variables, where one may extend the integrand to be zero in the region $\tau\in (-\infty,0)$.
		
		Recall that Young's convolution inequality states
		\begin{equation}\label{Young convolution}
			\norm{f\ast g}_{L^r}=\left\| \int f(\zeta )g(\xi -\zeta)\, d\zeta \right\|_{L^r_{\xi}} \leq \norm{f}_{L^p}\norm{g}_{L^q}, \text{ with }1+\frac{1}{r}=\frac{1}{p}+\frac{1}{q},
		\end{equation}
		where $\zeta, \xi$ could be either space or time variables in the upcoming applications. To make our notation clearer, here we use $L^r_\xi$ to denote the $L^r$ Lebesgue norm with respect to the $\xi$-variable. Let us take $L^m$ norm in the time variable to \eqref{comp of decay1}. Applying \eqref{Young convolution} for the exponents $p=1$, $q=m \in [1,2]$ with respect to the time-variable, one can have
		\begin{equation}\label{comp of decay2}
			\begin{split}
				\big\|  \nabla^\ell_x \wt u(\cdot,x,y) \big\|_{L^m(\R)} \lesssim y^{2s}\int_{\R^n} \left\{ \LC \int_{0}^\infty   \tau^{-(\frac{n+\ell}{2}+1+s)}e^{- (c \frac{|x-z|^2}{\tau}+\frac{y^2}{4\tau }) } \, d\tau \RC   \norm{u(\cdot, z)}_{L^m(\R)} \right\} dz.
			\end{split}
		\end{equation}
		Observing that
		\begin{equation}\label{comp of decay3}
			\int_{0}^\infty  \frac{e^{- (c \frac{|x-z|^2}{\tau}+\frac{y^2}{4\tau }) } }{ \tau^{\frac{n+\ell}{2}+s+1}}\, d\tau \approx \LC\abs{x-z}^2 +y^2 \RC^{-\frac{n+\ell}{2}-s}
		\end{equation}
		as shown in \eqref{comp1 in appendix} from Lemma \ref{Lemma: comp 1 in appendix}. Inserting \eqref{comp of decay2} and \eqref{comp of decay3} into \eqref{comp of decay1}, we obtain
		\begin{equation}\label{comp of decay x}
			\begin{split}
				\big\|   \nabla^\ell_x \wt u(\cdot,x ,y)  \big\|_{L^m(\R)}
				\lesssim y^{2s}\int_{\R^n} \left\{ \LC\abs{x-z}^2 +y^2 \RC^{-\frac{n+\ell}{2}-s} \norm{u(\cdot, z)}_{L^m(\R)} \right\} dz,
			\end{split}
		\end{equation}
		for $\ell=0,1$.
		
		Meanwhile, let us point out that the right-hand side of \eqref{comp of decay x} can be regarded as the convolution with respect to the space variable. Next, applying \eqref{Young convolution} again to the $x$-variable, for the exponents $r=\infty$, $p=\infty$ and $q=1$, we have
		\begin{equation}
			\big\|   \nabla^\ell_x \wt u(\cdot,\cdot ,y)  \big\|_{L^1(\R; L^\infty(\R^n))}\lesssim y^{-n-\ell} \norm{u}_{L^1(\R^{n+1})},
		\end{equation}
		where we used the fact that $\sup_{z\in \R^n} \big( \abs{z}^2 +y^2 \big)^{-\frac{n+\ell}{2}-s} \leq y^{-n-\ell-2s}$, for $\ell=0,1$. Similar to the computation \eqref{comp of decay x}, we obtain
		\begin{equation}
			\big\|   \nabla^\ell_x \wt u(\cdot,\cdot ,y)  \big\|_{L^m(\R; L^r(\R^n))}  \lesssim y^{n/p-n-\ell} \norm{u}_{L^m(\R; L^q(\R^{n}))},
		\end{equation}
		for $m\in [1,2]$, where we used the fact that
		\begin{equation}\label{L^p estimate for kernel}
			\bigg\| \frac{y^{2s}}{ (\abs{z}^2 +y^2 )^{\frac{n+\ell}{2}+s}}\bigg\|_{L^p(\R^n)}  \approx \bigg( \int_0^\infty \frac{y^{2sp}r^{n-1}}{(r^2+y^2)^{p(\frac{n+\ell}{2}+s)}} \, dr\bigg)^{1/p} \lesssim y^{n/p-n-\ell},
		\end{equation}
		for $\ell =0,1$. This shows \ref{item a decay x}.
		
		For \ref{item c decay y}, we can calculate
		\begin{equation}\label{comp of decay y1}
			\begin{split}
				\big| \p_y \wt u(t,x,y) \big| & \lesssim y^{2s-1} \int_{\R^n}\int_0^\infty p(x,z,\tau) e^{-\frac{y^2}{4\tau}}|u(t-\tau,z) |\, \frac{d\tau}{\tau^{1+s}} dz \\
				& \quad \,  + y^{2s+1} \int_{\R^n}\int_0^\infty p(x,z,\tau) e^{-\frac{y^2}{4\tau}}|u(t-\tau,z) |\, \frac{d\tau}{\tau^{2+s}} dz .
			\end{split}
		\end{equation}
		Similar to previous cases, we can make use of Young's convolution inequality again for both terms in the right-hand side of \eqref{comp of decay y1}, then direct computations yield that
		\begin{equation}\label{decay y detail 1}
			\begin{split}
				&\quad \, \left\| y^{2s-1} \int_{\R^n}\int_0^\infty p(x,z,\tau) e^{-\frac{y^2}{4\tau}}|u(t-\tau,z) |\, \frac{d\tau}{\tau^{1+s}} dz \right\|_{L^m(\R; L^r(\R^n))}\\
				& \lesssim \left\|  y^{2s-1} \int_{\R^n}\int_0^\infty \tau^{-(\frac{n}{2}+1+s)}e^{- (c \frac{|x-z|^2}{\tau}+\frac{y^2}{4\tau }) } |u(t-\tau,z) |\, d\tau dz \right\|_{L^m(\R; L^r(\R^n))} \\
				&\lesssim \left\|  y^{2s-1} \int_{\R^n} \LC\abs{x-z}^2 +y^2 \RC^{-\frac{n}{2}-s}  \norm{u(\cdot,z)}_{L^m(\R)} \,  dz\right\|_{L^r(\R^n)} \\
				&\lesssim 	 \left\| \frac{y^{2s-1}}{ (\abs{z}^2 +y^2 )^{\frac{n}{2}+s}}\right\|_{L^p(\R^n)} \norm{u}_{L^m(\R; L^q(\R^n))}  \\
				&\lesssim y^{n/p-n-1}\norm{u}_{L^m(\R; L^q(\R^n))} ,
			\end{split}
		\end{equation}
		and similarly,
		\begin{equation}\label{decay y detail 2}
			\begin{split}
				&\quad \, \left\|  y^{2s+1} \int_{\R^n}\int_0^\infty p(x,z,\tau) e^{-\frac{y^2}{4\tau}}|u(t-\tau,z) |\, \frac{d\tau}{\tau^{2+s}} dz\right\|_{L^m(\R; L^r(\R^n))}  \\
				&\lesssim \left\|  y^{2s+1} \int_{\R^n} \LC \abs{x-z}^2 +y^2 \RC^{-\frac{n}{2}-1-s} \norm{u(\cdot,z)}_{L^m(\R)} dz \right\|_{L^r(\R^n)} \\
				&\lesssim y^{n/p-n-1}\norm{u}_{ L^m(\R;L^q(\R^n))} ,
			\end{split}
		\end{equation}
		where we used the fact \eqref{L^p estimate for kernel}. This proves the assertion.
	\end{proof}
	
	We are ready to prove Proposition \ref{Prop: bound w}.
	
	\begin{proof}[Proof of Proposition \ref{Prop: bound w}]
		Since $u$ is supported in $\overline{\LC \Omega \cup W\RC_T}$, we assume that there exists a ball $B_R\subset \R^n$ centered at the origin and radius $R>0$ such that $\supp(u) \subset [-T,T]\times B_R$. Let us argue in five steps.
		
		\medskip
		
		{\it Step 1. Initial regularity.}
		
		\medskip
		
		\noindent By \eqref{comp of decay1}, \eqref{comp of decay3} and compact condition of $\overline{\supp(u)}$, we have that
		\begin{equation}\label{w1}
			\begin{split}
				\big| \nabla^\ell_x \wt u(t,x,y) \big|
				&\lesssim y^{2s}\int_{\R^n}\int_0^\infty  \left| \nabla^\ell_x p(x,z,\tau) \right| e^{-\frac{y^2}{4\tau }} \abs{u(t-\tau,z)}\,  \frac{d\tau}{\tau^{1+s}} dz \\
				&\lesssim y^{2s} \int_{\R^n}\int_0^\infty   \tau^{-(\frac{n+\ell}{2}+1+s)}e^{- (c \frac{|x-z|^2}{\tau}+\frac{y^2}{4\tau }) }  \abs{u(t-\tau,z)}\,  d\tau  dz\\
				&\lesssim y^{2s} \int_{\R^n}\LC\int_0^\infty   \tau^{-2(\frac{n+\ell}{2}+1+s)}e^{- 2(c \frac{|x-z|^2}{\tau}+\frac{y^2}{4\tau }) }\,d\tau \RC^{1/2} \LC\int_0^\infty\abs{u(t-\tau,z)}^2\,  d\tau\RC^{1/2}   dz\\
				&\lesssim y^{2s} \int_{\R^n}\LC\abs{x-z}^2 +y^2 \RC^{-\frac{n+\ell+1+2s}{2}} \norm{u(\cdot,z)}_{L^2(\R)}  \, dz,
			\end{split}
		\end{equation}
		for $\ell=0,1$, where we used H\"older's inequality and \eqref{comp1 in appendix}. Furthermore, we can obtain
		\begin{equation}
			\big| \nabla^\ell_x \wt u(t,x,y) \big|\lesssim y^{-n-\ell-1}\norm{u}_{ L^2(-T,T; L^1(\R^{n}))},
		\end{equation}
		where we used the fact $y^{2s}\big(\abs{x-z}^2 +y^2 \big)^{-\frac{n+\ell+1+2s}{2}} \leq y^{-n-\ell-1}$ for $\ell=0,1$.
		
		Next, we want to check that $w(t,x,y)$ is well-defined for $y>0$. To this end, we can estimate the function $w$ for $y>0$:
		\begin{equation}\label{w2}
			\begin{split}
				\big| \nabla^\ell_x w(t,x,y) \big| & \leq \int_{y}^{\infty} \mu^{1-2s} \big| \nabla^\ell_x \wt u(t,x,\mu) \big| \, d\mu  \\
				&\lesssim \int_{y}^\infty \mu^{1-2s-n-\ell-1}\norm{u}_{ L^2(-T,T; L^1(\R^{n}))}\, d\mu \\
				&\lesssim  y^{1-2s-n-\ell}\norm{u}_{ L^2(-T,T; L^1(\R^{n}))},
			\end{split}
		\end{equation}
		for $\ell=0,1$. Since $\overline{\supp(u)}$ is compact and $u\in \mathbf{H}^s(\R^{n+1})$, the right hand side of the estimate \eqref{w2} is finite for a.e. $(t,x,y)\in \R^{n+2}_+$.
		
		\medskip
		
		{\it Step 2. $L^2$-estimate for $v(t,x)=w(t,x,0)$.}
		
		\medskip
		
		\noindent First, Minkowski's integral inequality implies that
		\begin{equation}\label{M-integral}
			\begin{split}
				\bigg( \int_{-T}^T \LC \int_y^\infty \mu^{1-2s}\left| \wt u(t,x,\mu)  \right|  d\mu \RC^m dt\bigg)^{1/m}&\leq \int_y^\infty \mu^{1-2s} \bigg(  \int_{-T}^T \left| \wt u(t,x,\mu) \right|^m\, dt \bigg)^{1/m} d\mu \\
				&=\int_y^\infty \mu^{1-2s}\left\| \wt u (\cdot, x,\mu)\right\|_{L^m(-T,T)}\, d\mu,
			\end{split}
		\end{equation}
		for $m\in [1,2]$. By using \eqref{decay estimate in y}, we have that
		\begin{equation}\label{initial estimate for w}
			\begin{split}
				\big\| \nabla_x^\ell w(\cdot,\cdot ,y) \big\|_{ L^m(-T,T; L^\infty(\R^{n}))} &\leq \int_y^\infty \mu^{1-2s}\big\| \nabla_x^\ell \wt u(\cdot,\cdot,\mu) \big\|_{ L^m(-T,T; L^\infty(\R^{n}))} \,  d\mu \\
				&\lesssim  \norm{u}_{ L^m(-T,T; L^1(\R^{n}))} \int_y^{\infty} \mu^{1-2s-n-\ell} \, d\mu \\
				&\lesssim y^{2-2s-n-\ell}\norm{u}_{ L^m(-T,T; L^1(\R^{n}))},
			\end{split}
		\end{equation}
		for $\ell=0,1$ and $m\in [1,2]$. Here we use that $u\in \mathbf{H}^s(\R^{n+1})$ is supported in the compact set $\overline{(\Omega\cup W)_T}$ so that $\mathbf{H}^s(\R^{n+1})\subset L^m(-T,T; L^1(\R^{n}))$ for $m\in [1,2]$, and $\norm{u}_{ L^m(\R; L^1(\R^{n}))}$ is finite for $m\in [1,2]$.
		
		Next, we prove the local $L^2$ estimate for $v(t,x)=w(t,x,0)$ on the fixed bounded set $\Omega'$. Let
		\begin{equation}
			E(t,x):=\int_0^\infty \mu^{1-2s}\abs{\wt u(t,x,\mu)}\,d\mu.
		\end{equation}
		By Minkowski's integral inequality, \eqref{comp of decay x}, and Fubini's theorem, for $\ell=0,1$ and every $y>0$,
		\begin{equation}\label{estimate of w}
			\begin{split}
				\big\|\nabla_x^\ell w(\cdot,x,y)\big\|_{L^2(-T,T)}
				&\lesssim
				\int_{\R^n}
				\norm{u(\cdot,x+z)}_{L^2(\R)}
				\left(
				\int_y^\infty
				\frac{\mu}{(|z|^2+\mu^2)^{\frac{n+\ell}{2}+s}}\,d\mu
				\right)dz  \\
				&\lesssim
				\int_{\R^n}
				\frac{\norm{u(\cdot,x+z)}_{L^2(\R)}}{|z|^{n+\ell+2s-2}}\,dz .
			\end{split}
		\end{equation}
		Then $\abs{v(t,x)}\leq E(t,x)$ for a.e. $(t,x)$. By \eqref{estimate of w} with $\ell=0$, for every $y>0$,
		\begin{equation}
			\left\| w(\cdot,x,y)\right\|_{L^2(-T,T)}
			\lesssim
			\int_{B_{R_*}}
			\frac{\norm{u(\cdot,x+z)}_{L^2(\R)}}{\abs{z}^{n+2s-2}}\,dz,
		\end{equation}
		where $R_*>0$ is chosen so that $x+z\in \pi_x\supp(u)$ and $x\in\Omega'$ imply $z\in B_{R_*}$. Since $s\in(0,1)$, the kernel
		\begin{equation}
			K(z):=\chi_{B_{R_*}}(z)\abs{z}^{-(n+2s-2)}
		\end{equation}
		belongs to $L^1(\R^n)$. Hence, by Young's convolution inequality,
		\begin{equation}\label{local L2 estimate for E}
			\begin{split}
				\norm{E}_{L^2(\Omega'_T)}
				&\lesssim
				\big\|\norm{u(\cdot,\cdot)}_{L^2_t(\R)}\ast K\big\|_{L^2(\R^n)}
				\lesssim \norm{K}_{L^1(\R^n)}\norm{u}_{L^2(\R^{n+1})}.
			\end{split}
		\end{equation}
		Therefore $v=w(\cdot,\cdot,0)\in L^2(\Omega'_T)$ and
		\begin{equation}
			\norm{v}_{L^2(\Omega'_T)}
			\lesssim \norm{u}_{L^2(\R^{n+1})}
			\lesssim \norm{u}_{\mathbf H^s(\R^{n+1})}.
		\end{equation}
		
		\medskip
		
		{\it Step 3. Gradient $L^2$-estimate.}
		
		\medskip
		
		\noindent We prove that $v(t,x)\in L^2(-T,T; H^1(\Omega'))$ for any bounded open set $\Omega'\subset\R^n$. To this end, let us consider the cases $s\in (0,\frac{1}{2})$ and $s\in [\frac{1}{2},1)$.
		
		\smallskip
		
		\noindent {\it Step 3a. For $s\in (0,\frac{1}{2})$.} By using \eqref{estimate of w} as $\ell=1$, similar computations based on \eqref{estimate of w} and Young's convolution inequality yield that
		\begin{equation}
			\int_{\Omega'}\int_{-T}^T |\nabla v(t,x)|^2 \, dtdx \lesssim \big\| \chi_{B_R}(\cdot)\abs{\cdot}^{-(n+2s-1)} \big\|_{L^1(\R^n)}^2 \norm{u}_{L^2(\R^{n+1})}^2<\infty,
		\end{equation}
		since $\abs{\cdot}^{-(n+2s-1)}$ is locally integrable for $s\in (0,\frac{1}{2})$. However, for the case $s\in [\frac{1}{2},1)$, the arguments in the previous step are not enough since $\abs{\cdot}^{-(n+2s-1)}$ is not locally integrable. Therefore, we need a more detailed analysis.
		
		\smallskip
		
		\noindent {\it Step 3b. For $s\in [\frac{1}{2},1)$.} To show this case, let $\Omega'\subset\R^n$ be a fixed bounded Lipschitz open set. We choose $0<R_1<R$ so large that
		\begin{equation}\label{choice R1 R for g}
			\overline{\Omega'}\cup \pi_x\supp(u)\subset B_{R_1}\Subset B_R,
		\end{equation}
		where $\pi_x\supp(u)$ denotes the projection of $\supp(u)\subset\R_t\times\R^n_x$ onto the $x$-space. Let $g=g(x)\in C^\infty_c(B_R)$ be a time-independent function such that $0\leq g\leq1$ and $g=1$ in $B_{R_1}$. Let $\wt g=\wt g(x,y)$ be the extension of $g$, i.e., $\wt g$ is the solution to
		\begin{equation}
			\begin{cases}
				\nabla_{x,y}\cdot \LC y^{1-2s}\wt \sigma \nabla_{x,y} \wt g\RC =0&\text{ for }(x,y)\in \R^{n+1}_+, \\
				\wt g(x,0)=g(x)&\text{ for }x\in \R^{n}.
			\end{cases}
		\end{equation}
		By \cite{ST10}, it is known that $\wt g$ can be expressed by
		\begin{equation}\label{representation formula for wt g}
			\wt g(x,y)=c_s y^{2s}\int_{\R^{n}} \int_0^\infty  e^{-\frac{y^2}{4\tau }} p(x,z,\tau)g(z)\, \frac{d\tau }{\tau^{1+s}} dz,
		\end{equation}
		where $p(x,z,\tau)$ denotes the heat kernel satisfying \eqref{heat kernel} and $c_s$ is the constant given by \eqref{c_s}.
		
		Next, as shown in the proof of Lemma \ref{Lemma: decay estimate} and the estimate \eqref{comp of decay2}, with \eqref{representation formula for wt u} at hand, we can get
		\begin{equation}\label{more analy for s big}
			\begin{split}
				&\quad \, \bigg|   \nabla w(t,x,y)-\nabla\bigg(  \int_y^{\infty} \mu^{1-2s}\wt g(x,\mu)\, d\mu \bigg) u(t,x)  \bigg|  \\
				&=c_s\bigg|\int_y^{\infty} \mu\int_{\R^n}\int_0^\infty \frac{e^{-\frac{\mu^2}{4\tau}} \nabla_x p(x,z,\tau)}{\tau^{1+s}} \left[ u(t-\tau ,z) -g(z)u(t,x)\right] d\tau  dz   d\mu   \bigg| .
			\end{split}
		\end{equation}
		To proceed, for $y>0$, by Fubini's theorem, we find
		\begin{equation}\label{more analy for s big 2}
			\begin{split}
				&\quad \, \int_{-T}^T\bigg| \nabla w(\cdot,x,y)-\nabla\bigg(  \int_y^{\infty} \mu^{1-2s}\wt g(x,\mu)\, d\mu \bigg) u(\cdot,x) \bigg|^2 dt  \\
				&\lesssim \int_{-T}^T \left|  \int_{y}^\infty \mu \int_{\R^n}  \int_0^\infty  \frac{e^{-\frac{\mu^2}{4\tau}} \left| \nabla_x p(x,z,\tau)\right|}{\tau^{1+s}} \left| u(t-\tau ,z) -g( z)u(t,x)\right| d\tau  dz d\mu\right|^2  dt    \\
				&\approx \int_{-T}^T \left|  \int_{\R^n}  \int_0^\infty \LC \int_{y}^\infty \frac{\mu}{2\tau} e^{-\frac{\mu^2}{4\tau}} \, d\mu \RC \frac{\left|\nabla_x p(x,z,\tau)\right|}{\tau^{s}} \left| u(t-\tau ,z) -g(z)u(t,x)\right| d\tau  dz \right|^2  dt  \\
				&\leq \int_{-T}^T\left| \int_{\R^n}\int_0^\infty \frac{ \left| \nabla_x p(x,z,\tau)\right|}{\tau^{s}} \left| u(t-\tau ,z) -g(z)u(t,x)\right| d\tau  dz  \right|^2 dt\\
				&\lesssim \int_{-T}^T\left| \int_{B_R}\int_0^\infty \frac{ \left| \nabla_x p(x,z,\tau)\right|}{\tau^{s}} \left| u(t-\tau ,z) - u(t,x)\right| d\tau  dz  \right|^2 dt.
			\end{split}
		\end{equation}
		Here we used $\int_y^\infty \frac{\mu}{2\tau} e^{-\frac{\mu^2}{4\tau}}\,d\mu=e^{-\frac{y^2}{4\tau}}\leq1$, the support condition $\supp(u)\subset (B_{R_1})_T$, and $g=1$ on $B_{R_1}$.
		
		By using the triangle inequality $\abs{u(t-\tau,z)-u(t,x)}\leq \abs{u(t-\tau,z)-u(t,z)}+ \abs{u(t,z)-u(t,x)}$, we have
		\begin{equation}\label{more analy for s big 3}
			\int_{-T}^T\left| \nabla w(\cdot,x,y)-\nabla\LC  \int_y^{\infty} \mu^{1-2s}\wt g(x,\mu)\, d\mu \RC u(\cdot,x) \right|^2 dt \lesssim J_1 + J_2,
		\end{equation}
		where
		\begin{equation}
			\begin{split}
				J_1&:=\int_{-T}^T\left| \int_{B_R}\int_0^\infty \frac{ \left| \nabla_x p(x,z,\tau)\right|}{\tau^{s}} \left| u(t-\tau ,z) - u(t,z)\right| d\tau  dz  \right|^2 dt, \\
				J_2&:=\int_{-T}^T\left| \int_{B_R}\int_0^\infty \frac{ \left| \nabla_x p(x,z,\tau)\right|}{\tau^{s}} \left| u(t ,z) - u(t,x)\right| d\tau  dz  \right|^2 dt.
			\end{split}
		\end{equation}
		Our remaining task is to estimate $J_1$ and $J_2$.
		
		\smallskip
		
		{\it Step 3b-1. Estimate for $J_1$.} We estimate $J_1$ by using the smoothing of the heat semigroup rather than integrating out the heat kernel into a singular spatial kernel. This avoids the non-integrable singularity at $x=z$.
		
		Let
		\begin{equation}
			F_\tau(t,z):=u(t-\tau,z)-u(t,z),
		\end{equation}
		and define
		\begin{equation}
			\mathcal K_\tau F_\tau(t,x):=\int_{B_R}\left|\nabla_xp(x,z,\tau)\right|\left|F_\tau(t,z)\right|\,dz.
		\end{equation}
		Then
		\begin{equation}
			\bigg(\int_{\Omega'}J_1\,dx\bigg)^{1/2}
			\leq \int_0^\infty \tau^{-s}\left\|\mathcal K_\tau F_\tau\right\|_{L^2(\R;L^2(\Omega'))}\,d\tau.
		\end{equation}
		We split the right-hand side into the regions $\tau\in(0,1)$ and $\tau\in(1,\infty)$.
		
		For $\tau\in(0,1)$, by \eqref{heat kernel estimate} and Young's convolution inequality,
		\begin{equation}\label{small tau gradient semigroup}
			\left\|\mathcal K_\tau F_\tau\right\|_{L^2(\R;L^2(\Omega'))}
			\leq \left\|\mathcal K_\tau F_\tau\right\|_{L^2(\R^{n+1})}
			\lesssim \tau^{-1/2}\left\|u(\cdot-\tau,\cdot)-u(\cdot,\cdot)\right\|_{L^2(\R^{n+1})}.
		\end{equation}
		Therefore, H\"older's inequality in the $\tau$-variable gives
		\begin{equation}\label{J1 small tau}
			\begin{split}
				\int_0^1 \tau^{-s}\left\|\mathcal K_\tau F_\tau\right\|_{L^2(\R;L^2(\Omega'))}\,d\tau
				&\lesssim \int_0^1 \tau^{-s-\frac12}\left\|u(\cdot-\tau,\cdot)-u(\cdot,\cdot)\right\|_{L^2(\R^{n+1})}\,d\tau\\
				&\leq \bigg(\int_0^1 \tau^{-s}\,d\tau\bigg)^{1/2}
				\bigg(\int_0^1\frac{\|u(\cdot-\tau,\cdot)-u(\cdot,\cdot)\|_{L^2(\R^{n+1})}^2}{\tau^{1+s}}\,d\tau\bigg)^{1/2}\\
				&\lesssim \norm{u}_{H^{s/2}(\R;L^2(\R^n))}
				\lesssim \norm{u}_{\mathbf H^s(\R^{n+1})}.
			\end{split}
		\end{equation}
		For $\tau\in(1,\infty)$, we use the $L^1_x\to L^2_x$ estimate coming from \eqref{heat kernel estimate}:
		\begin{equation}\label{large tau gradient semigroup}
			\left\|\mathcal K_\tau F_\tau\right\|_{L^2(\R;L^2(\Omega'))}
			\leq \left\|\mathcal K_\tau F_\tau\right\|_{L^2(\R;L^2(\R^n))}
			\lesssim \tau^{-\frac n4-\frac12}\left\|u(\cdot-\tau,\cdot)-u(\cdot,\cdot)\right\|_{L^2(\R;L^1(\R^n))}.
		\end{equation}
		Since $u$ is supported in $[-T,T]\times B_R$, H\"older's inequality in the $x$-variable yields
		\begin{equation}
			\left\|u(\cdot-\tau,\cdot)-u(\cdot,\cdot)\right\|_{L^2(\R;L^1(\R^n))}
			\lesssim \norm{u}_{L^2(\R^{n+1})}.
		\end{equation}
		Hence, for $n\geq2$ and $s\in(0,1)$,
		\begin{equation}\label{J1 large tau}
			\begin{split}
				\int_1^\infty \tau^{-s}\left\|\mathcal K_\tau F_\tau\right\|_{L^2(\R;L^2(\Omega'))}\,d\tau
				&\lesssim \norm{u}_{L^2(\R^{n+1})}\int_1^\infty \tau^{-s-\frac n4-\frac12}\,d\tau\\
				&\lesssim \norm{u}_{L^2(\R^{n+1})}.
			\end{split}
		\end{equation}
		Combining \eqref{J1 small tau} and \eqref{J1 large tau}, we obtain
		\begin{equation}\label{J_1 estimate}
			\bigg(\int_{\Omega'}J_1\,dx\bigg)^{1/2}\lesssim\norm{u}_{\mathbf H^s(\R^{n+1})}.
		\end{equation}
		
		\smallskip
		
		{\it Step 3b-2. Estimate for $J_2$.} By a straightforward calculation, we can have that
		\begin{equation}
			\begin{split}
				&\quad \, \int_{-T}^T\left| \int_{B_R}\int_0^\infty \frac{\left|\nabla_xp(x,z,\tau)\right|}{\tau^s}\left|u(t,z)-u(t,x)\right|\,d\tau dz\right|^2dt\\
				&\lesssim \int_{-T}^T\left|\int_{B_R}\LC\int_0^\infty e^{-c\frac{\abs{x-z}^2}{\tau}}\tau^{-(\frac{n+1}{2}+s)}\,d\tau\RC\left|u(t,z)-u(t,x)\right|\,dz\right|^2dt\\
				&\lesssim \int_{-T}^T\left|\int_{B_R}\frac{\left|u(t,z)-u(t,x)\right|}{\abs{x-z}^{n+2s-1}}\,dz\right|^2dt\\
				&\leq \int_{-T}^T\LC\int_{B_R}\frac{\left|u(t,z)-u(t,x)\right|^2}{\abs{x-z}^{n+2s}}\,dz\RC
				\LC\int_{B_R}\frac{1}{\abs{x-z}^{n+2s-2}}\,dz\RC dt\\
				&\lesssim \int_{-T}^T\int_{B_R}\frac{\left|u(t,z)-u(t,x)\right|^2}{\abs{x-z}^{n+2s}}\,dzdt.
			\end{split}
		\end{equation}
		Therefore,
		\begin{equation}\label{J_2 estimate}
			\begin{split}
				\LC \int_{\Omega'}J_2\,dx\RC^{1/2}
				&\lesssim \bigg(\int_{-T}^T\int_{\Omega'}\int_{B_R}\frac{\left|u(t,z)-u(t,x)\right|^2}{\abs{x-z}^{n+2s}}\,dzdxdt\bigg)^{1/2}\\
				&\leq \bigg(\int_{-T}^T\norm{u(t,\cdot)}_{H^s(\R^n)}^2\,dt\bigg)^{1/2}\\
				&\lesssim \norm{u}_{\mathbf{H}^s(\R^{n+1})},
			\end{split}
		\end{equation}
		where we used $H^s(\R^n)=W^{s,2}(\R^n)$, which is characterized in \cite[Section 3]{ML-strongly-elliptic-systems} for instance. Therefore, combined with \eqref{more analy for s big 2}, \eqref{J_1 estimate} and \eqref{J_2 estimate}, we can conclude that
		\begin{equation}\label{J_1+J_2 estimate}
			\left\|  \LC  \nabla w(t,x,y)-\nabla\LC  \int_y^{\infty} \mu^{1-2s}\wt g(x,\mu)\, d\mu \RC u(t,x) \RC  \right\|_{L^2(\Omega'_T)}  \lesssim \norm{u}_{\mathbf{H}^s(\R^{n+1})}<\infty,
		\end{equation}
		for any bounded open set $\Omega'\subset \R^n$ and for any $y>0$. Note that the upper bound of \eqref{J_1+J_2 estimate} is independent of $y>0$.
		
		The goal is to estimate $\left\| \nabla w(\cdot,\cdot,0) \right\|_{L^2(\Omega'_T)}$. To this end, we observe that
		\begin{equation}
			\begin{split}
				\left\| \nabla\LC  \int_0^{\infty} \mu^{1-2s}\wt g(x,\mu)\, d\mu  \RC u(t,x) \right\|_{L^2(\Omega'_T)} \lesssim \norm{u}_{L^2(\Omega'_T)}\left\| \nabla\LC  \int_0^{\infty} \mu^{1-2s}\wt g(x,\mu)\, d\mu \RC \right\|_{L^\infty(\Omega')}.
			\end{split}
		\end{equation}
		Note that the function $\int_0^{\infty} \mu^{1-2s}\wt g(x,\mu)\, d\mu$ is time-independent, and
		\begin{equation}\label{bound for g}
			\left\| \nabla\LC  \int_0^{\infty} \mu^{1-2s}\wt g(x,\mu )\, d\mu \RC \right\|_{L^\infty(\Omega')}<\infty,
		\end{equation}
		as shown in the proof of \cite[Proposition 6.1]{CGRU2023reduction}. For the sake of self-containedness, we prove \eqref{bound for g} for $n\geq2$ in the next step. Hence, with these estimates at hand, we can obtain that
		\begin{equation}\label{gradient estimate of w(t,x,0)}
			\begin{split}
				\left\| \nabla w(\cdot,\cdot,0) \right\|_{L^2(\Omega'_T)}
				&\leq  \liminf_{y\to0}	\left\|  \LC  \nabla w(t,x,y)-\nabla\LC  \int_y^{\infty} \mu^{1-2s}\wt g(x,\mu)\, d\mu \RC u(t,x) \RC  \right\|_{L^2(\Omega'_T)} \\
				&\quad \,+ \left\| \nabla\LC  \int_0^{\infty} \mu^{1-2s}\wt g(x,\mu)\, d\mu \RC u(t,x) \right\|_{L^2(\Omega'_T)}  \\
				&\lesssim \norm{u}_{\mathbf{H}^s(\R^{n+1})}<\infty,
			\end{split}
		\end{equation}
		which proves the desired local gradient estimate.
		
			\medskip
		
		{\it Step 4. Auxiliary estimate.}
		
		\medskip
		
		\noindent We prove \eqref{bound for g} for $s\in[\frac12,1)$ and $n\ge2$ by a direct heat-semigroup estimate. Let
		\[
		F(x):=\int_0^\infty \tau^{-s}P_\tau g(x)\,d\tau,
		\quad
		P_\tau g(x):=\int_{\R^n}p(x,z,\tau)g(z)\,dz.
		\]
		It is enough to show that
		\[
		\sup_{x\in\Omega'}\int_0^\infty \tau^{-s}
		|\nabla_xP_\tau g(x)|\,d\tau<\infty.
		\]
		Choose $0<R_1<R$ as in \eqref{choice R1 R for g}, so that
		$\overline{\Omega'}\Subset B_{R_1}$ and $g=1$ in $B_{R_1}$. Let
		\[
		\delta:=\operatorname{dist}\big(\Omega',\R^n\setminus B_{R_1}\big)>0.
		\]
		For $0<\tau\le1$, since $P_\tau1=1$, we have
		\[
		\nabla_xP_\tau g(x)
		=
		-\nabla_xP_\tau(1-g)(x).
		\]
		Using the gradient heat-kernel estimate \eqref{heat kernel estimate} and the fact that
		$1-g$ is supported in the region whose distance from $\Omega'$ is at least $\delta$, we obtain
		\[
		\begin{split}
			|\nabla_xP_\tau g(x)|
			\le
			C\tau^{-\frac{n+1}{2}}
			\int_{\{|x-z|\ge\delta\}}e^{-c\frac{|x-z|^2}{\tau}}\,dz \le C_\delta \tau^{-\frac12}e^{-c_\delta/\tau},
			\quad x\in\Omega',\quad 0<\tau\le1.
		\end{split}
		\]
		Hence
		\[
		\int_0^1\tau^{-s}|\nabla_xP_\tau g(x)|\,d\tau
		\le
		C_\delta\int_0^1\tau^{-s-\frac12}e^{-c_\delta/\tau}\,d\tau<\infty.
		\]
		For $\tau\ge1$, the compact support of $g$ and \eqref{heat kernel estimate} give
		\[
		|\nabla_xP_\tau g(x)|
		\le
		C\tau^{-\frac{n+1}{2}}\|g\|_{L^1(\R^n)}.
		\]
		Therefore,
		\[
		\int_1^\infty\tau^{-s}|\nabla_xP_\tau g(x)|\,d\tau
		\le
		C\|g\|_{L^1(\R^n)}
		\int_1^\infty\tau^{-s-\frac{n+1}{2}}\,d\tau<\infty
		\]
		for every $n\ge2$ and $s>0$. Combining the two estimates yields
		\[
		\|\nabla F\|_{L^\infty(\Omega')}<\infty.
		\]
		Since
		\[
		\nabla\left(\int_0^\infty \mu^{1-2s}\widetilde g(x,\mu)\,d\mu\right)
		\]
		differs from $\nabla F$ only by the dimensional constant coming from the extension formula, this proves \eqref{bound for g}.
		
		\medskip
		
		{\it Step 5. Conclusion.}
		
		\medskip
		
		\noindent We now combine the preceding estimates. Step 2 gives the $L^2$ bound for $v=w(\cdot,\cdot,0)$ on every bounded open set. Step 3a gives the local gradient estimate for $s\in(0,\frac12)$, while Step 3b gives the local gradient estimate for $s\in[\frac12,1)$. Hence, for every bounded Lipschitz open set $\Omega'\subset\R^n$, we have
		\begin{equation}
			\left\| v \right\|_{L^2(-T,T;H^1(\Omega'))}\lesssim\norm{u}_{\mathbf H^s(\R^{n+1})}.
		\end{equation}
		Taking $\Omega'=\Omega$, we obtain \eqref{estimate u and v}. This proves the assertion of Proposition \ref{Prop: bound w}.
	\end{proof}

	\begin{remark}[Exterior estimates]\label{Remark: exterior estimates for v}
		Let $B_R$ be a ball such that $\pi_x\supp(u)\subset B_R$. The same computation as in \eqref{estimate of w} gives, for $x\in\R^n\setminus B_{2R}$ and $\ell=0,1$,
		\begin{equation}\label{exterior estimate v conditional}
			\left\| \nabla^\ell v(\cdot,x)\right\|_{L^2(-T,T)}
			\lesssim \int_{B_R}\frac{\norm{u(\cdot,z)}_{L^2(\R)}}{\abs{x-z}^{n+\ell+2s-2}}\,dz.
		\end{equation}
		In particular, for $\ell=1$, one has
		\begin{equation}
			\left\|\nabla v\right\|_{L^2((\R^n\setminus B_{2R})_T)}
			\lesssim \norm{u}_{L^2(\R^{n+1})},
		\end{equation}
		for every $n\geq2$ and $s\in(0,1)$. For the zeroth order term, \eqref{exterior estimate v conditional} yields the global exterior $L^2$ bound whenever $n+2s-2>\frac n2$. In the present article, only the local estimate \eqref{estimate u and v local} is needed for the construction of the local Cauchy data on $\Omega_T$.
	\end{remark}

	\section{The local equation for the averaged solution}\label{sec: key equation}
	
	With the regularity of the averaging transform established in Section \ref{sec: regularity}, we now prove that the averaged function satisfies the local parabolic equation. This is the point where the formal computation in Section \ref{sec: duality} becomes a weak identity. Throughout this section, the expression $\lim_{y\to0}y^{1-2s}\p_y\wt u$ is understood in the sense of $\mathbf H^{-s}(\R^{n+1})$, as in Proposition \ref{Prop: extension}.
	
	\begin{lemma}[Weak identity for the averaged function]\label{Lemma: weak formulation}
		Let $s\in(0,1)$ and $n\geq2$. Let $\wt\sigma\in C^2(\R^n;\R^{(n+1)\times(n+1)})$ be of the form \eqref{tilde sigma}. Let $u\in\mathbf H^s(\R^{n+1})$ have compact support, and let $\wt u\in\mathcal L^{1,2}(\R^{n+2}_+;y^{1-2s}dtdxdy)$ be its parabolic extension solving \eqref{degen para PDE}. Assume that for some bounded Lipschitz set $\Omega'\subset\R^n$, the function
		\[
		v(t,x)=\int_0^\infty y^{1-2s}\wt u(t,x,y)\,dy
		\]
		belongs to $L^2(-T,T;H^1(\Omega'))$. Then, for every $\varphi\in C_c^\infty(\Omega'_T)$,
		\begin{equation}\label{weak form}
			\int_{\Omega'_T}\LC v\p_t\varphi-\sigma(x)\nabla v\cdot\nabla\varphi\RC\,dtdx=\big\langle \lim_{y\to0}y^{1-2s}\p_y\wt u,\varphi\big\rangle_{\mathbf H^{-s}(\R^{n+1}),\mathbf H^s(\R^{n+1})}.
		\end{equation}
		Equivalently, in the sense of distributions in $\Omega'_T$,
		\begin{equation}\label{distributional key identity}
			\LC\p_t-\nabla\cdot\sigma\nabla\RC v=-\lim_{y\to0}y^{1-2s}\p_y\wt u=-d_s^{-1}H_\sigma^s u.
		\end{equation}
	\end{lemma}
	
	\begin{remark}
		The right-hand side in \eqref{weak form} is understood as the $\mathbf H^{-s}(\R^{n+1})$-$\mathbf H^s(\R^{n+1})$ duality pairing. Notice that our weak formulation convention is
		\begin{equation}
			\int_{\Omega'_T}\LC v\p_t\varphi-\sigma\nabla v\cdot\nabla\varphi\RC\,dtdx
			=-\left\langle H_\sigma v,\varphi\right\rangle.
		\end{equation}
		Hence \eqref{weak form} is equivalent to
		\begin{equation}
			H_\sigma v=-\lim_{y\to0}y^{1-2s}\p_y\wt u=-d_s^{-1}H_\sigma^s u
		\end{equation}
		in the sense of distributions. In the application to the nonlocal equation, $H_\sigma^su=0$ in $\Omega_T$, and therefore $H_\sigma v=0$ in $\Omega_T$.
	\end{remark}
	
	\begin{proof}
		We extend $\varphi$ by zero outside $\Omega'_T$. Let $\eta\in C^\infty([0,\infty))$ satisfy $0\leq\eta\leq1$, $\eta(y)=1$ for $0\leq y\leq1$, and $\eta(y)=0$ for $y\geq2$. Set $\eta_k(y)=\eta(y/k)$. Define
		\[
		v_k(t,x):=\int_0^\infty y^{1-2s}\wt u(t,x,y)\eta_k(y)\,dy.
		\]
		We first note that $v_k\to v$ in $L^2(\Omega'_T)$ and $\nabla v_k\to\nabla v$ in $L^2(\Omega'_T)$. Indeed, by Lemma \ref{Lemma: decay estimate} and the compact support of $u$,
		\[
		\norm{v-v_k}_{L^2(\Omega'_T)}\lesssim\int_k^\infty y^{1-2s}\norm{\wt u(\cdot,\cdot,y)}_{L^2(-T,T;L^\infty(\R^n))}\,dy\lesssim\int_k^\infty y^{1-2s-n}\,dy\cdot\norm{u}_{L^2(\R;L^1(\R^n))}\to0,
		\]
		and similarly,
		\[
		\norm{\nabla v-\nabla v_k}_{L^2(\Omega'_T)}\lesssim\int_k^\infty y^{1-2s}\norm{\nabla\wt u(\cdot,\cdot,y)}_{L^2(-T,T;L^\infty(\R^n))}\,dy\lesssim\int_k^\infty y^{-2s-n}\,dy\cdot\norm{u}_{L^2(\R;L^1(\R^n))}\to0.
		\]
		Therefore,
		\begin{equation}\label{limit weak identity first}
			\int_{\Omega'_T}\LC v\p_t\varphi-\sigma\nabla v\cdot\nabla\varphi\RC\,dtdx=\lim_{k\to\infty}\int_{\R^{n+1}}\LC v_k\p_t\varphi-\sigma\nabla v_k\cdot\nabla\varphi\RC\,dtdx.
		\end{equation}
		
		By the definition of $v_k$ and Fubini's theorem,
		\begin{equation}\label{vk expression}
			\int_{\R^{n+1}}\LC v_k\p_t\varphi-\sigma\nabla v_k\cdot\nabla\varphi\RC\,dtdx=\int_0^\infty\int_{\R^{n+1}}y^{1-2s}\LC \wt u\eta_k\p_t\varphi-\eta_k\sigma\nabla\wt u\cdot\nabla\varphi\RC\,dtdxdy.
		\end{equation}
		We now test the extension equation \eqref{degen para PDE} with $\varphi(t,x)\eta_k(y)$. Using the weak formulation of the extension problem, we get
		\begin{equation}\label{extension tested}
			\int_0^\infty\int_{\R^{n+1}}y^{1-2s}\LC \wt u\eta_k\p_t\varphi-\eta_k\sigma\nabla\wt u\cdot\nabla\varphi-\p_y\wt u\,\varphi\,\p_y\eta_k\RC\,dtdxdy=\big\langle\lim_{y\to0}y^{1-2s}\p_y\wt u,\varphi\big\rangle.
		\end{equation}
		Combining \eqref{vk expression} and \eqref{extension tested}, it remains to show that the cutoff error
		\[
		E_k:=\int_0^\infty\int_{\R^{n+1}}y^{1-2s}\p_y\wt u(t,x,y)\varphi(t,x)\p_y\eta_k(y)\,dtdxdy
		\]
		vanishes as $k\to\infty$. Since $\p_y\eta_k$ is supported in $[k,2k]$ and $|\p_y\eta_k|\lesssim k^{-1}$, Lemma \ref{Lemma: decay estimate} gives
		\[
		\begin{split}
			\abs{E_k}&\lesssim\frac{1}{k}\norm{\varphi}_{L^\infty(\R^{n+1})}\int_k^{2k}y^{1-2s}\norm{\p_y\wt u(\cdot,\cdot,y)}_{L^1(\R;L^1(\R^n))}\,dy\\
			&\lesssim\frac{1}{k}\int_k^{2k}y^{-2s}\,dy\cdot\norm{u}_{L^1(\R^{n+1})}\lesssim k^{-2s}\norm{u}_{L^1(\R^{n+1})}.
		\end{split}
		\]
		Hence $E_k\to0$. Letting $k\to\infty$ in \eqref{extension tested} and using \eqref{limit weak identity first}, we obtain \eqref{weak form}. Finally, \eqref{distributional key identity} follows from \eqref{extension DN relation}.
	\end{proof}
	
	We now apply the weak identity to solutions of the nonlocal parabolic equation.
	
	\begin{theorem}[Key equation]\label{Thm: key equation}
		Let $s\in(0,1)$ and $n\geq2$. Let $\Omega,W\subset\R^n$ be nonempty bounded open sets with Lipschitz boundaries such that $\overline{\Omega}\cap\overline W=\emptyset$. Let $\wt\sigma$ be of the form \eqref{tilde sigma}, where $\sigma\in C^2(\R^n;\R^{n\times n})$ satisfies \eqref{ellipticity condition} and $|\sigma|+|\nabla\sigma|\leq M$ in $\R^n$ for some $M>0$. Let $u_f\in\mathbf H^s(\R^{n+1})$ be the unique solution of \eqref{equ nonlocal para} with exterior data $f\in\widetilde{\mathbf H}^s(W_T)$, and let $\wt u_f\in\mathcal L^{1,2}(\R^{n+2}_+;y^{1-2s}dtdxdy)$ be its extension. Define
		\[
		v_f(t,x):=\int_0^\infty y^{1-2s}\wt u_f(t,x,y)\,dy.
		\]
		Then $v_f\in L^2(-T,T;H^1(\Omega))$, $\p_t v_f\in L^2(-T,T;H^{-1}(\Omega))$, and $v_f$ is a weak solution of
		\begin{equation}\label{local equation for vf}
			\LC\p_t-\nabla\cdot\sigma\nabla\RC v_f=0\quad\text{in }\Omega_T.
		\end{equation}
		Equivalently,
		\begin{equation}\label{key weak equation}
			\int_{\Omega_T}\LC v_f\p_t\varphi-\sigma\nabla v_f\cdot\nabla\varphi\RC\,dtdx=0
		\end{equation}
		for every $\varphi\in C_c^\infty(\Omega_T)$. Moreover, $v_f(t,x)=0$ for $t\leq -T$.
	\end{theorem}
	
	\begin{proof}
		We first assume that $f\in C_c^\infty(W_T)$. Then the corresponding solution $u_f$ has the compact support property described in Section \ref{subsection: well-posed}. By Proposition \ref{Prop: bound w}, the averaged function $v_f$ belongs to $L^2(-T,T;H^1(\Omega))$. Since $u_f$ solves $H_\sigma^s u_f=0$ in $\Omega_T$, Proposition \ref{Prop: extension} gives $\lim_{y\to0}y^{1-2s}\p_y\wt u_f=0
		\quad\text{in }\Omega_T$ in the sense of distributions. Applying Lemma \ref{Lemma: weak formulation} with $\Omega'=\Omega$, we obtain \eqref{key weak equation}. Hence $v_f$ solves \eqref{local equation for vf} in $\Omega_T$.
		
		For general $f\in\widetilde{\mathbf H}^s(W_T)$, choose $f_m\in C_c^\infty(W_T)$ such that
		$f_m\to f$ in $\widetilde{\mathbf H}^s(W_T)$. By the well-posedness estimate
		\eqref{equ nonlocal para estimate} and Proposition \ref{Prop: bound w}, the corresponding averages
		$v_{f_m}$ converge to $v_f$ in $L^2(-T,T;H^1(\Omega))$. Passing to the limit in the weak formulation gives the assertion.
		
		It remains to record the time regularity. From \eqref{key weak equation}, we have $\p_t v_f=\nabla\cdot\sigma\nabla v_f$ in $L^2(-T,T;H^{-1}(\Omega))$. Indeed,
		\[
		\norm{\p_t v_f}_{L^2(-T,T;H^{-1}(\Omega))}\lesssim\norm{\sigma\nabla v_f}_{L^2(\Omega_T)}\lesssim\norm{v_f}_{L^2(-T,T;H^1(\Omega))}.
		\]
		Finally, the representation formula for the extension and the past-time support of $u_f$ implies that $\wt u_f(t,x,y)=0$ for $t\leq -T$. Therefore, $v_f(t,x)=0$ for $t\leq -T$, and in particular, the local solution has zero initial data at $t=-T$.
	\end{proof}

	\section{Density approach}\label{sec: density}
	
	In this section, we prove that the Cauchy data that are generated from the nonlocal parabolic equation form a dense set in the set of the Cauchy data for the local parabolic equation. The main tool is the averaged function defined by \eqref{v=w0}.
	
	\begin{proposition}\label{Prop: density}
		Given $s\in (0,1)$ and $n\geq 2$, let $\wt \sigma \in C^2(\R^n;\R^{(n+1)\times (n+1)})$ be of the form \eqref{tilde sigma}, where $\sigma\in C^2(\R^n;\R^{n\times n})$ satisfies \eqref{ellipticity condition} and $|\sigma|+|\nabla\sigma|\leq M$ in $\R^n$ for some constant $M>0$. Suppose that $\Omega\subset\R^n$ is a bounded Lipschitz domain with connected exterior $\Omega_e$, and that $W\subset\Omega_e$ is a nonempty bounded Lipschitz open set with $\overline{\Omega}\cap\overline W=\emptyset$. Let $\wt u_f=\wt u_f(t,x,y)\in\mathcal L^{1,2}(\R^{n+2}_+;y^{1-2s}dtdxdy)$ be the weak solution of
		\begin{equation}\label{density extension problem}
			\begin{cases}
				y^{1-2s}\p_t \wt u_f-\nabla_{x,y}\cdot\LC y^{1-2s}\wt \sigma(x)\nabla_{x,y}\wt u_f\RC=0 &\text{ in }\R^{n+2}_+,\\
				\wt u_f(t,x,0)=f &\text{ on }(\Omega_e)_T,\\
				\displaystyle\lim_{y\to 0}y^{1-2s}\p_y \wt u_f=0 &\text{ in }\Omega_T,\\
				\wt u_f(t,x,y)=0 &\text{ for }t\leq -T,\ (x,y)\in\R^{n+1}_+,
			\end{cases}
		\end{equation}
		where $f\in C^\infty_c(W_T)$ is extended by zero to $(\Omega_e)_T\setminus W_T$. Define
		\begin{equation}
			\begin{split}
				\mathcal{V}&:=\left\{ v_f(t,x):=\int_0^\infty y^{1-2s}\wt u_f(t,x,y)\, dy:\ f\in C^\infty_c(W_T) \right\}, \\
				\mathcal{V}'&:=\left\{v|_{(\p \Omega)_T}:\, v\in \mathcal{V} \right\}.
			\end{split}
		\end{equation}
		Let $\mathcal D_\sigma$ and $\mathcal T_\sigma$ be defined by \eqref{local solution space D} and \eqref{local energy trace space}. Then 
		\begin{equation}\label{density interior conclusion}
			\overline{\mathcal V}^{\,L^2(-T,T;H^1(\Omega))}
			=
			\mathcal D_\sigma,
		\end{equation}
		and
		\begin{equation}\label{density boundary conclusion}
			\overline{\mathcal V'}^{\,L^2(-T,T;H^{1/2}(\partial\Omega))}
			=
			\mathcal T_\sigma.
		\end{equation}
		In particular, if the local Dirichlet problem for $H_\sigma$ is well posed for every
		$g\in L^2(-T,T;H^{1/2}(\partial\Omega))$, then
		\begin{equation}
			\mathcal T_\sigma=L^2(-T,T;H^{1/2}(\partial\Omega)),
		\end{equation}
		and hence
		\begin{equation}
			\overline{\mathcal V'}^{\,L^2(-T,T;H^{1/2}(\partial\Omega))}
			=
			L^2(-T,T;H^{1/2}(\partial\Omega)).
		\end{equation}
	\end{proposition}
	
	\begin{remark}
		Proposition \ref{Prop: density} is a fixed-coefficient density statement. It does not require $\sigma=\mathrm{Id}$ in $\Omega_e$. The estimates use only the ellipticity of $\sigma$ and the uniform bounds on $\sigma$ and $\nabla\sigma$. In particular, if $\eta_k(x)=\eta(x/k)$, then $|\nabla\eta_k|\lesssim k^{-1}$ and $|D^2\eta_k|\lesssim k^{-2}$, so that
		\begin{equation}\label{cutoff variable coefficient estimate}
			|\nabla\cdot\sigma\nabla\eta_k|\leq |\sigma||D^2\eta_k|+|\nabla\sigma||\nabla\eta_k|\lesssim k^{-1}.
		\end{equation}
		The exterior agreement assumption \eqref{exterior agreement assumption} is used later only in the comparison step, where two coefficients are compared, and the equality of exterior Cauchy data is propagated from $W_T$ to $(\partial\Omega)_T$.
	\end{remark}
	
	\subsection{Idea of the density argument}
	
	We briefly describe the mechanism behind the density result. Suppose, for the moment, that $s=1/2$ and that the coefficients are constant. To prove density, one argues by the Hahn-Banach theorem. If a functional $\psi\in L^2(-T,T;\widetilde H^{-1}(\Omega))$ annihilates all averaged solutions $v_f$, one solves a backward adjoint problem in the extension space with source $\psi$, zero Dirichlet data on $(\Omega_e)_T\times\{0\}$, and zero weighted Neumann data on $\Omega_T\times\{0\}$. Integration by parts then shows that the weighted Neumann trace of the adjoint solution vanishes on $W_T\times\{0\}$. Since its Dirichlet trace also vanishes there, boundary unique continuation implies that the adjoint solution vanishes in the exterior cylinder. Testing against local parabolic solutions and using vertical cutoffs then shows that $\psi$ annihilates every local solution. This gives the interior density, and the boundary density follows by the trace theorem.
	
	The proof below implements this argument for all $s\in(0,1)$ and for variable coefficients. The main additional points are the weighted energy solvability of the adjoint problem, the cutoff estimates in the extension variable, and the use of boundary unique continuation for the degenerate parabolic extension equation.

	\subsection{A rigorous proof for the density result}
	
	The following statement is used only in exterior regions where the coefficient is fixed. It follows from the strong unique continuation property for $H_\sigma^s$ in \cite[Theorem 1.5]{BS2024calderon}, together with the extension DN relation.
	
	\begin{lemma}[Boundary unique continuation for the parabolic extension]\label{Lemma: boundary UCP extension}
		Let $I\subset\R$ be an open interval, let $\mathcal O\subset\R^n$ be a connected open set, and let
		$\mathcal W\subset\mathcal O$ be a nonempty open subset. Let $\wt\sigma$ be of the form \eqref{tilde sigma}, where $\sigma$ satisfies \eqref{ellipticity condition} and the regularity assumptions stated above. Let $U\in \mathcal L^{1,2}_{\loc}(I\times\mathcal O\times(0,\infty);y^{1-2s}dtdxdy)$ solve
		\[
		y^{1-2s}\p_t U-\nabla_{x,y}\cdot
		\left(y^{1-2s}\wt\sigma(x)\nabla_{x,y}U\right)=0
		\quad\text{in }I\times\mathcal O\times(0,\infty).
		\]
		Assume that
		\[
		U(t,x,0)=\lim_{y\to0}y^{1-2s}\p_yU(t,x,y)=0
		\]
		on $I\times\mathcal W\times\{0\}$ in the trace sense. Then
		\[
		U\equiv0\quad\text{in }I\times\mathcal O\times(0,\infty).
		\]
	\end{lemma}

	\begin{lemma}[Solvability]\label{Lemma: weak solvable}
		Given $s\in (0,1)$ and $n\geq 2$, let $\wt \sigma \in C^2(\R^n;\R^{(n+1)\times (n+1)})$ be of the form \eqref{tilde sigma}, where $\sigma$ satisfies \eqref{ellipticity condition} and $|\sigma|+|\nabla\sigma|\leq M$ in $\R^n$. Let $\Omega\subset \R^n$ be a bounded open set with Lipschitz boundary $\p \Omega$. Given $\psi \in L^2(-T,T;\wt H^{-1}(\Omega))$, consider the problem
		\begin{equation}\label{adjoint problem generalized}
			\begin{cases}
				y^{1-2s}\p_t w - \nabla_{x,y} \cdot \LC y^{1-2s} \wt \sigma \nabla_{x,y} w\RC= y^{1-2s}\psi &\text{ in }\R^n_T \times(0,\infty), \\
				w=0 &\text{ in } \LC \Omega_e \RC_T \times \{0\}, \\
				\displaystyle \lim_{y\to 0}y^{1-2s}\p_y w =0 &\text{ in }\Omega_T \times \{0\},\\
				w(t,x,y)=0 &\text{ for }(x,y)\in \R_+^{n+1}, t\leq -T.
			\end{cases}
		\end{equation}
		Then there exists a distributional solution $w\in\mathcal L^{1,2}_{\mathrm{loc}}(\R^{n+2}_+;y^{1-2s}dtdxdy)$ of \eqref{adjoint problem generalized} with the following property. There exist functions $u_1,\mathcal E_su_1,u_2$ such that
		\begin{equation}\label{adjoint decomposition}
			w=\wt u_1-\mathcal E_su_1+u_2,\quad \wt u_1(t,x,y):=u_1(t,x),
		\end{equation}
		where
		\begin{equation}
			u_1\in L^2(\R;H^1(\R^n))\cap L^\infty(\R;L^2(\R^n)),\quad
			\p_tu_1\in L^2(\R;H^{-1}(\R^n)),
		\end{equation}
		and
		\begin{equation}\label{energy estimate u1}
			\sup_{\tau\in[-T,T]}\norm{u_1(\tau,\cdot)}_{L^2(\R^n)}^2
			+\norm{u_1}_{L^2(-T,T;H^1(\R^n))}^2
			+\norm{\p_tu_1}_{L^2(-T,T;H^{-1}(\R^n))}^2
			\lesssim \norm{\psi}_{L^2(-T,T;\wt H^{-1}(\Omega))}^2.
		\end{equation}
		We also have
		\begin{equation}
			\mathcal E_su_1,u_2\in\mathcal L^{1,2}(\R^{n+2}_+;y^{1-2s}dtdxdy).
		\end{equation}
		Moreover, for every $\varphi\in\mathscr C_{c,0}$, there holds
		\begin{equation}\label{weak solvable identity}
			\begin{split}
				&\quad \, \int_{\R^{n+2}_+} y^{1-2s}\left( -w \p_t \varphi+\wt\sigma\nabla_{x,y}w\cdot\nabla_{x,y}\varphi\right)\,dtdxdy\\
				&=\left\langle \psi,\int_0^\infty y^{1-2s}\varphi(t,x,y)\,dy
				\right\rangle_{L^2(-T,T;\wt H^{-1}(\Omega))\times L^2(-T,T;H^1(\Omega))}.
			\end{split}
		\end{equation}
		The same identity is used below for limits of functions in $\mathscr C_{c,0}$ whenever all terms are obtained by approximation in the displayed pairings. Here, 
	\begin{equation}\label{admissible test function set}
		\mathcal{L}^{1,2}_{c,0}(\R^{n+2}_+; y^{1-2s}dtdxdy):=\overline{\mathscr C_{c,0}}^{\,\|\cdot\|_{\mathfrak X}},
	\end{equation}
	with
	\begin{equation}
		\begin{split}
			\mathscr C_{c,0}:=\big\{\varphi\in C^\infty(\overline{\R^{n+2}_+}):\,
			&\operatorname{supp}\varphi\subset [-T_0,T]\times K	\text{ for some }T_0>0
			\text{ and some compact }K\subset\overline{\R^{n+1}_+},\\
			&\varphi(t,x,y)=0\text{ for }t\geq T,\quad	\varphi|_{(\Omega_e)_T\times\{0\}}=0
			\big\}.
		\end{split}
	\end{equation}
	and
	\begin{equation}
		\begin{split}
			\|\varphi\|_{\mathfrak X}:=\|\varphi\|_{\mathcal L^{1,2}(\R^{n+2}_+;y^{1-2s}dtdxdy)}
			+\|\varphi(\cdot,\cdot,0)\|_{\mathbf H^s(\R^{n+1})}+\bigg\|\int_0^\infty y^{1-2s}\varphi(\cdot,\cdot,y)\,dy\bigg\|_{L^2(\R;H^1(\R^n))}.
		\end{split}
	\end{equation}
	\end{lemma}
	
	\begin{remark}
		For the solution $w=\wt u_1-\mathcal E_su_1+u_2$ constructed in Lemma \ref{Lemma: weak solvable}, the weighted normal trace is defined by duality. More precisely, the term $\wt u_1$ is independent of $y$ and hence has zero $y$-normal contribution, while $\mathcal E_su_1$ and $u_2$ have weighted normal traces in $\mathbf H^{-s}(\R^{n+1})$. Thus
		\begin{equation}
			\lim_{y\to0}y^{1-2s}\p_yw=-\lim_{y\to0}y^{1-2s}\p_y\mathcal E_su_1		+\lim_{y\to0}y^{1-2s}\p_yu_2
		\end{equation}
		is understood as an element of $\mathbf H^{-s}_{\mathrm{loc}}(\R^{n+1})$. As a matter of fact, we can obtain that
		\begin{equation}\label{int. by part w varphi}
			\begin{split}
				&\quad \, \int_{\R^{n+2}_+} y^{1-2s} \LC  -w \p_t \varphi + \wt \sigma \nabla_{x,y} w \cdot \nabla_{x,y} \varphi \RC  dtdxdy \\
				&=\Big\langle \psi , \int_0^\infty y^{1-2s}\varphi(t,x,y)\, dy\Big\rangle_{L^2(-T,T; \wt H^{-1}(\Omega))\times  L^2(-T,T; H^{1}(\Omega))}+\big\langle			\lim_{y\to0}y^{1-2s}\p_yw,\varphi(\cdot,\cdot,0)\big\rangle_{\mathbf H^{-s}_{\mathrm{loc}},\mathbf H^s_{\mathrm c}}.
			\end{split}
		\end{equation}
		for every $\varphi\in\mathscr C_c$, and by approximation for the admissible limits used below, where
		\begin{equation}\label{admissible test function set compact}
			\mathcal{L}^{1,2}_{c}(\R^{n+2}_+; y^{1-2s}dtdxdy):=\overline{\mathscr C_{c}}^{\,\|\cdot\|_{\mathfrak X}},
		\end{equation}
		with
		\begin{equation}
			\begin{split}
				\mathscr C_{c}:=
				\big\{
				\varphi\in C^\infty(\overline{\R^{n+2}_+}):\,
				&\operatorname{supp}\varphi\subset [-T_0,T]\times K
				\text{ for some }T_0>0
				\text{ and some compact }K\subset\overline{\R^{n+1}_+},\\
				&\varphi(t,x,y)=0\text{ for }t\geq T
				\big\}.
			\end{split}
		\end{equation}
		Here, $\mathbf H^s_{\mathrm c}$ denotes compactly supported test traces in $\mathbf H^s_{\mathrm{loc}}$. The last term in \eqref{int. by part w varphi} makes sense due to the trace property $\varphi(t,x,0)\in \mathbf{H}^s(\R^{n+1})$, which is shown in the proof of Lemma \ref{Lemma: weak solvable}.
	\end{remark}
	
	We observe a simple and useful result in the next lemma.
	
	\begin{lemma}\label{Lemma: reg of local para imply nonlocal para}
		Let $u=u(t,x)$ be a function satisfying
		\begin{equation}\label{para regu}
			u\in L^2(\R;H^1(\R^n)) \quad \text{and} \quad \p_t u\in L^2(\R;H^{-1}(\R^n)).
		\end{equation}
		Then $u\in \mathbf{H}^s(\R^{n+1})$ for any $s\in (0,1)$.
	\end{lemma}
	
	\begin{proof}
		Since we have \eqref{para regu}, the Fourier transform implies that
		\begin{equation}\label{Fourier bound for u_1}
			\begin{split}
				\int_{\R^{n+1}}\LC 1+|\xi|^2 \RC \left| \widehat{u}(\rho,\xi)\right|^2\, d\rho d\xi <\infty , \quad \text{and}\quad 
				\int_{\R^{n+1}}\LC 1+ |\xi|^2 \RC^{-1} \abs{\rho}^2\left| \widehat{u}(\rho,\xi)\right|^2\, d\rho d\xi <\infty ,
			\end{split}
		\end{equation}
		where $\widehat{u}$ denotes the Fourier transform for $u$ with respect to both space and time variables. By utilizing \eqref{Fourier bound for u_1}, one has $\int_{\R^{n+1}} \left| \widehat{u}(\rho,\xi)\right|^2\, d\rho d\xi <\infty$. Meanwhile, combined with \eqref{Fourier bound for u_1} and H\"older's inequality, we get
		\begin{equation}\label{Fourier bound for u_1 2}
			\begin{split}
				&\quad \int_{\R^{n+1}} \abs{\rho}\left| \widehat{u}(\rho,\xi)\right|^2\, d\rho d\xi \\
				&\leq \bigg( 	\int_{\R^{n+1}}\LC 1+|\xi|^2 \RC \left| \widehat{u}(\rho,\xi)\right|^2\, d\rho d\xi  \bigg)^{1/2}
				\bigg(	\int_{\R^{n+1}}\LC 1+ |\xi|^2 \RC^{-1} \abs{\rho}^2\left| \widehat{u}(\rho,\xi)\right|^2\, d\rho d\xi\bigg)^{1/2}<\infty.
			\end{split}
		\end{equation}
		Therefore, \eqref{Fourier bound for u_1} and \eqref{Fourier bound for u_1 2} yield that
		\begin{equation}
			\int_{\R^{n+1}}\LC 1+\LC |\xi|^4+|\rho|^2 \RC^{1/2}\RC^{s} \left| \widehat{u}(\rho,\xi)\right|^2\, d\rho d\xi <\infty,
		\end{equation}
		for $s\in (0,1)$, which infers $u\in \mathbf{H}^s(\R^{n+1})=\mathbb{H}^s(\R^{n+1})$ as desired.
	\end{proof}
	
	Now, we are ready to prove Lemma \ref{Lemma: weak solvable}.
	
	\begin{proof}[Proof of Lemma \ref{Lemma: weak solvable}]
		We split the proof into three steps.
		
		\medskip
		
		{\it Step 1. Initiation.}
		
		\medskip
		
		\noindent Notice that $\psi \in L^2(-T,T;\wt H^{-1}(\Omega))$, so we may assume that $\supp(\psi)\subset \Omega_T\subset \R^{n}_T$ with $\psi \in L^2(-T,T;  H^{-1}(\R^n))$. Note that $\overline{\Omega_T}$ is a compact set in $\R^{n+1}$. Consider the equation
		\begin{equation}\label{equ u_1}
			\begin{cases}
				\LC  \p_t -\nabla \cdot \sigma \nabla \RC u_1 =\psi  &\text{ in }\R^{n}_T,\\
				u_1(-T,x)=0 &\text{ for }x \in \R^n,
			\end{cases}
		\end{equation}
		then the preceding equation is solvable by using the variational method, for instance, see \cite[Chapter XVIII, Section 3]{DL1992}. That is, there exists a solution $u_1\in L^{2}(-T,T;H^1(\R^n))$ to \eqref{equ u_1}. Meanwhile, $\nabla \cdot \sigma \nabla u_1 \in L^2(-T,T; H^{-1}(\R^n))$ so that
		\begin{equation}\label{regularity of u_1}
			u_1\in L^{2}(-T,T; H^1(\R^n)) \quad  \text{and}  \quad \p_t u_1 \in L^2(-T,T; H^{-1}(\R^n)).
		\end{equation}
		Moreover, the standard energy estimate for \eqref{equ u_1} gives \eqref{energy estimate u1}.
		Indeed, testing the weak formulation of \eqref{equ u_1} by $u_1$ on $(-T,\tau)$ and using $u_1(-T,\cdot)=0$, the ellipticity condition \eqref{ellipticity condition}, and Young's inequality, one obtains
		\begin{equation}
			\frac12\norm{u_1(\tau,\cdot)}_{L^2(\R^n)}^2
			+\lambda\int_{-T}^{\tau}\int_{\R^n}\abs{\nabla u_1}^2\,dxdt
			\leq
			\norm{\psi}_{L^2(-T,T;H^{-1}(\R^n))}
			\norm{u_1}_{L^2(-T,T;H^1(\R^n))}.
		\end{equation}
		Taking the supremum over $\tau\in[-T,T]$ and using the equation again gives \eqref{energy estimate u1}.
		
		To proceed, we extend the function $u_1$ to the whole space-time with the same notation. We extend $u_1$ by zero to the set $\{t\leq -T\}\times\R^n$ by the initial condition in \eqref{equ u_1}. For $t\geq T$, we define $u_1(t+T,x)=u_1(T-t,x)$ for $t\geq 0$ and $x\in \R^n$. Then the extended function $u_1(t,x)$ is compactly supported in time in $\{-T\leq t\leq 3T\}$ and satisfies
		\begin{equation}
			u_1\in L^{2}(\R; H^1(\R^n)) \quad  \text{and}  \quad \p_t u_1 \in L^2(\R; H^{-1}(\R^n)).
		\end{equation}
		By Lemma \ref{Lemma: reg of local para imply nonlocal para}, we can see that $u_1 \in \mathbf{H}^s(\R^{n+1})$.
		
		\medskip
		
		{\it Step 2. Artificial solutions.}
		
		\medskip
		
		\noindent Let us set an auxiliary function $\wt u_1=\wt u_1(t,x,y)$ so that $\wt u_1$ is constant with respect to the $y$-direction. In other words, $\wt u_1(t,x,y):=u_1(t,x)$, where $u_1$ is a solution to \eqref{equ u_1}. It follows directly that $\wt u_1$ is a solution to
		\begin{equation}\label{equ wt u_1}
			\begin{cases}
				y^{1-2s}\p_t \wt u_1 -\nabla_{x,y}\cdot \LC y^{1-2s}\wt \sigma \nabla_{x,y}\wt u_1 \RC  =y^{1-2s}\LC \p_t  u_1 -\nabla \cdot \sigma \nabla  u_1 \RC=y^{1-2s}\psi &\text{ in }\R^n_T \times(0,\infty), \\
				\wt u_1 (t,x,y)=0 &\text{ in }\{t\leq -T\}\times \R^{n+1}_+,
			\end{cases}
		\end{equation}
		where we utilized the equation \eqref{equ u_1}.
		
		Let us consider the operator $\mathcal{E}_s$ satisfying
		\begin{equation}
			\begin{split}
				\mathcal{E}_s: \mathbf{H}^s (\R^{n+1})\to  \mathcal{L}^{1,2}(\R^{n+2}_+; y^{1-2s}dtdxdy), \quad 	u_1 \mapsto \mathcal{E}_s u_1,
			\end{split}
		\end{equation}
		which stands for the extension operator. Then $\mathcal{E}_su_1=\LC \mathcal{E}_su_1\RC(t,x,y)$ is a solution to the extension problem
		\begin{equation}\label{equ E_s u_1 no intial}
			\begin{cases}
				y^{1-2s}\p_t \mathcal{E}_su_1 -\nabla_{x,y}\cdot \LC y^{1-2s}\wt \sigma \nabla_{x,y} \mathcal{E}_su_1 \RC =0 &\text{ in }\R^{n+2}_+, \\
				\LC \mathcal{E}_s u_1\RC (t,x,0)=u_1(t,x) & \text{ for }(t,x)\in \R^{n+1}.
			\end{cases}
		\end{equation}
		Thus, via Proposition \ref{Prop: extension} \ref{item b prop extension}, we have $\lim_{y\to 0 }y^{1-2s} \p_y \mathcal{E}_s u_1\in \mathbf{H}^{-s}(\R^{n+1})$. Moreover, by the representation formula \eqref{wt u expression} and the past-time support of $u_1$, we have
		\[
		\LC \mathcal{E}_s u_1\RC (t,x,y)=0 \quad\text{for }(x,y)\in \R^{n+1}_+,\ t\leq -T.
		\] Combined with \eqref{equ E_s u_1 no intial}, one can derive
		\begin{equation}\label{equ E_s u_1}
			\begin{cases}
				y^{1-2s}\p_t \mathcal{E}_su_1 -\nabla_{x,y}\cdot \LC y^{1-2s}\wt \sigma \nabla_{x,y} \mathcal{E}_su_1 \RC =0 &\text{ in }\R^{n+2}_+, \\
				\LC \mathcal{E}_s u_1\RC (t,x,0)=u_1(t,x) & \text{ for }(t,x)\in \R^{n+1},\\
				\LC \mathcal{E}_s u_1\RC (t,x,y)=0 &\text{ for }(x,y)\in \R^{n+1}_+,\ t\leq -T.
			\end{cases}
		\end{equation}
		
		On the other hand, we consider the problem
		\begin{equation}\label{equ u_2}
			\begin{cases}
				y^{1-2s}\p_t u_2 -\nabla_{x,y}\cdot \LC y^{1-2s}\wt \sigma \nabla_{x,y} u_2 \RC  =0 &\text{ in }\R^n_T \times(0,\infty), \\
				u_2 =0 &\text{ on }(\Omega_e)_T \times \{0\}, \\
				\displaystyle \lim_{y\to 0}y^{1-2s} \p_y u_2 =\lim_{y\to 0 }y^{1-2s} \p_y \mathcal{E}_s u_1 &\text{ on } \Omega_T \times \{0\}, \\
				u_2 (t,x,y)=0 &\text{ for }(x,y)\in \R^{n+1}_+,\ t\leq -T.
			\end{cases}
		\end{equation}
		We discuss the solvability of \eqref{equ u_2} by considering
		\begin{equation}\label{comp of u_2 s}
			\begin{split}
				\int_{\R^{n+2}_+} y^{1-2s}\LC  -u_2 \p_t \varphi + \wt \sigma \nabla_{x,y} u_2 \cdot \nabla_{x,y} \varphi \RC  dtdxdy =\big\langle
				\lim_{y\to0}y^{1-2s}\p_y\mathcal E_su_1,
				\varphi(\cdot,\cdot,0)
				\big\rangle_{\mathbf H^{-s}(\R^{n+1}),\mathbf H^s(\R^{n+1})}.
			\end{split}
		\end{equation}
		for every $\varphi\in\mathscr C_{c,0}$. As we showed before,
		$\lim_{y\to 0} y^{1-2s}\p_y \mathcal{E}_s u_1\in \mathbf H^{-s}(\R^{n+1})$.
		By the definition of the test space \eqref{admissible test function set}, the trace map $\varphi\mapsto \varphi(\cdot,\cdot,0)$ is continuous from $\mathcal L^{1,2}_{c,0}$ into $\mathbf H^s(\R^{n+1})$. Therefore,
		\begin{equation}
			\mathcal L^{1,2}_{c,0}(\R^{n+2}_+;y^{1-2s}dtdxdy)\ni\varphi
			\mapsto
			\big\langle
			\lim_{y\to0}y^{1-2s}\p_y\mathcal E_su_1,
			\varphi(\cdot,\cdot,0)
			\big\rangle_{\mathbf H^{-s}(\R^{n+1}),\mathbf H^s(\R^{n+1})}
		\end{equation}
		is bounded. Thus, by the standard variational method, the problem \eqref{equ u_2} possesses a unique solution
		$u_2\in\mathcal L^{1,2}(\R^n_T\times(0,\infty),y^{1-2s}dtdxdy)$.
		By a standard extension in time, one may extend $u_2(t,x,y)$ for $t>T$ such that
		$u_2\in\mathcal L^{1,2}(\R^{n+2}_+,y^{1-2s}dtdxdy)$.
		
		Now, let $\varphi\in\mathscr C_{c,0}$ be arbitrary, then we have
		\begin{equation}\label{comp of E_1 u_1}
			 \int_{\R^{n+2}_+} y^{1-2s}\LC  -\mathcal{E}_s u_1 \p_t \varphi + \wt \sigma \nabla_{x,y} \mathcal{E}_s u_1 \cdot \nabla_{x,y} \varphi \RC  dtdxdy =
				\big\langle \lim_{y\to0}y^{1-2s}\p_y\mathcal E_su_1, \varphi(\cdot,\cdot,0) \big\rangle_{\mathbf H^{-s}(\R^{n+1}),\mathbf H^s(\R^{n+1})},
		\end{equation}
		and
		\begin{equation}\label{comp of u_2}
			\begin{split}
			 \int_{\R^{n+2}_+} y^{1-2s}\LC  -u_2 \p_t \varphi + \wt \sigma \nabla_{x,y} u_2 \cdot \nabla_{x,y} \varphi \RC  dtdxdy =
			 \big\langle \lim_{y\to0}y^{1-2s}\p_y\mathcal E_su_1, \varphi(\cdot,\cdot,0) \big\rangle_{\mathbf H^{-s}(\R^{n+1}),\mathbf H^s(\R^{n+1})}.
			\end{split}
		\end{equation}
		Since $\varphi|_{(\Omega_e)_T\times\{0\}}=0$, this global pairing is supported in $\R\times\Omega$.
		
		\medskip
		
		{\it Step 3. Construction of solutions.}
		
		\medskip
		
		\noindent We prove that
		\begin{equation}\label{w constructed in Lemma solve}
			w:= \wt u_1-\mathcal{E}_s u_1+u_2,\quad \wt u_1(t,x,y):=u_1(t,x),
		\end{equation}
		is an admissible distributional solution to \eqref{adjoint problem generalized}. Notice that $\wt u_1$ is independent of $y$, and hence it is not expected to belong to the global weighted space $\mathcal L^{1,2}(\R^{n+2}_+;y^{1-2s}dtdxdy)$. However, $\wt u_1$ is admissible against compactly supported test functions in the $y$-variable, while $-\mathcal E_su_1+u_2\in \mathcal L^{1,2}(\R^{n+2}_+;y^{1-2s}dtdxdy)$. Therefore, the function $w$ belongs to $\mathcal L^{1,2}_{\mathrm{loc}}(\R^{n+2}_+;y^{1-2s}dtdxdy)$ and all the pairings below are well-defined for test functions in $\mathscr C_{c,0}$. The extension to admissible limits used later is understood through the approximation convention in the statement.
		
		For the term $\wt u_1$, we use the weak formulation of \eqref{equ u_1} with the test function
		\[
		\int_0^\infty y^{1-2s}\varphi(\cdot,\cdot,y)\,dy.
		\]
		This test function is admissible. Indeed, since $\varphi$ has compact support in the $y$-variable, for $M>0$ sufficiently large,
		\begin{equation}
			\begin{split}
				\int_{\R^{n+1}}
				\left|
				\nabla\left(\int_0^\infty y^{1-2s}\varphi\,dy\right)
				\right|^2\,dtdx
				&=
				\int_{\R^{n+1}}
				\left|
				\int_0^M y^{1-2s}\nabla\varphi\,dy
				\right|^2\,dtdx\\
				&\leq
				\int_{\R^{n+1}}
				\left(\int_0^M y^{1-2s}\,dy\right)
				\left(\int_0^M y^{1-2s}|\nabla\varphi|^2\,dy\right)
				\,dtdx\\
				&\leq
				\frac{M^{2-2s}}{2-2s}
				\int_{\R^{n+2}_+}y^{1-2s}|\nabla\varphi|^2\,dtdxdy<\infty.
			\end{split}
		\end{equation}
		The same argument gives
		\[
		\int_0^\infty y^{1-2s}\varphi(\cdot,\cdot,y)\,dy
		\in L^2(\R;H^1(\R^n)).
		\]
		Moreover, since $\varphi(t,x,y)=0$ for $t\geq T$, this test function also vanishes for $t\geq T$. Therefore, by the weak formulation of \eqref{equ u_1}, we have
		\begin{equation}\label{comp of u_1}
			\begin{split}
				&\quad\, \int_{\R^{n+2}_+}y^{1-2s}\bigg(-\wt u_1\p_t\varphi			+\wt\sigma\nabla_{x,y}\wt u_1\cdot\nabla_{x,y}\varphi\bigg)\,dtdxdy\\
				&=\int_{\R^{n+1}}\bigg\{-u_1\p_t\bigg(\int_0^\infty y^{1-2s}\varphi\,dy\bigg)
				+\sigma\nabla u_1\cdot	\nabla \bigg(\int_0^\infty y^{1-2s}\varphi\,dy\bigg)			\bigg\}\,dtdx\\
				&= \bigg\langle 	\psi,\int_0^\infty y^{1-2s}\varphi(\cdot,\cdot,y)\,dy
				\bigg\rangle_{L^2(-T,T;H^{-1}(\R^n)),\,L^2(-T,T;H^1(\R^n))}.
			\end{split}
		\end{equation}		
		Hence, one can conclude that $w:=\wt u_1-\mathcal E_su_1+u_2$ satisfies the weak identity \eqref{weak solvable identity} for every $\varphi\in\mathscr C_{c,0}$. By approximation, the same identity applies to the admissible limits used below. Moreover, $w\in\mathcal L^{1,2}_{\mathrm{loc}}(\R^{n+2}_+;y^{1-2s}dtdxdy)$ and $w-\wt u_1\in\mathcal L^{1,2}(\R^{n+2}_+;y^{1-2s}dtdxdy)$. This proves the assertion.
	\end{proof}
	
	\begin{remark}\label{Remark: adjoint solution}
		Let us emphasize that
		\begin{enumerate}[(a)]
			\item Lemma \ref{Lemma: weak solvable} also implies the existence of solutions to the following backward adjoint problem
			\begin{equation}\label{adjoint problem generalized_rmk}
				\begin{cases}
					y^{1-2s}\p_t \wt w + \nabla_{x,y} \cdot \LC y^{1-2s} \wt \sigma \nabla_{x,y}\wt  w\RC= y^{1-2s}\psi &\text{ in }\R^n_T \times(0,\infty), \\
					\wt w=0 &\text{ in } \LC \Omega_e \RC_T \times \{0\}, \\
					\displaystyle \lim_{y\to 0}y^{1-2s}\p_y \wt w =0 &\text{ in }\Omega_T \times \{0\}, \\
					\wt w(t,x,y) =0 &\text{ for }t\geq T \text{ and }(x,y)\in \R^{n+1}_+.
				\end{cases}
			\end{equation}
			Indeed, apply Lemma \ref{Lemma: weak solvable} to the source
			\begin{equation}
				\psi^\sharp(t,x):=-\psi(-t,x).
			\end{equation}
			Let $w^\sharp$ be the corresponding forward solution of \eqref{adjoint problem generalized} with source $\psi^\sharp$, and set $\wt w(t,x,y):=w^\sharp(-t,x,y)$, then $\wt w$ solves \eqref{adjoint problem generalized_rmk} with source $\psi$ and satisfies the terminal condition $\wt w(t,x,y)=0$ for $t\geq T$.
			
			\item From Step 1 in the proof of Lemma \ref{Lemma: weak solvable}, we know that if $u$ is a parabolic solution satisfying \eqref{para regu}, then $u$ will satisfy the regularity for the nonlocal parabolic equation, i.e., $u\in \mathbf{H}^s(\R^{n+1})$. This also gives us some hints to relate nonlocal and local parabolic equations.
		\end{enumerate}
	\end{remark}
	
	\begin{proof}[Proof of Proposition \ref{Prop: density}]
		We first prove the interior density
		\begin{equation}\label{density interior proof target}
			\overline{\mathcal V}^{\,L^2(-T,T;H^1(\Omega))}=\mathcal D_\sigma,
		\end{equation}
		where $\mathcal D_\sigma$ is defined in \eqref{local solution space D}. By Theorem \ref{Thm: key equation}, every $v_f\in\mathcal V$ belongs to $\mathcal D_\sigma$. By the Hahn--Banach theorem, it is enough to show that if $\psi \in L^2(-T,T; \wt H^{-1}(\Omega))$ with $\psi(v_f)=0$ for all $f\in C^\infty_c(W_T)$, then there holds $\psi(v)=0$ for all $v\in \mathcal D_\sigma$. In the rest of the proof, we adopt the notation that $\wt u_f$ denotes solutions of \eqref{degen para PDE} with $\wt u_f(t,x,0)=u_f(t,x)$ and $v_f=\int_{0}^\infty y^{1-2s}\wt u_f\,dy$, where $u_f\in \mathbf{H}^s(\R^{n+1})$ is the solution of \eqref{equ nonlocal para} with the exterior data $f$.
		
		\medskip
		
		{\it Step 1. Smooth cutoffs.}
		
		\medskip
		
		\noindent Let us consider the auxiliary problem \eqref{adjoint problem generalized}, which is solvable by Lemma \ref{Lemma: weak solvable}. We introduce two cutoff functions, one for $x$-variable and the other for $y$-variable. On one hand, let $\zeta_k(y)=\zeta(y/k)$ for $k\in \N$, where
		$\zeta \in C^\infty_c([0,2])$ is a smooth cutoff function satisfying
		$0\leq \zeta\leq1$ and $\zeta\equiv 1$ on $[0,1]$. On the other hand, we may assume that $\overline{\Omega}\cup \overline{W}\subset B_R$ for sufficiently large $R>0$, where $B_R$ is the ball in $\R^n$ of radius $R$ and center at the origin. Let $\eta_k(x):=\eta(x/k)$ for $k\in \N$, where $\eta \in C^\infty_c(B_{2R})$ is a radial cutoff function such that $\eta=1$ in $B_R$. With these smooth cutoff functions at hand, we can see that
		\begin{equation}
			\wt u_{f,k}(t,x,y):=\wt u_f(t,x,y)\eta_k(x)\zeta_k(y)\in \mathcal{L}^{1,2}_{c}(\R^{n+2}_+;y^{1-2s}dtdxdy),
		\end{equation}
		for any $f\in C^\infty_c(W_T)$.
		
		As in Remark \ref{Remark: adjoint solution}, we have constructed that the function $\wt w$ is a solution to the backward equation \eqref{adjoint problem generalized_rmk}. Recall that $\psi\in L^2(-T,T;\wt H^{-1}(\Omega))$, then
		\begin{equation}\label{pairing for psi(v_f)}
			\begin{split}
				0&=\psi(v_f)=\lim_{k\to \infty} \Big\langle   \psi , \int_0^\infty y^{1-2s}\wt u_{f,k} \, dy  \Big\rangle _{L^2(-T,T; \wt H^{-1}(\Omega)), L^2(-T,T;H^1(\Omega))} \\
				&=\lim_{k\to \infty} \int_{\R^{n+2}_+} \left[ y^{1-2s} \p_t \wt w +\nabla_{x,y}\cdot \LC y^{1-2s}\wt \sigma \nabla_{x,y}\wt w\RC  \right] \wt u_{f,k} \, dtdxdy \\
				&= \lim_{k \to \infty}  \bigg\{ \int_{\R^{n+2}_+} y^{1-2s} \wt u_{f,k} \p_t \wt w\, dtdxdy+ \big\langle \lim_{y\to0}y^{1-2s}\p_y\wt w, \wt u_{f,k}(\cdot,\cdot,0)
				\big\rangle  \\
				&\quad\quad \quad \quad \quad  -\int_{\R^{n+2}_+}y^{1-2s}\wt \sigma \nabla_{x,y} \wt w \cdot \nabla_{x,y} \wt u_{f,k}\, dtdxdy \bigg\} \\
				&=\big\langle\lim_{y\to0}y^{1-2s}\p_y\wt w, f\big\rangle_{\mathbf H^{-s}(W_T),\widetilde{\mathbf H}^s(W_T)}
				+\lim_{k\to \infty} I_k ,
			\end{split}
		\end{equation}
		where
		\begin{equation}
			\begin{split}
				I_k:= &- \int_{\R^{n+2}_+}y^{1-2s}\wt \sigma \nabla_{x,y} \wt w \cdot \nabla_{x,y} \wt u_{f,k}\, dtdxdy +\int_{\R^{n+2}_+} y^{1-2s} \wt u_{f,k} \p_t \wt w\, dtdxdy\\
				= &- \int_{\R^{n+2}_+}y^{1-2s}\wt \sigma \nabla_{x,y} \wt u_{f,k} \cdot \nabla_{x,y} \wt w\, dtdxdy +\int_{\R^{n+2}_+} y^{1-2s} \wt u_{f,k} \p_t \wt w\, dtdxdy.
			\end{split}
		\end{equation}
		Next, integration by parts in the time variable yields that
		\begin{equation}\label{I_k}
			\begin{split}
				I_k &  = -\int_{\R^{n+2}_+}y^{1-2s}\eta_k \zeta_k \wt \sigma \LC \nabla_{x,y} \wt u_f \RC \cdot \nabla_{x,y}\wt w \, dtdxdy  \\
				&\quad - \int_{\R^{n+2}_+}y^{1-2s}\wt u_f \wt \sigma  \nabla_{x,y}\LC \eta_k\zeta_k \RC\cdot \nabla_{x,y}\wt w\,dtdxdy
				-\int_{\R^{n+2}_+}y^{1-2s}\wt w\LC \p_t\wt u_{f,k}\RC\,dtdxdy.
			\end{split}
		\end{equation}
		Here we used $\wt w(T,x,y)=\wt u_{f,k}(-T,x,y)=0$.
		
		In \eqref{pairing for psi(v_f)}, we also point out that the limit holds $\int_0^\infty y^{1-2s} \wt u_{f,k}(t,x,y) \, dy \to  \int_0^\infty y^{1-2s} \wt u_f(t,x,y)\,dy$ in $L^2(\R;H^1(\Omega))$ as $k\to\infty$. In fact, there holds
		\begin{equation}
			\begin{split}
				&\quad \, \left\| \int_0^\infty y^{1-2s} \zeta_k(y)\eta_k(x) \wt u_f(t,x,y) \, dy- \int_0^\infty y^{1-2s}\wt u_f(t,x,y)\,dy\right\|_{L^2(\R;H^1(\Omega))} \\
				&\leq  \left\|  \int_k^{\infty} y^{1-2s}\LC \left| \wt u_f \right| + \left|\nabla  \wt u_f \right| \RC dy \right\|_{L^2(\R;L^2(\Omega))}\\
				&\lesssim  \int_{k}^{\infty} y^{1-2s} \LC  \left\| \wt u_f(\cdot,\cdot,y)\right\|_{L^2(\R;L^\infty(\R^n))} + \left\|\nabla \wt u_f(\cdot,\cdot,y)\right\|_{L^2(\R;L^\infty(\R^n))} \RC dy \\
				&\lesssim \int_{k}^{\infty} y^{1-2s-n} \left\|  u_f\right\|_{L^2(\R;L^1(\R^n))}\,dy \\
				&\lesssim k^{2-2s-n}\norm{f}_{\widetilde{\mathbf{H}}^s(W_T)}\to0,
			\end{split}
		\end{equation}
		as $k\to\infty$, for $n\geq2$, where we used \eqref{decay estimate in x}, \eqref{equ nonlocal para estimate}, and the compact support of $u_f$ in $\overline{(\Omega\cup W)_T}$. We next analyze the limit of $I_k$ as $k\to\infty$.
		
		\medskip
		
		{\it Step 2. $\displaystyle\lim_{k\to\infty}I_k=0$.}
		
		\medskip
		
		\noindent We want to claim that $\lim_{k\to\infty}I_k=0$. The argument is similar to the proof of \cite[Proposition 3.1]{CGRU2023reduction}; we provide a detailed derivation for the sake of completeness. Integrating by parts on the right hand side of $I_k$ in \eqref{I_k}, one has
		\begin{equation}\label{comp I_k1}
			\begin{split}
				I_k &=-\int_{\R^{n+2}_+} y^{1-2s}  \wt w \LC \p_t \wt u_{f,k} \RC\,dtdxdy-\int_{\R^{n+2}_+} y^{1-2s}\wt \sigma \nabla_{x,y}\wt u_f\cdot\nabla_{x,y}\LC \eta_k\zeta_k\wt w\RC\,dtdxdy\\
				&\quad \, +\int_{\R^{n+2}_+}y^{1-2s}\wt w\wt \sigma\nabla_{x,y}\wt u_f\cdot\nabla_{x,y}\LC \eta_k\zeta_k\RC\,dtdxdy\\
				&\quad \, +\int_{\R^{n+2}_+}\wt w\nabla_{x,y}\cdot\LC y^{1-2s}\wt u_f\wt \sigma\nabla_{x,y}\LC \eta_k\zeta_k\RC\RC\,dtdxdy\\
				&=\big\langle	\lim_{y\to0}y^{1-2s}\p_y\wt u_f, \eta_k\wt w(\cdot,\cdot,0)
				\big\rangle \\
				&\quad \, +\int_{\R^{n+2}_+}\eta_k\zeta_k\wt w\left[\nabla_{x,y}\cdot\LC y^{1-2s}\wt \sigma\nabla_{x,y}\wt u_f\RC-y^{1-2s}\p_t\wt u_f\right]\,dtdxdy\\
				&\quad \, +2\int_{\R^{n+2}_+}y^{1-2s}\wt w\wt \sigma\nabla_{x,y}\LC \eta_k\zeta_k\RC\cdot\nabla_{x,y}\wt u_f\,dtdxdy\\
				&\quad \, +\int_{\R^{n+2}_+}\wt w y^{1-2s}\wt u_f \widetilde{\mathcal L}\LC \eta_k\zeta_k\RC\,dtdxdy
				+(1-2s)\int_{\R^{n+2}_+}y^{-2s}\wt w\wt u_f\eta_k\p_y\zeta_k\,dtdxdy\\
				&=\int_0^{2k}\int_{B_{2Rk}}\int_{\R}y^{1-2s}\wt w\Big\{2\wt \sigma\nabla_{x,y}\LC \eta_k\zeta_k\RC\cdot\nabla_{x,y}\wt u_f+\Big( \widetilde{\mathcal L}\big(  \eta_k\zeta_k\big)+\frac{1-2s}{y}\eta_k\p_y\zeta_k\Big)\wt u_f\Big\}\,dtdxdy,
			\end{split}
		\end{equation}
		where we used the notation $\widetilde{\mathcal L}:=\nabla_{x,y}\cdot\wt \sigma\nabla_{x,y}$. In the last equality, the boundary term is zero because $\wt w=0$ on $(\Omega_e)_T\times\{0\}$ and $\lim_{y\to0}y^{1-2s}\p_y\wt u_f=0$ in $\Omega_T$. The other bulk term is zero since $\wt u_f$ is a solution to \eqref{degen para PDE}.
		
		We next estimate $I_k$. Let $A_k:=B_{2Rk}\setminus B_{Rk}$, where $R>0$ is the radius given from previous steps. With the uniform boundedness of $\wt \sigma$ and $\nabla\wt \sigma$ at hand, one can derive
		\begin{equation}
			\begin{split}
				\left|I_k\right|&\lesssim \int_0^{2k}\int_{B_{2Rk}}\int_{\R}y^{1-2s}\abs{\wt w}\abs{\nabla_{x,y}\LC\eta_k\zeta_k\RC}\abs{\nabla_{x,y}\wt u_f}\,dtdxdy\\
				&\quad +\int_0^{2k}\int_{B_{2Rk}}\int_{\R}y^{1-2s}\abs{\wt w}\LC \abs{\widetilde{\mathcal L}\LC\eta_k\zeta_k\RC}+y^{-1}\abs{\eta_k}\abs{\p_y\zeta_k}\RC\abs{\wt u_f}\,dtdxdy\\
				&\lesssim \int_0^{2k}\int_{B_{2Rk}}\int_{\R}y^{1-2s}\abs{\wt w}\LC \abs{\nabla\eta_k}+\abs{\p_y\zeta_k}\RC\abs{\nabla_{x,y}\wt u_f}\,dtdxdy\\
				&\quad +\int_0^{2k}\int_{B_{2Rk}}\int_{\R}y^{1-2s}\abs{\wt w}\LC \abs{\nabla\cdot\sigma\nabla\eta_k}+\abs{\p_y^2\zeta_k}+y^{-1}\abs{\p_y\zeta_k}\RC\abs{\wt u_f}\,dtdxdy\\
				&\lesssim I_{1,k}\LC\wt w\RC+I_{2,k}\LC\wt w\RC.
			\end{split}
		\end{equation}
		Here we used \eqref{cutoff variable coefficient estimate} for $\abs{\nabla\cdot\sigma\nabla\eta_k}\lesssim k^{-1}$. More explicitly, since $\eta_k(x)=\eta(x/k)$, we have $|\nabla\eta_k|\lesssim k^{-1}$ and $|D^2\eta_k|\lesssim k^{-2}$, and therefore
		\begin{equation}
			|\nabla\cdot\sigma\nabla\eta_k|\leq |\sigma||D^2\eta_k|+|\nabla\sigma||\nabla\eta_k|\lesssim k^{-1}.
		\end{equation}
		This is the point where the parabolic proof differs from the elliptic reduction: the weaker $k^{-1}$ bound is sufficient for the estimates below, and no exterior identity assumption is used.
		
		We have set
		\begin{equation}
			I_{1,k}\LC \wt w\RC := k^{-1}\int_k^{2k} \int_{\R^{n+1}} y^{1-2s}\left| \wt w \right| \LC \left| \nabla_{x,y} \wt u_f \right| + k^{-1}\left| \wt u_f \right|  \RC dtdxdy,
		\end{equation}
		and
		\begin{equation}
			I_{2,k}\LC\wt w\RC := k^{-1}\int_0^{2k}\int_{A_k}\int_{\R}y^{1-2s}\left| \wt w \right| \LC \left|\nabla_{x,y}\wt u_f\right|+\left|\wt u_f\right|\RC dtdxdy.
		\end{equation}
		We next estimate $I_{1,k}(\wt w)$ and $I_{2,k}(\wt w)$ separately.
		
		\medskip
		
		\noindent{\it Step 2a. Estimate for $I_{1,k}(\wt w)$.} Note that the function $\wt w$ is constructed from the solution $w=\wt u_1-\mathcal{E}_s u_1+u_2$ given by \eqref{w constructed in Lemma solve} by reversing time $t\mapsto -t$, and we abuse the notation $\wt w$ as the form $\wt u_1-\mathcal{E}_s u_1+u_2$. Here we already replace $u_1(t,x)$ and $u_2(t,x,y)$ by $u_1(-t,x)$ and $u_2(-t,x,y)$, respectively. Let us consider the bound
		\begin{equation}
			\int_{\R^{n+1}}\abs{\mathbf{w}(t,x,y)}\big|\nabla_{x,y}^\ell\wt u_f(t,x,y)\big|\,dtdx
		\end{equation}
		for $\ell=0,1$, where $\mathbf{w}$ could be any of the functions $\wt u_1$, $\mathcal{E}_s u_1$ or $u_2$. By Lemma \ref{Lemma: decay estimate} and Lemma \ref{Lemma: weak solvable}, it is known that both functions $\mathcal{E}_s u_1$ and $u_2$ have decay in the $y$-direction, but the function $\wt u_1$ does not have such decay. Hence, we divide the proof into two parts: $\mathbf{w}=\wt u_1$ and $\mathbf{w}=\mathcal{E}_s u_1,u_2$.
		
		\begin{itemize}
			\item For $\mathbf{w}=\wt u_1$, for $\ell=0,1$, by H\"older's inequality, we have
			\begin{equation}\label{Step 2a 1}
				\begin{split}
					\int_{\R^{n+1}}\abs{\mathbf{w}(t,x,y)}\left|\nabla_{x,y}^\ell \wt u_f(t,x,y)\right|\,dtdx 	&\leq \norm{\mathbf{w}(\cdot,\cdot,y)}_{L^2(\R^{n+1})}\big\|\nabla_{x,y}^\ell\wt u_f(\cdot,\cdot,y)\big\|_{L^2(\R^{n+1})}\\
					&\lesssim y^{\frac{n}{2}-n-\ell}\norm{\mathbf{w}}_{L^\infty(\R;L^2(\R^n))}\left\|u_f\right\|_{L^2(\R;L^1(\R^n))}.
				\end{split}
			\end{equation}
			By H\"older's inequality and the compact support of $u_f$, one has
			\begin{equation}\label{estimate wt u_f}
				\left\|u_f\right\|_{L^2(\R;L^1(\R^n))}\lesssim \left\|f\right\|_{\widetilde{\mathbf H}^s(W_T)}.
			\end{equation}
			As a result, we can summarize
			\begin{equation}\label{I_{1,k} estimate case 1}
				\begin{split}
					I_{1,k}(\wt w)&\lesssim k^{-1}\norm{\mathbf{w}}_{L^\infty(\R;L^2(\R^n))}\norm{f}_{\widetilde{\mathbf H}^s(W_T)}\int_k^{2k}y^{1-\frac n2-\ell-2s}\,dy\\
					&\lesssim k^{1-\frac n2-\ell-2s}\norm{\mathbf{w}}_{L^\infty(\R;L^2(\R^n))}\norm{f}_{\widetilde{\mathbf H}^s(W_T)}.
				\end{split}
			\end{equation}
			Via \eqref{I_{1,k} estimate case 1}, it follows that $I_{1,k}\to0$ as $k\to\infty$.
			
			\item For $\mathbf{w}=\mathcal{E}_s u_1,u_2$, we denote by $\mathbf{w}$ any of the functions $\mathcal{E}_s u_1,u_2$. For $\ell\in\{0,1\}$, the decay estimate \eqref{decay estimate in y} for the function $\wt u_f$ implies that
			\begin{equation}
				\begin{split}
					\int_{\R^{n+1}}\abs{\mathbf{w}(t,x,y)}\left|\nabla_{x,y}^\ell \wt u_f(t,x,y)\right|\,dtdx
					&\leq \norm{\mathbf{w}(\cdot,\cdot,y)}_{L^2(\R^{n+1})}\big\|\nabla_{x,y}^\ell\wt u_f(\cdot,\cdot,y)\big\|_{L^2(\R^{n+1})}\\
					&\lesssim y^{-\frac n2-\ell}\norm{\mathbf{w}(\cdot,\cdot,y)}_{L^2(\R^{n+1})}\left\|u_f\right\|_{L^2(\R;L^1(\R^n))}.
				\end{split}
			\end{equation}
			Then H\"older's inequality with respect to the $y$-direction yields that
			\begin{equation}\label{I_{1,k} estimate case 2}
				\begin{split}
					I_{1,k}(\wt w)
					&\lesssim k^{-1}\int_k^{2k}y^{1-2s-\frac n2-\ell}\norm{\mathbf{w}(\cdot,\cdot,y)}_{L^2(\R^{n+1})}\left\|u_f\right\|_{L^2(\R;L^1(\R^n))}\,dy\\
					&\lesssim k^{-1}\norm{\mathbf{w}}_{\mathcal L^{1,2}(\R^{n+2}_+;y^{1-2s}dtdxdy)}\norm{f}_{\widetilde{\mathbf H}^s(W_T)}\bigg(\int_k^{2k}y^{1-n-2\ell-2s}\,dy\bigg)^{1/2}\\
					&\lesssim k^{-\frac n2-\ell-s}\norm{\mathbf{w}}_{\mathcal L^{1,2}(\R^{n+2}_+;y^{1-2s}dtdxdy)}\norm{f}_{\widetilde{\mathbf H}^s(W_T)}.
				\end{split}
			\end{equation}
			Via \eqref{I_{1,k} estimate case 2}, one can also see that $I_{1,k}(\wt w)\to0$ as $k\to\infty$.
		\end{itemize}
		Therefore, $I_{1,k}(\wt w)\to0$ as $k\to\infty$.
		
		\medskip
		
		\noindent{\it Step 2b. Estimate for $I_{2,k}(\wt w)$.}
		Recall \eqref{comp of decay x}; we have for $m\in[1,2]$ that
		\begin{equation}
			\begin{split}
				\big\|\nabla_x^\ell\wt u(\cdot,x,y)\big\|_{L^m(\R)}
				\lesssim y^{2s}\int_{\R^n}\big( \abs{x-z}^2+y^2\big)^{-\frac{n+\ell}{2}-s}\norm{u(\cdot,z)}_{L^m(\R)}\,dz,
			\end{split}
		\end{equation}
		for $\ell=0,1$. To proceed, we need to analyze the kernel function $\mathbf K_y(x):=\frac{y^{2s}}{(|x|^2+y^2)^{\frac n2+s}}$. Via \eqref{comp of decay x}, one has
		\begin{equation}\label{no time estimate in I_{2,k}}
			\begin{split}
				\left\|\wt u_f(\cdot,\cdot,y)\right\|_{L^m(\R;L^r(A_k))}^r
				&\lesssim y^{2sr}\int_{A_k}\LC \int_{\Omega\cup W}\frac{\norm{u_f(\cdot,z)}_{L^m(\R)}}{(|x-z|^2+y^2)^{\frac n2+s}}\,dz\RC^r\,dx\\
				&\lesssim \frac{y^{2sr}k^n}{(k^2+y^2)^{(\frac n2+s)r}}\left\|u_f\right\|_{L^m(\R;L^1(\R^n))}^r.
			\end{split}
		\end{equation}
		Similarly, there holds
		\begin{equation}\label{I_{2,k} estimate case Lm}
			\big\|\nabla_{x,y}^\ell\wt u_f(\cdot,\cdot,y)\big\|_{L^m(\R;L^r(A_k))}
			\lesssim \frac{y^{2s-\ell}k^{\frac nr}}{(k^2+y^2)^{\frac n2+s}}\left\|u_f\right\|_{L^m(\R;L^1(\R^n))},
		\end{equation}
		for $\ell=0,1$.
		
		We use an analogous strategy as in Step 2a, i.e., we derive the estimate by considering two cases.
		
		\begin{itemize}
			\item For $\mathbf{w}=\wt u_1$, for $\ell=0,1$, by H\"older's inequality as in the previous computation, we have
			\begin{equation}
				\begin{split}
					\int_{A_k}\int_{\R}\abs{\mathbf{w}(t,x,y)}\left|\nabla_{x,y}^\ell\wt u_f(t,x,y)\right|\,dtdx &\leq \norm{\mathbf{w}}_{L^2(\R^{n+1})}\big\|\nabla_{x,y}^\ell\wt u_f(t,x,y)\big\|_{L^2(\R;L^2(A_k))}\\
					&\lesssim \frac{y^{2s-\ell}k^{\frac n2}}{(k^2+y^2)^{\frac n2+s}}\norm{\mathbf{w}}_{L^2(\R^{n+1})}\left\|u_f\right\|_{L^2(\R;L^1(\R^n))}.
				\end{split}
			\end{equation}
			This infers
			\begin{equation}\label{I_{2,k} estimate case L11}
				\begin{split}
					&\quad\, k^{-1}\int_0^{2k}\int_{A_k}\int_{\R}y^{1-2s}\abs{\mathbf{w}(t,x,y)}\big|\nabla_{x,y}^\ell\wt u_f(t,x,y)\big|\,dtdxdy\\
					&\lesssim \norm{\mathbf{w}}_{L^2(\R^{n+1})}\left\|u_f\right\|_{L^2(\R;L^1(\R^n))}k^{-1+\frac n2}\int_0^{2k}\frac{y^{1-\ell}}{(k^2+y^2)^{\frac n2+s}}\,dy\\
					&\lesssim \norm{\mathbf{w}}_{L^2(\R^{n+1})}\norm{f}_{\widetilde{\mathbf H}^s(W_T)}
					k^{1-\frac n2-\ell-2s}
					\int_0^2\frac{\tau^{1-\ell}}{(1+\tau^2)^{\frac n2+s}}\,d\tau.
				\end{split}
			\end{equation}
			Since $\int_0^2 \frac{\tau^{1-\ell}}{(1+\tau^2)^{\frac n2+s}}\,d\tau<\infty$ for $\ell=0,1$, via \eqref{I_{2,k} estimate case L11}, it follows that $I_{2,k}(\wt w)\to0$ as $k\to\infty$, for $n\ge2$ and $s\in(0,1)$, for both $\ell=0,1$.
			
			\item For $\mathbf{w}=\mathcal{E}_s u_1,u_2$, we denote by $\mathbf{w}$ any of the functions $\mathcal{E}_s u_1,u_2$. For $\ell\in\{0,1\}$, similar to the previous case, the decay estimate \eqref{I_{2,k} estimate case Lm} for the function $\wt u_f$ infers
			\begin{equation}
				\begin{split}
					\int_{A_k}\int_{\R}\abs{\mathbf{w}(t,x,y)}\left|\nabla_{x,y}^\ell\wt u_f(t,x,y)\right|\,dtdx&\leq \norm{\mathbf{w}(\cdot,\cdot,y)}_{L^2(\R^{n+1})}\big\|\nabla_{x,y}^\ell\wt u_f(t,x,y)\big\|_{L^2(\R;L^2(A_k))}\\
					&\lesssim \frac{y^{2s-\ell}k^{\frac n2}}{(k^2+y^2)^{\frac n2+s}}\norm{\mathbf{w}(\cdot,\cdot,y)}_{L^2(\R^{n+1})}\left\|f\right\|_{\widetilde{\mathbf H}^s(W_T)}.
				\end{split}
			\end{equation}
			In addition, H\"older's inequality yields that
			\begin{equation}\label{I_{2,k} estimate case L21}
				\begin{split}
					&\quad \, k^{-1}\int_0^{2k}\int_{A_k}\int_{\R}y^{1-2s}\abs{\mathbf{w}(t,x,y)}\left|\nabla_{x,y}^\ell\wt u_f(t,x,y)\right|\,dtdxdy\\
					&\lesssim \norm{\mathbf{w}}_{\mathcal L^{1,2}(\R^{n+2}_+;y^{1-2s}dtdxdy)}\norm{f}_{\widetilde{\mathbf H}^s(W_T)}k^{-1+\frac n2}\bigg(\int_0^{2k}\frac{y^{1+2s-2\ell}}{(k^2+y^2)^{n+2s}}\,dy\bigg)^{1/2}\\
					&\lesssim \norm{\mathbf{w}}_{\mathcal L^{1,2}(\R^{n+2}_+;y^{1-2s}dtdxdy)}				\norm{f}_{\widetilde{\mathbf H}^s(W_T)}k^{-\frac n2-s-\ell} \left(\int_0^2\frac{\tau^{1+2s-2\ell}}{(1+\tau^2)^{n+2s}}\,d\tau\right)^{1/2}.
				\end{split}
			\end{equation}
			Since $\int_0^2\frac{\tau^{1+2s-2\ell}}{(1+\tau^2)^{n+2s}}\,d\tau<\infty$, via \eqref{I_{2,k} estimate case L21}, we have that $I_{2,k}(\wt w)\to0$ as $k\to\infty$.
		\end{itemize}
		Therefore, $I_{2,k}(\wt w)\to0$ as $k\to\infty$.
		
		In summary, one can conclude that $\lim_{k\to\infty}I_k=0$. With \eqref{pairing for psi(v_f)} at hand, one has
		\begin{equation}
			\big\langle \lim_{y\to0}y^{1-2s}\p_y\wt w,f \big\rangle_{\mathbf H^{-s}(W_T),\widetilde{\mathbf H}^s(W_T)}=0.
		\end{equation}
		By arbitrary choice of $f\in C^\infty_c(W_T)$, we must have
		\[
		\lim_{y\to0}y^{1-2s}\p_y\wt w=0
		\quad\text{on }W_T\times\{0\}.
		\]
		Moreover, since $\wt w=0$ on $W_T\times\{0\}$ as well, we apply Lemma \ref{Lemma: boundary UCP extension} to the time-reversed function
		\[
		U(t,x,y):=\wt w(-t,x,y).
		\]
		Equivalently, the same boundary unique continuation statement holds for the backward parabolic extension equation. Hence, with $I=(-T,T)$, $\mathcal O=\Omega_e$ and $\mathcal W=W$, we obtain
		\begin{equation}
			\wt w\equiv0\quad\text{in }(\Omega_e)_T\times(0,\infty).
		\end{equation}
		In particular, one has
		\[
		\left.\wt w\right|_{(\p\Omega)_T\times(0,\infty)}=0,
		\quad
		\left.\wt\sigma\nabla_{x,y}\wt w\right|_{(\p\Omega)_T\times(0,\infty)}=0
		\]
		in the sense of distributions. Moreover,
		\[
		\lim_{y\to0}y^{1-2s}\p_y\wt w=0
		\quad\text{on }\R^n_T\times\{0\}
		\]
		in the trace sense.
		
		Let us review auxiliary functions constructed in \cite[Proposition 3.1]{CGRU2023reduction}. Let $\mu:[0,1]\to[0,1]$ be a smooth function on $[0,1]$ with $\mu(0)=0$, $\mu(1)=1$. Moreover, one can assume that
		\begin{equation}
			\int_0^1\mu(y)\,dy=\frac12\quad\text{and}\quad |\p_y^\ell\mu(y)|\leq C,\quad y\in(0,1),
		\end{equation}
		for $\ell=0,1,2$, where $C>1$ is a constant. Given $b\in(0,1)$, let $\gamma_b:(-\infty,\infty)\to(0,b)$ be a smooth function defined by
		\begin{equation}
			\gamma_b(y)=
			\begin{cases}
				0, & \text{ if } y<0,\\
				b\mu(y), & \text{ if } y\in[0,1],\\
				b, & \text{ if } y\in[1,\frac{1}{1-b}],\\
				b\mu(\frac{2-b}{1-b}-y), & \text{ if } y\in[\frac{1}{1-b},\frac{2-b}{1-b}],\\
				0, & \text{ if } y>\frac{2-b}{1-b}.
			\end{cases}
		\end{equation}
		From the construction, it is easy to see that
		\begin{equation}
			\int_0^\infty\gamma_b(y)\,dy=\frac{b}{1-b}\quad\text{and}\quad |\p_y^\ell\gamma_b|\leq Cb,
		\end{equation}
		for $\ell=0,1,2$, where $C>1$ is independent of $b\in(0,1)$. Consider
		\begin{equation}
			J_{b,k}:=\int_0^\infty(y+k)^{1-2s}\gamma_b(y)\,dy=\int_0^{\frac{2-b}{1-b}}(y+k)^{1-2s}\gamma_b(y)\,dy,
		\end{equation}
		where $J_{b,k}$ depends continuously on the parameter $b\in(0,1)$.
		
		We observe for $0\leq y\leq \frac{2-b}{1-b}$ that
		\begin{equation}
			\begin{cases}
				k^{1-2s}\leq (y+k)^{1-2s}\leq k^{1-2s}\LC 1+\frac{2-b}{k(1-b)}\RC^{1-2s}, & \text{ if }s\in(0,\frac12],\\
				k^{1-2s}\LC 1+\frac{2-b}{k(1-b)}\RC^{1-2s}\leq (y+k)^{1-2s}\leq k^{1-2s}, & \text{ if }s\in(\frac12,1).
			\end{cases}
		\end{equation}
		Then, we have the estimate for $J_{b,k}$:
		\begin{equation}
			\begin{cases}
				\frac{b}{1-b}k^{1-2s}\leq J_{b,k}\leq \frac{b}{1-b}k^{1-2s}\LC 1+\frac{2-b}{k(1-b)}\RC^{1-2s}, & \text{ if }s\in(0,\frac12],\\
				\frac{b}{1-b}k^{1-2s}\LC 1+\frac{2-b}{k(1-b)}\RC^{1-2s}\leq J_{b,k}\leq \frac{b}{1-b}k^{1-2s}, & \text{ if }s\in(\frac12,1).
			\end{cases}
		\end{equation}
		it follows that for $b\in(0,1)$, the value $J_{b,k}$ can be both arbitrarily large and arbitrarily close to $0$. Hence, by continuity, for any $k\in\N$, we can find $b_{k,s}\in(0,1)$ such that $J_{b_{k,s},k}=1$. Consider $\beta_k(y):=\gamma_{b_{k,s}}(y-k)$, and $R_{k,s}:=k+\frac{1}{1-b_{k,s}}$. For $0<s<\frac12$, we observe that $\big( 1+\frac{2-b}{k(1-b)}\big)^{1-2s}>1$ and needs $b k^{1-2s}\leq \frac{b}{1-b}k^{1-2s}\leq1$. Thus, if $0<s<\frac12$, then there exists a $k_s$ such that for $k\geq k_s$,
		\begin{equation}\label{b_k,s}
			b_{k,s}\leq k^{2s-1},\quad R_{k,s}=k+\frac{1}{1-b_{k,s}}\leq k+2.
		\end{equation}
		Combined with the previous constructions, for $s\in(0,1)$, let us consider the function $\beta_k:(0,\infty)\to[0,1]$ in the following form:
		\begin{itemize}
			\item For $s\neq1/2$,
			\begin{equation}\label{beta_k}
				\begin{cases}
					\supp\LC \beta_k\RC\subseteq \LC k,R_{k,s}+1\RC,\\
					\beta_k(y)=b_{k,s}\text{ for }y\in(k+1,R_{k,s}),\\
					|\p_y^\ell\beta_k(y)|\leq Cb_{k,s},\\
					\int_0^\infty y^{1-2s}\beta_k(y)\,dy=1,
				\end{cases}
			\end{equation}
			for any $\ell=0,1,2$ and $k\in\N$, where the constant $C>1$ is independent of $k\in\N$.
			
			\item For $s=1/2$, let us consider $\beta\in C^\infty_c(0,\infty)$ such that
			\begin{equation}\label{beta_k s=0.5}
				\beta\geq0,\quad \int_0^\infty\beta(y)\,dy=1,\quad \supp(\beta)\subset(1,2),\quad \text{and}\quad \beta_k(y):=k^{-1}\beta(y/k)
			\end{equation}
			for $k\in\N$.
		\end{itemize}
		We use this function $\beta_k$ to prove the density result by the classical Hahn-Banach argument, and we will handle the arguments in different cases as $s\neq1/2$ and $s=1/2$.
		
	Let $v\in\mathcal D_\sigma$. Since $v(-T,\cdot)=0$ in $\Omega$, we choose an extension $Ev$ such that
	\begin{equation}\label{reg of Ev}
		Ev\in L^2(\R;H^1(\R^n))\quad\text{and}\quad \p_tEv\in L^2(\R;H^{-1}(\R^n)),
	\end{equation}
	with $Ev=v$ in $\Omega_T$, $Ev=0$ for $t\leq -T$, and
	\[
	\supp(Ev)\subset (-\infty,\wt T)\times\Omega',
	\]
	where $\Omega'\supset\Omega$ is a bounded open set in $\R^n$ and $\wt T>T$. With the condition $\psi\in L^2(-T,T;\wt H^{-1}(\Omega))$ at hand, we have
		\begin{equation}\label{psi(v)}
			\begin{split}
				\psi(v)
				&=\bigg\langle \psi,v\int_0^\infty y^{1-2s}\beta_k\,dy\bigg\rangle_{L^2(-T,T;H^{-1}(\R^n)),L^2(-T,T;H^1(\R^n))}\\
				&=\bigg\langle \psi,\int_0^\infty y^{1-2s}\beta_k Ev\,dy\bigg\rangle_{L^2(\R;H^{-1}(\R^n)),L^2(\R;H^1(\R^n))},
			\end{split}
		\end{equation}
		where $\beta_k$ is given by \eqref{beta_k} or \eqref{beta_k s=0.5}. On the other hand, by \eqref{reg of Ev} and Lemma \ref{Lemma: reg of local para imply nonlocal para}, the function $\beta_kEv$ has the admissible regularity for the time-reversed weak formulation associated with \eqref{adjoint problem generalized_rmk}. Moreover, it is compactly supported in the $(t,x,y)$ variables and vanishes for $t\leq -T$. Similar to the computations \eqref{pairing for psi(v_f)}, via \eqref{psi(v)}, we can deduce
		\begin{equation}\label{psi(v) comp}
			\begin{split}
				\psi(v)
				&= \int_{\R^{n+2}_+} \left[ y^{1-2s}\beta_kEv\,\p_t\wt w+\nabla_{x,y}\cdot\LC y^{1-2s}\wt\sigma\nabla_{x,y}\wt w\RC\beta_kEv\right]\,dtdxdy\\
				&= \int_{\R^{n+2}_+}\left[y^{1-2s}\beta_kEv\,\p_t\wt w-y^{1-2s}\wt\sigma\nabla_{x,y}\wt w\cdot\nabla_{x,y}\LC\beta_kEv\RC\right]\,dtdxdy\\
				&= \int_{\R^{n+2}_+}y^{1-2s}\beta_kEv\,\p_t\wt w\,dtdxdy-\int_{\R^{n+2}_+}y^{1-2s}Ev\,\p_y\beta_k\,\p_y\wt w\,dtdxdy\\
				&\quad -\int_{\R^{n+1}}\sigma\nabla\LC Ev\RC\cdot\nabla\LC\int_0^\infty y^{1-2s}\beta_k\wt w\,dy\RC\,dtdx\\
				&=\int_{\Omega_T}v\,\p_t\LC\int_0^\infty y^{1-2s}\beta_k\wt w\,dy\RC\,dtdx-\int_{\Omega_T}\sigma\nabla v\cdot\nabla\LC\int_0^\infty y^{1-2s}\beta_k\wt w\,dy\RC\,dtdx\\
				&\quad -\int_0^\infty\int_{\Omega_T}y^{1-2s}v\,\p_y\beta_k\,\p_y\wt w\,dtdxdy.
			\end{split}
		\end{equation}
		We next claim that the first two terms in \eqref{psi(v) comp} vanish by making use of the equation satisfied by $v\in\mathcal D_\sigma$.
		
		We point out that the function $\int_0^\infty y^{1-2s}\beta_k(y)\wt w(t,x,y)\,dy$ is an admissible test function for the local equation satisfied by $v$. Indeed,
		\begin{equation}
			\begin{split}
				&\quad\, \bigg\|\int_0^\infty y^{1-2s}\beta_k(y)\wt w(t,x,y)\,dy		\bigg\|_{L^2(-T,T;H^1(\Omega))}^2\\
				&\leq \int_{\Omega_T}\bigg(\int_k^{R_{k,s}+1}y^{1-2s}\beta_k(y)|\wt w|\,dy
				\bigg)^2\,dtdx+\int_{\Omega_T}\bigg(			\int_k^{R_{k,s}+1}y^{1-2s}\beta_k(y)|\nabla\wt w|\,dy
				\bigg)^2\,dtdx\\
				&\leq \int_{\Omega_T}\bigg(	\int_k^{R_{k,s}+1}y^{1-2s}\beta_k(y)^2\,dy\bigg)
				\bigg(\int_k^{R_{k,s}+1}y^{1-2s}\big(|\wt w|^2+|\nabla\wt w|^2\big)\,dy\bigg)\,dtdx\\
				&\lesssim
				\|\wt w\|_{\mathcal L^{1,2}(\Omega_T\times(0,R_{k,s}+1),y^{1-2s}dtdxdy)}^2<\infty.
			\end{split}
		\end{equation}
		By the unique continuation result, $\wt w=0$ on $(\partial\Omega)_T\times(0,\infty)$. Hence, $\int_0^\infty y^{1-2s}\beta_k(y)\wt w(t,x,y)\,dy=0\quad\text{on }(\partial\Omega)_T$, and then $\int_0^\infty y^{1-2s}\beta_k(y)\wt w(t,x,y)\,dy \in L^2(-T,T;H^1_0(\Omega))$.
		Since $\int_0^\infty y^{1-2s}\beta_k(y)\wt w(t,x,y)\,dy \in L^2(-T,T;H^1_0(\Omega))$,
		we may use it as a test function in the local energy formulation of
		$H_\sigma v=0$. More precisely, using $v(-T,\cdot)=0$ and the terminal condition
		$\wt w(T,\cdot,\cdot)=0$, integration by parts in time gives
		\begin{equation}\label{psi(v) comp 1}
			\begin{split}
				&\quad \, \int_{\Omega_T}	v\,\partial_t
				\bigg(\int_0^\infty y^{1-2s}\beta_k(y)\wt w(t,x,y)\,dy\bigg)\,dtdx
				-\int_{\Omega_T}\sigma\nabla v\cdot\nabla
				\bigg(\int_0^\infty y^{1-2s}\beta_k(y)\wt w(t,x,y)\,dy\bigg)\,dtdx\\
				&=-\left\langle \partial_t v,
				\int_0^\infty y^{1-2s}\beta_k(y)\wt w(t,x,y)\,dy
				\right\rangle-\int_{\Omega_T}\sigma\nabla v\cdot\nabla
				\bigg(\int_0^\infty y^{1-2s}\beta_k(y)\wt w(t,x,y)\,dy\bigg)\,dtdx\\
				&=0.
			\end{split}
		\end{equation}
		Inserting \eqref{psi(v) comp 1} into \eqref{psi(v) comp}, we get
		\begin{equation}
			\psi(v)=-\int_0^\infty\int_{\Omega_T}y^{1-2s}v\,\p_y\beta_k\,\p_y\wt w\,dtdxdy.
		\end{equation}
		The desired result $\psi(v)=0$ can be achieved by passing to the limit $k\to\infty$. To this end, note that $\p_y\wt w=-\p_y\mathcal{E}_s u_1+\p_yu_2$, since $\wt u_1$ is $y$-independent. Therefore, if $\wt w$ is any of the functions $\mathcal{E}_s u_1,u_2$, one can estimate as the estimate \eqref{Step 2a 1} in Step 2a. By H\"older's inequality, one obtains
		\begin{equation}\label{psi(v) comp 2}
			\begin{split}
				\int_0^\infty y^{1-2s}\abs{\p_y\beta_k}\int_{\Omega_T}\abs{v\,\p_y\wt w}\,dtdxdy		\leq \norm{v}_{L^2(\Omega_T)}\int_0^\infty y^{1-2s}\abs{\p_y\beta_k}\left\|\p_y\wt w(\cdot,\cdot,y)\right\|_{L^2(\Omega_T)}\,dy.
			\end{split}
		\end{equation}
		
		To proceed, let us split the case for $s\neq 1/2$ and $s=1/2$.
		
		\begin{itemize}
			\item For $s\neq1/2$: Since $\beta_k$ is given by \eqref{beta_k}, using the numbers $b_{k,s}$, $R_{k,s}$ satisfying \eqref{beta_k}, we can see
			\begin{equation}\label{psi(v) comp 3}
				\begin{split}
					&\quad \, \int_0^\infty y^{1-2s}\abs{\p_y\beta_k}\left\|\p_y\wt w(\cdot,\cdot,y)\right\|_{L^2(\Omega_T)}\,dy\\
					&\lesssim \norm{\wt w}_{\mathcal L^{1,2}(\Omega_T\times(0,\infty),y^{1-2s}dtdxdy)}
					\bigg(\int_0^\infty y^{1-2s}\abs{\p_y\beta_k}^2\,dy\bigg)^{1/2}\\
					&\lesssim |b_{k,s}|\bigg( \int_k^{k+1}y^{1-2s}\,dy+\int_{R_{k,s}}^{R_{k,s}+1}y^{1-2s}\,dy\bigg)^{1/2}
					\norm{\wt w}_{\mathcal L^{1,2}(\Omega_T\times(0,\infty),y^{1-2s}dtdxdy)}\\
					&\lesssim |b_{k,s}|\big( k^{1-2s}+R_{k,s}^{1-2s}\big)^{1/2}
					\norm{\wt w}_{\mathcal L^{1,2}(\Omega_T\times(0,\infty),y^{1-2s}dtdxdy)}.
				\end{split}
			\end{equation}
			Taking into account \eqref{b_k,s} for $0<s<\frac12$, let us estimate $|b_{k,s}|\big( k^{1-2s}+R_{k,s}^{1-2s}\big)^{1/2}$ in the following:
			\begin{equation}\label{psi(v) comp 4}
				|b_{k,s}|\big( k^{1-2s}+R_{k,s}^{1-2s}\big)^{1/2}\lesssim
				\begin{cases}
					k^{\frac{2s-1}{2}}, &\text{ if }s\in(0,\frac12),\\
					k^{\frac{1-2s}{2}}, &\text{ if }s\in(\frac12,1).
				\end{cases}
			\end{equation}
			Let $k\to\infty$. Then \eqref{psi(v) comp 1}, \eqref{psi(v) comp 2}, \eqref{psi(v) comp 3} and \eqref{psi(v) comp 4} infer that $\psi(v)=0$ holds true for $s\neq1/2$.
			
			\item For $s=1/2$: Since $\beta_k$ is given by \eqref{beta_k s=0.5}, by using a similar argument, one has
			\begin{equation}\label{psi(v) comp s=1/2}
				\begin{split}
					\int_0^\infty \abs{\p_y\beta_k}\left\|\p_y\wt w(\cdot,\cdot,y)\right\|_{L^2(\Omega_T)}\,dy
					&\lesssim \left\|\p_y\wt w\right\|_{L^2(\Omega_T\times(0,\infty))}
					\bigg(\int_k^{2k}\abs{\p_y\beta_k}^2\,dy\bigg)^{1/2}\\
					&\lesssim k^{-3/2}\left\|\p_y\wt w\right\|_{L^2(\Omega_T\times(0,\infty))}\to0, \quad  \text{as }k\to \infty.
				\end{split}
			\end{equation}
		\end{itemize}
		In summary, for both cases $s\neq1/2$ and $s=1/2$, we conclude that for any $\psi\in L^2(-T,T;\wt H^{-1}(\Omega))$ with $\psi(v_f)=0$ for all $f\in C^\infty_c(W_T)$, one also obtains $\psi(v)=0$ for any $v\in\mathcal D_\sigma$. By the Hahn--Banach theorem, this proves \eqref{density interior proof target}.
		
		Finally, we pass from the interior density to boundary traces. By the classical trace estimate,
		\begin{equation}
			\left\|U|_{(\partial\Omega)_T}\right\|_{L^2(-T,T;H^{1/2}(\partial\Omega))}
			\lesssim
			\left\|U\right\|_{L^2(-T,T;H^1(\Omega))},
		\end{equation}
		for all $U\in L^2(-T,T;H^1(\Omega))$. Since $\mathcal V\subset\mathcal D_\sigma$, we first have
		\begin{equation}
			\overline{\mathcal V'}^{\,L^2(-T,T;H^{1/2}(\partial\Omega))}
			\subset
			\mathcal T_\sigma.
		\end{equation}
		Conversely, let $g\in\mathcal T_\sigma$. By the definition of $\mathcal T_\sigma$ in \eqref{local energy trace space}, there exists a sequence $V_j\in\mathcal D_\sigma$ such that
		\begin{equation}
			V_j|_{(\partial\Omega)_T}\to g
			\quad\text{in }L^2(-T,T;H^{1/2}(\partial\Omega)).
		\end{equation}
		For each fixed $j$, the interior density \eqref{density interior proof target} gives a sequence $v_{f_{j,k}}\in\mathcal V$ such that $v_{f_{j,k}}\to V_j$
		in $L^2(-T,T;H^1(\Omega))$ as $k\to\infty$.
		Hence, by the trace estimate,
		\begin{equation}
			v_{f_{j,k}}|_{(\partial\Omega)_T}\to V_j|_{(\partial\Omega)_T}
			\quad\text{in }L^2(-T,T;H^{1/2}(\partial\Omega)).
		\end{equation}
		A diagonal argument yields a sequence $v_{f_m}\in\mathcal V$ such that
		\begin{equation}
			v_{f_m}|_{(\partial\Omega)_T}\to g
			\quad\text{in }L^2(-T,T;H^{1/2}(\partial\Omega)).
		\end{equation}
		Therefore, $\overline{\mathcal V'}^{\,L^2(-T,T;H^{1/2}(\partial\Omega))}=\mathcal T_\sigma$. If, in addition, the local Dirichlet problem for $H_\sigma$ is well posed for every boundary value
		$g\in L^2(-T,T;H^{1/2}(\partial\Omega))$, then $\mathcal T_\sigma=L^2(-T,T;H^{1/2}(\partial\Omega))$, and hence, $\overline{\mathcal V'}^{\,L^2(-T,T;H^{1/2}(\partial\Omega))}=	L^2(-T,T;H^{1/2}(\partial\Omega))$, which proves the assertion.
	\end{proof}
	
	\begin{corollary}[Density of averaged Cauchy data]\label{Corollary: density Cauchy data}
		Under the assumptions of Proposition \ref{Prop: density}, define
		\begin{equation}
			\mathscr A_\sigma:=
			\left\{
			\left(
			v_f|_{(\partial\Omega)_T},
			\sigma\nabla v_f\cdot\nu|_{(\partial\Omega)_T}
			\right):
			f\in C_c^\infty(W_T)
			\right\}.
		\end{equation}
		Then
		\begin{equation}
			\overline{\mathscr A_\sigma}^{\,L^2(-T,T;H^{1/2}(\partial\Omega))\times
				L^2(-T,T;H^{-1/2}(\partial\Omega))}
			=
			\mathcal C_{\sigma,(\partial\Omega)_T}.
		\end{equation}
	\end{corollary}
	
	\begin{proof}
		Let $V\in\mathcal D_\sigma$. By Proposition \ref{Prop: density}, there exists a sequence
		$v_{f_k}\in\mathcal V$ such that
		\begin{equation}
			v_{f_k}\to V
			\quad\text{in }L^2(-T,T;H^1(\Omega)).
		\end{equation}
		By the trace theorem,
		\begin{equation}
			v_{f_k}|_{(\partial\Omega)_T}\to V|_{(\partial\Omega)_T}
			\quad\text{in }L^2(-T,T;H^{1/2}(\partial\Omega)).
		\end{equation}
		Set $U_k:=v_{f_k}-V$. Since both $v_{f_k}$ and $V$ solve $H_\sigma U=0$ in $\Omega_T$, we have
		\begin{equation}
			\partial_t U_k=\nabla\cdot\sigma\nabla U_k
			\quad\text{in }L^2(-T,T;H^{-1}(\Omega)).
		\end{equation}
		Hence,
		\begin{equation}
			\|\partial_t U_k\|_{L^2(-T,T;H^{-1}(\Omega))}
			\lesssim
			\|U_k\|_{L^2(-T,T;H^1(\Omega))}\to0.
		\end{equation}
		The weak conormal trace estimate therefore gives
		\begin{equation}
			\sigma\nabla v_{f_k}\cdot\nu|_{(\partial\Omega)_T}
			\to
			\sigma\nabla V\cdot\nu|_{(\partial\Omega)_T}
			\quad\text{in }L^2(-T,T;H^{-1/2}(\partial\Omega)).
		\end{equation}
		Thus, every element of $\mathcal C_{\sigma,(\partial\Omega)_T}$ is the limit of elements of $\mathscr A_\sigma$.
		The reverse inclusion follows from $v_f\in\mathcal D_\sigma$ for every $f\in C_c^\infty(W_T)$.
		This proves the claim.
	\end{proof}
	
	\begin{remark}
		The density argument above applies to both scalar and matrix-valued coefficients in the unweighted parabolic class $H_\sigma=\p_t-\nabla\cdot\sigma\nabla$, under the assumptions stated in Proposition \ref{Prop: density}.
	\end{remark}

	\section{Proofs of main results}\label{sec: proof of thm}
	
	We now prove Theorem \ref{Thm: main}.

	\begin{lemma}[Exterior propagation under exterior agreement]\label{Lemma: exterior comparison}
		Let the assumptions of Theorem \ref{Thm: main} hold. Let $\sigma_1,\sigma_2$ satisfy \eqref{exterior agreement assumption}. For $j=1,2$, let $u_{j,f}$ be the solution of the nonlocal problem associated with $\sigma_j$ and exterior value $f\in C_c^\infty(W_T)$, and let $\wt u_{j,f}$ be its parabolic extension. Assume that \eqref{equal nonlocal DN maps} holds, then $\wt u_{1,f}=\wt u_{2,f}$ in $(\Omega_e)_T\times(0,\infty)$. Consequently, $\left.\wt u_{1,f}\right|_{(\partial\Omega)_T\times(0,\infty)}=\left.\wt u_{2,f}\right|_{(\partial\Omega)_T\times(0,\infty)}$ and $\left.y^{1-2s}\sigma_1\nabla_x\wt u_{1,f}\cdot\nu\right|_{(\partial\Omega)_T\times(0,\infty)}=\left.y^{1-2s}\sigma_2\nabla_x\wt u_{2,f}\cdot\nu\right|_{(\partial\Omega)_T\times(0,\infty)}$ in the weak trace sense. Therefore, $\left.v_{1,f}\right|_{(\partial\Omega)_T}=	\left.v_{2,f}\right|_{(\partial\Omega)_T}$ and $\left.\sigma_1\nabla v_{1,f}\cdot\nu\right|_{(\partial\Omega)_T}=\left.\sigma_2\nabla v_{2,f}\cdot\nu\right|_{(\partial\Omega)_T}$.
	\end{lemma}
	
	\begin{proof}
		Set
		\[
		\wt U:=\wt u_{1,f}-\wt u_{2,f}.
		\]
		Since the exterior Dirichlet data are the same, one has $\wt U(t,x,0)=0$ on $W_T$. Moreover, by the equality of the partial nonlocal DN maps and by the extension DN relation,
		\[
		d_s\lim_{y\to0}y^{1-2s}\p_y\wt U
		=
		H_{\sigma_1}^s u_{1,f}-H_{\sigma_2}^s u_{2,f}
		=0
		\quad\text{on }W_T.
		\]
		By the exterior agreement assumption \eqref{exterior agreement assumption}, the two extension operators have the same coefficients in $\Omega_e\times(0,\infty)$. Hence $\wt U$ solves the homogeneous extension equation with this common coefficient in $(\Omega_e)_T\times(0,\infty)$. Since both its Dirichlet trace and weighted Neumann trace vanish on $W_T\times\{0\}$, Lemma \ref{Lemma: boundary UCP extension}, applied with $I=(-T,T)$, $\mathcal O=\Omega_e$ and $\mathcal W=W$, gives
		\[
		\wt U=0\quad\text{in }(\Omega_e)_T\times(0,\infty),
		\]
		where we use that $\Omega_e$ is connected.
		
		Taking weak traces on $(\partial\Omega)_T\times(0,\infty)$ gives the equality of the extension Dirichlet traces. For each fixed $y>0$, the hypersurface $\partial\Omega\times\{y\}$ is an interior hypersurface for the extension equation in the $x$-variables. Since each $\wt u_{j,f}$ is a global weak solution in the spatial variables, its conormal trace across $\partial\Omega\times\{y\}$ agrees from the two sides in the weak sense. Hence, the equality in the exterior, together with $\sigma_1=\sigma_2$ on $\partial\Omega$ by continuity, gives the equality of the conormal traces.
		
		The integration in $y$ is justified by Proposition \ref{Prop: bound w} and by the definition of the weak conormal trace for the averaged functions. Therefore, $\left.v_{1,f}\right|_{(\partial\Omega)_T}=\left.v_{2,f}\right|_{(\partial\Omega)_T}$ and $\left.\sigma_1\nabla v_{1,f}\cdot\nu\right|_{(\partial\Omega)_T}=\left.\sigma_2\nabla v_{2,f}\cdot\nu\right|_{(\partial\Omega)_T}$.
	\end{proof}

	\begin{proof}[Proof of Theorem \ref{Thm: main} and Proposition \ref{Proposition main theorem}]
		By Theorem \ref{Thm: key equation}, for each $f\in \widetilde{\mathbf H}^s(W_T)$, the averaged function $v_f(t,x):=\int_0^\infty y^{1-2s}\wt u_f(t,x,y)\,dy\in L^2(-T,T;H^1(\Omega))$ with $\p_t v_f\in L^2(-T,T;H^{-1}(\Omega))$, and solves
		\begin{equation}\label{eq:v_f_equ}
			\LC\p_t-\nabla\cdot\sigma\nabla\RC v_f=0\quad\text{in }\Omega_T.
		\end{equation}
		By Proposition \ref{Prop: density}, the operator
		\begin{equation}
			\begin{split}
				\mathrm{T}_\sigma:  \widetilde{\mathbf{H}}^s (W_T) &\to L^2(-T,T;H^{1/2}(\p \Omega)),\\
				f& \mapsto \left. v_f(t,x) \right|_{(\p \Omega)_T}:=\left.\int_0^\infty y^{1-2s} \wt u_f(t,x,y) \, dy \right|_{(\p \Omega)_T}
			\end{split}
		\end{equation}
		is linear, bounded, and has dense range in the local trace space $\mathcal T_\sigma$. Moreover, since $v_f$ solves \eqref{eq:v_f_equ} with $v_f\in L^2(-T,T;H^1(\Omega))$ and $\p_t v_f\in L^2(-T,T;H^{-1}(\Omega))$, the weak conormal trace satisfies
		\[
		\left\|\sigma\nabla v_f\cdot\nu\right\|_{L^2(-T,T;H^{-1/2}(\partial\Omega))}
		\lesssim
		\|v_f\|_{L^2(-T,T;H^1(\Omega))}
		+
		\|\partial_t v_f\|_{L^2(-T,T;H^{-1}(\Omega))}.
		\]
		Thus, the map
		\[
		f\mapsto
		\left(
		v_f|_{(\partial\Omega)_T},
		\sigma\nabla v_f\cdot\nu|_{(\partial\Omega)_T}
		\right)
		\]
		is bounded in the local Cauchy data space.
		
		We now prove the comparison assertion. Let $\sigma_1,\sigma_2$ satisfy \eqref{exterior agreement assumption}, and assume
		\[
		\Lambda_{\sigma_1,W_T}^s=\Lambda_{\sigma_2,W_T}^s.
		\]
		For $j=1,2$ and $f\in C_c^\infty(W_T)$, let $u_{j,f}$ be the corresponding nonlocal solution, let $\wt u_{j,f}$ be its extension, and set
		\[
		v_{j,f}(t,x)=\int_0^\infty y^{1-2s}\wt u_{j,f}(t,x,y)\,dy.
		\]
		By Lemma \ref{Lemma: exterior comparison}, for every $f\in C_c^\infty(W_T)$,
		\[
		\left.v_{1,f}\right|_{(\partial\Omega)_T}=\left.v_{2,f}\right|_{(\partial\Omega)_T},
		\quad
		\left.\sigma_1\nabla v_{1,f}\cdot\nu\right|_{(\partial\Omega)_T}
		=
		\left.\sigma_2\nabla v_{2,f}\cdot\nu\right|_{(\partial\Omega)_T}.
		\]
		Define
		\[
		\mathscr A_j:=\big\{\big(v_{j,f}\big|_{(\partial\Omega)_T},\, \sigma_j\nabla v_{j,f}\cdot\nu\big|_{(\partial\Omega)_T}\big): \, f\in C_c^\infty(W_T)\big\}, \quad j=1,2.
		\]
		Then Lemma \ref{Lemma: exterior comparison} gives $\mathscr A_1=\mathscr A_2$. On the other hand, Corollary \ref{Corollary: density Cauchy data} gives
		\begin{equation}
		\overline{\mathscr A_j}^{\,L^2(-T,T;H^{1/2}(\partial\Omega))\times
		L^2(-T,T;H^{-1/2}(\partial\Omega))}=\mathcal C_{\sigma_j,(\partial\Omega)_T},
		\quad j=1,2.
		\end{equation}
		Taking closures in the identity $\mathscr A_1=\mathscr A_2$, we obtain $\mathcal C_{\sigma_1,(\partial\Omega)_T}=\mathcal C_{\sigma_2,(\partial\Omega)_T}$. In any boundary trace setting where the local DN maps are well defined and continuous and the corresponding Cauchy data sets are the closed graphs of these maps, this is equivalent to $\Lambda_{\sigma_1}=\Lambda_{\sigma_2}$ on $(\partial\Omega)_T$. This proves Theorem \ref{Thm: main}. The fixed-coefficient statement in Proposition \ref{Proposition main theorem} follows from Proposition \ref{Prop: density}, Corollary \ref{Corollary: density Cauchy data}, the trace theorem, and the continuity of the local Cauchy data map.
	\end{proof}

\begin{proof}[Proof of Corollary \ref{Corollary: uniqueness}]
	By Theorem \ref{Thm: main}, the equality $\Lambda_{\sigma_1,W_T}^s=\Lambda_{\sigma_2,W_T}^s$ together with the exterior agreement $\gamma_1=\gamma_2$ in $\Omega_e$ implies equality of the corresponding full local lateral Cauchy data sets. In any trace setting where the scalar local DN maps are well defined, this is equivalent to $\Lambda_{\sigma_1}=\Lambda_{\sigma_2}$. The scalar uniqueness theorem for the local parabolic Calder\'on problem then gives $\gamma_1=\gamma_2$ in $\Omega$.
\end{proof}

	\begin{remark}
		Remark \ref{Remark: anisotropic gauge} explains why we do not state the anisotropic non-uniqueness result as a corollary in the unweighted class. The local change of variables naturally leads to the weighted operator $\kappa\partial_t-\nabla\cdot\sigma\nabla$, where $\kappa=\Phi_\ast1$. A corresponding nonlocal non-uniqueness statement can be formulated after extending the whole transfer argument to the fractional powers of weighted parabolic operators. This lies outside the unweighted formulation used in Theorem \ref{Thm: main}.
	\end{remark}

	\section*{Statements and Declarations}
	
	\para{Funding}
	Y.-H. Lin is partially supported by the National Science and Technology Council (NSTC), Taiwan, under the project 113-2115-M-A49-017-MY3. Y.-H. Lin also acknowledges financial support from the Alexander von Humboldt Foundation through the Henriette Herz Scouting Programme, hosted by Universit\"at Duisburg-Essen, Germany. G. Uhlmann was partially supported by the NSF, a Walker Professorship at the University of Washington.
	
	\para{Data availability}
	No datasets were generated or analyzed during this study.
	
	\para{Competing interests}
	The authors have no competing interests to declare that are relevant to the content of this article.

	\bibliography{refs} 
	
	\bibliographystyle{alpha}
	
\end{document}